\documentclass[12pt]{amsart}

\usepackage[usenames,dvipsnames]{color}
\usepackage{amsthm,amsfonts,amssymb,amsmath,amsxtra}
\usepackage[all, color]{xy}
\SelectTips{cm}{}
\usepackage{xr-hyper}
\usepackage[colorlinks=
   citecolor=Black,
   linkcolor=Red,
   urlcolor=Blue]{hyperref}
\usepackage{verbatim}
\usepackage{stmaryrd}
\usepackage[margin=1.25in]{geometry}
\usepackage{mathrsfs}
\usepackage{multirow}
\usepackage{mathtools}
\usepackage{tikz}
\usetikzlibrary{shapes.geometric,arrows,fit,matrix,positioning}
\tikzset
{
    treenode/.style = {draw=black, align=center, minimum size=1cm},
}
\usetikzlibrary{chains}
\usepackage{float}
\DeclareFontFamily{U}{musix}{}%
\DeclareFontShape{U}{musix}{m}{n}{%
  <-12>   musix11
  <12-15> musix13
  <15-18> musix16
  <18-23> musix20
  <23->   musix29
}{}%
\newcommand*\musix{\usefont{U}{musix}{m}{n}\selectfont}
\DeclareTextFontCommand{\textmusix}{\musix}

\newcommand*\doublesharp{\raisebox{.6ex}{\textmusix{5}}}

\tikzset{node distance=2em, ch/.style={circle,draw,on chain,inner sep=2pt},chj/.style={ch,join},every path/.style={shorten >=4pt,shorten <=4pt},line width=1pt,baseline=-1ex}

\newcommand{\alabel}[1]{%
}

\let\dlabel=\alabel

\newcommand{\dnode}[2][chj]{%
\node[#1,label={below:\dlabel{#2}}] {};
}

\newcommand{\dnodenj}[1]{%
\dnode[ch]{#1}
}

\newcommand{\dnodebr}[1]{%
\node[chj,label={below right:\dlabel{#1}}] {};
}

\newcommand{\dydots}{%
\node[chj,draw=none,inner sep=1pt] {\dots};
}

\RequirePackage{xspace}
\RequirePackage{etoolbox}
\RequirePackage{varwidth}
\RequirePackage{enumitem}
\RequirePackage{tensor}
\RequirePackage{mathtools}
\RequirePackage{longtable}
\RequirePackage{multirow}

\def\ge{\geqslant}
\def\le{\leqslant}
\def\a{\alpha}

\def\g{\gamma}
\def\G{\Gamma}
\def\d{\delta}
\def\D{\Delta}
\def\L{\Lambda}
\def\e{\epsilon}

\def\o{\omega}

\def\s{\sigma}
\def\t{\tau}
\def\th{\theta}
\def\k{\kappa}
\def\l{\lambda}

\def\i{^{-1}}

\def\<{\langle}
\def\>{\rangle}

\newcommand{\sR}{\ensuremath{\mathscr{R}}\xspace}

\newcommand{\fka}{\ensuremath{\mathfrak{a}}\xspace}

\newcommand{\bG}{\mathbf G}
\newcommand{\bK}{\mathbf K}

\newcommand{\bQ}{\mathbf Q}

\newcommand{\BA}{\ensuremath{\mathbb {A}}\xspace}

\newcommand{\BC}{\ensuremath{\mathbb {C}}\xspace}

\newcommand{\BF}{\ensuremath{\mathbb {F}}\xspace}
\newcommand{{\BG}}{\ensuremath{\mathbb {G}}\xspace}
\newcommand{\BH}{\ensuremath{\mathbb {H}}\xspace}

\newcommand{{\BK}}{\ensuremath{\mathbb {K}}\xspace}

\newcommand{\BN}{\ensuremath{\mathbb {N}}\xspace}

\newcommand{\BP}{\ensuremath{\mathbb {P}}\xspace}
\newcommand{\BQ}{\ensuremath{\mathbb {Q}}\xspace}
\newcommand{\BR}{\ensuremath{\mathbb {R}}\xspace}
\newcommand{\BS}{\ensuremath{\mathbb {S}}\xspace}

\newcommand{\BZ}{\ensuremath{\mathbb {Z}}\xspace}

\newcommand{\CC}{\ensuremath{\mathcal {C}}\xspace}

\newcommand{\CF}{\ensuremath{\mathcal {F}}\xspace}

\newcommand{\CH}{\ensuremath{\mathcal {H}}\xspace}
\newcommand{\CI}{\ensuremath{\mathcal {I}}\xspace}

\newcommand{\CK}{\ensuremath{\mathcal {K}}\xspace}

\newcommand{\CO}{\ensuremath{\mathcal {O}}\xspace}

\newcommand{\Ad}{{\mathrm{Ad}}}

\DeclareMathOperator{\End}{End}

\DeclareMathOperator{\Adm}{Adm}

\def\fbq{\mathbf q}

\DeclareMathOperator{\Hom}{Hom}

\let\Im\relax
\DeclareMathOperator{\Im}{Im}

\newcommand{\rig}{{\mathrm{rig}}}
\newcommand{\nrig}{{\mathrm{nrig}}}

\def\tW{\tilde W}

\def\kk{\mathbf k}
\DeclareMathOperator{\supp}{supp}

\newtheorem{theorem}{Theorem}
\newtheorem{proposition}[theorem]{Proposition}
\newtheorem{lemma}[theorem]{Lemma}
\newtheorem {conjecture}[theorem]{Conjecture}
\newtheorem{corollary}[theorem]{Corollary}

\theoremstyle{definition}

\newtheorem{example}[theorem]{Example}

\newtheorem{remark}[theorem]{Remark}

\numberwithin{equation}{section}
\numberwithin{theorem}{section}

\setitemize[0]{leftmargin=*,itemsep=\the\smallskipamount}
\setenumerate[0]{leftmargin=*,itemsep=\the\smallskipamount}

\renewcommand{\to}{%
   \ifbool{@display}{\longrightarrow}{\rightarrow}%
   }
\let\shortmapsto\mapsto
\renewcommand{\mapsto}{%
   \ifbool{@display}{\longmapsto}{\shortmapsto}%
   }
\newlength{\olen}
\newlength{\ulen}
\newlength{\xlen}
\newcommand{\xra}[2][]{%
   \ifbool{@display}%
      {\settowidth{\olen}{$\overset{#2}{\longrightarrow}$}%
       \settowidth{\ulen}{$\underset{#1}{\longrightarrow}$}%
       \settowidth{\xlen}{$\xrightarrow[#1]{#2}$}%
       \ifdimgreater{\olen}{\xlen}%
          {\underset{#1}{\overset{#2}{\longrightarrow}}}%
          {\ifdimgreater{\ulen}{\xlen}%
             {\underset{#1}{\overset{#2}{\longrightarrow}}}
             {\xrightarrow[#1]{#2}}}}%
      {\xrightarrow[#1]{#2}}
   }
\makeatother
\newcommand{\xyra}[2][]{%
   \settowidth{\xlen}{$\xrightarrow[#1]{#2}$}%
   \ifbool{@display}%
      {\settowidth{\olen}{$\overset{#2}{\longrightarrow}$}%
       \settowidth{\ulen}{$\underset{#1}{\longrightarrow}$}%
       \ifdimgreater{\olen}{\xlen}%
          {\mathrel{\xymatrix@M=.12ex@C=3.2ex{\ar[r]^-{#2}_-{#1} &}}}%
          {\ifdimgreater{\ulen}{\xlen}%
             {\mathrel{\xymatrix@M=.12ex@C=3.2ex{\ar[r]^-{#2}_-{#1} &}}}
             {\mathrel{\xymatrix@M=.12ex@C=\the\xlen{\ar[r]^-{#2}_-{#1} &}}}}}%
      {\mathrel{\xymatrix@M=.12ex@C=\the\xlen{\ar[r]^-{#2}_-{#1} &}}}%
   }
\makeatletter
\newcommand{\xla}[2][]{%
   \ifbool{@display}%
      {\settowidth{\olen}{$\overset{#2}{\longleftarrow}$}%
       \settowidth{\ulen}{$\underset{#1}{\longleftarrow}$}%
       \settowidth{\xlen}{$\xleftarrow[#1]{#2}$}%
       \ifdimgreater{\olen}{\xlen}%
          {\underset{#1}{\overset{#2}{\longleftarrow}}}%
          {\ifdimgreater{\ulen}{\xlen}%
             {\underset{#1}{\overset{#2}{\longleftarrow}}}
             {\xleftarrow[#1]{#2}}}}%
      {\xleftarrow[#1]{#2}}
   }
\newcommand{\isoarrow}{%
   \ifbool{@display}{\overset{\sim}{\longrightarrow}}{\xrightarrow\sim}%
   }

\newcommand{\tocless}[2]{\bgroup\let\addcontentsline=\nocontentsline#1{#2}\egroup}
   
\begin{document}

\title[Hecke algebras and $p$-adic groups]{Hecke algebras and $p$-adic groups}
\author[X. He]{Xuhua He}
\address{Department of Mathematics, University of Maryland, College Park, MD 20742}
\email{xuhuahe@math.umd.edu}

\thanks{}

\subjclass[2010]{}

\date{\today}

\begin{abstract}
This survey article, is written as an extended note and supplement of my lectures in the current developments in mathematics conference in 2015. We discuss some recent developments on the conjugacy classes of affine Weyl groups and $p$-adic groups, and some applications to Shimura varieties and to representations of affine Hecke algebras. 
\end{abstract}

\maketitle

\tableofcontents

\section*{Introduction}

\subsection{}
\addtocontents{toc}{\protect\setcounter{tocdepth}{1}}
In \cite{Tits}, Tits explains the analogy between the symmetric group $S_n$ and the general linear group over a finite field $\mathbb F_q$ and indicates that $S_n$ should be regarded as the general linear group over $\mathbb F_1$, the field of one element. 

Following Tits' philosophy, we may informally regard affine Weyl groups as reductive groups over $\mathbb Q_1$, the $1$-adic field. Although it might be premature to develop the theory of the $1$-adic field at the current stage, we do have a fairly good understanding of conjugacy classes of affine Weyl groups, together with the length function on them, and such knowledge allows us to reveal a great part of the structure of conjugacy classes of $p$-adic groups. 

In this note, we will focus on some questions related to conjugacy classes and we will look at the two sides of a coin:

\begin{itemize}
\item An affine Weyl group as a degeneration of a $p$-adic group;

\item A $p$-adic group as a deformation of an affine Weyl group.
\end{itemize}

We will see how such considerations help us to understand several important problems and we will discuss some applications in arithmetic geometry (of Shimura varieties) and representation theory (of affine Hecke algebras). Note that both areas have been intensively studied. To enumerate the recent achievements in each area would lead to a more-than-100-page survey article and it is not my intention to do so here. I will just focus on a few problems to illustrate the mysterious connection between affine Weyl groups and $p$-adic groups. 

\subsection{}
\addtocontents{toc}{\protect\setcounter{tocdepth}{1}}
A Coxeter group is generated by simple reflections and has a well-defined length function on it. It is obvious that in a given conjugacy class of the Coxeter group, there exists a sequence of conjugations by simple reflections, starting from any given element in the class and ending with a minimal-length element. For various reasons (both algebraic and geometric), it is desirable to have such a sequence, in which the lengths of the elements {\bf weakly decrease}. If such a sequence exists, then for many questions (on Hecke algebras, algebraic groups, Deligne-Lusztig varieties, etc.) one may reduce the questions on arbitrary elements to questions on minimal length elements. 

Moreover, in a given conjugacy class, in general, there are more than one minimal length elements. For example, in type $E_8$, the conjugacy class $E_6(a_2)+A_2$ has $403200$ elements and $16374$ of them are of minimal length. Hence, it is also desirable to have a better understanding of the relations between these minimal length elements. 

Minimal length elements were first studied by Geck and Pfeiffer \cite{GP93}. They discover that the minimal length elements, in a finite Weyl group, have many remarkable properties. 

Such properties have found important applications in the study of the ``character table'' of finite Hecke algebras (see e.g. \cite{GP93}, and \cite{GM}), in the study of Deligne-Lusztig theory on the representations of finite groups of Lie type (see, e.g. \cite{OR}, \cite{BR}, and \cite{HL}) and in the study of links between conjugacy classes in finite Weyl groups and unipotent conjugacy classes in reductive groups (see e.g. \cite{Lu11}).

\subsection{}
\addtocontents{toc}{\protect\setcounter{tocdepth}{1}}
A natural question to ask is whether such properties hold for affine Weyl groups. In the study of this question, I find it enlightening to consider an affine Weyl group as a degeneration of a $p$-adic group. In fact, the theory of $p$-adic groups motivates us in the discovery of the following interesting features of affine Weyl groups:

\begin{enumerate}
\item The arithmetic invariants of conjugacy classes of affine Weyl groups;

\item The straight conjugacy classes and the reduction from non-straight conjugacy classes to straight conjugacy classes;

\item The ``special'' partial conjugacy classes and their distinguished representatives;

\item The partial order on the set of ``special'' partial conjugacy classes and on the set of straight conjugacy classes;

\item The parameterization of conjugacy classes in terms of standard quadruples. 
\end{enumerate} 

This is what we will discuss in Part I. 

Items (1) and (2) are motivated by Kottwitz's classification of the Frobenius-twisted conjugacy classes of $G(\breve \BQ_p)$, the reductive group $G$ over the completion of maximal unramified extension of $\BQ_p$. They are key ingredients in the understanding of conjugacy classes of affine Weyl groups together with the length function on them. 

Items (3) and (4) (for ``special'' partial conjugacy classes) were discovered in the process of understanding Lusztig's $G$-stable pieces and closure relations between these pieces. 

Item (4) for straight conjugacy classes were discovered in the process of understanding the closure relations between Frobenius-twisted conjugacy classes of $G(\breve F)$ and the closure relations between Newton strata of Shimura varieties. 

Item (5) is motivated by the parametrization of Frobenius-twisted conjugacy classes of $G(\breve F)$ in terms of basic Frobenius-twisted conjugacy classes of Levi subgroups. Such parametrization leads to the definition of the rigid cocenter of affine Hecke algebras, which play a crucial role in our discussion in Part III. 

\subsection{}
\addtocontents{toc}{\protect\setcounter{tocdepth}{1}} 
Now let us look at the other side of the coin. We will see how the knowledge of conjugacy classes of affine Weyl groups, together with the length function, helps us to study some problems related to conjugacy classes of $p$-adic groups. 

Let $\s$ be the Frobenius morphism on $G(\breve F)$. Let $\breve \CI$ be a $\s$-stable Iwahori subgroup of $G(\breve F)$. In Part II, we study the intersection of $\breve \CI \dot w \breve \CI$ with a $\s$-conjugacy class $[b]$ of $G(\breve F)$. We focus on some basic questions: 
\begin{itemize}
\item When the intersection is nonempty?

\item If so, what is the dimension? 
\end{itemize}

We also regard this intersection as a group-theoretic model to understand the relation between the Kottwitz-Rapoport strata (related to certain elements $w$) and the Newton strata (related to certain $\s$-conjugacy classes $[b]$) of Shimura varieties. 

It is worth mentioning that there is another group-theoretic model to understand those strata of Shimura varieties, affine Deligne-Lusztig varieties $$X_w(b)=\{g \breve \CI \in G(\breve F)/\breve \CI; g \i b \s(g) \in \breve \CI \dot w \breve \CI\}.$$ 

It is easy to see that $X_w(b) \neq \emptyset$ if and only if $\breve \CI \dot w \breve \CI \cap [b] \neq \emptyset$. However, there are a few differences between $X_w(b)$ and $\breve \CI \dot w \breve \CI$. 

\begin{itemize}
\item We have the freedom to degenerate the $\s$-conjugacy class $[b]$ when considering $\breve \CI \dot w \breve \CI \cap [b]$. Thus $\breve \CI \dot w \breve \CI \cap [b]$ may be used to understand closure relations between Newton strata of Shimura varieties. 

\item Although $\breve \CI \dot w \breve \CI \cap [b]$ is infinite dimensional, there is a way to define the dimension of it (which is a finite number). This dimension is expected to be the dimension of the intersection of the corresponding Kottwitz-Rapoport stratum and corresponding Newton stratum in Shimura varieties. 

\item The affine Deligne-Lusztig variety $X_w(b)$, in general, may contain infinitely many irreducible components. The intersection $\breve \CI \dot w \breve \CI \cap [b]$, on the other hand, has only finitely many irreducible components. 
\end{itemize}

Except these differences, the method we use to study affine Deligne-Lusztig varieties and the intersection $\breve \CI \dot w \breve \CI \cap [b]$ are similar. In \cite{He99} we develop a reduction method to study affine Deligne-Lusztig varieties, and give a connection between affine Deligne-Lusztig varieties and class polynomials of the associated affine Hecke algebras. The method and the result remain valid, mutatis mutandis, for our new model $\breve \CI \dot w \breve \CI \cap [b]$ as well. The key idea is to do the reduction on $\breve \CI \dot w \breve \CI$ by $\s$-conjugation using a sequence of conjugations by simple reflections in affine Weyl group which weakly decrease the length.  

After establishing the connection between the intersection $\breve \CI \dot w \breve \CI \cap [b]$ and class polynomials, we discuss several situations where the explicit non-emptiness pattern and/or dimension formula are established, including: 

\begin{itemize}
\item The non-emptiness pattern and dimension formula of $\breve \CK \e^\mu \breve \CK \cap [b]$ for any special maximal parahoric subgroup $\breve \CK$. 

\item The non-emptiness pattern of $\breve \CI \dot w \breve \CI \cap [b]$ for basic $\s$-conjugacy class $[b]$. 

\item The dimension formula of $\breve \CI \dot w \breve \CI \cap [b]$ for basic $\s$-conjugacy class $[b]$ and an element $w$ in the Shrunken Weyl chamber. 

\item The non-emptiness pattern and dimension formula of $\breve \CI \dot w \breve \CI \cap [\e^\mu]$ for a residually split group and an element $w$ in the very Shrunken Weyl chamber. 

\item The non-emptiness pattern of $\breve \CK \Adm(\mu)_K \breve \CK \cap [b]$ for any parahoric subgroup $\breve \CK$. 
\end{itemize}

\subsection{}
\addtocontents{toc}{\protect\setcounter{tocdepth}{1}} 
Now we move to a different topic, the affine Hecke algebras. 

A basic philosophy in representation theory is that {\it characters tell all}. This is the case for finite groups. What happens for affine Hecke algebras? 

In Part III, we will discuss the case for affine Hecke algebras. The strategy is to first understand the cocenter of affine Hecke algebras, and then to investigate how the cocenter controls (both the ordinary and the modular) representations. 

An affine Hecke algebra, is a deformation of the group algebra of an affine Weyl group. The cocenter of the group algebra, is very simple. The elements in the same conjugacy class have the same image in the cocenter and the cocenter has a standard basis indexed by conjugacy classes. 

For affine Hecke algebras, the situation is more complicated. The elements in the same conjugacy class may not have the same image in the cocenter. In order to understand the cocenter of affine Hecke algebras, in addition to conjugacy classes, one also need to take into account the length of the elements inside the given conjugacy class. The upshot is that the cocenter still has a standard basis indexed by conjugacy classes, but they are the image of minimal length elements in the conjugacy class, not arbitrary element compared to the group algebra case. 

Start from an element, not necessarily of minimal length in its conjugacy class, we have a reduction method to reach (possibly more than one) minimal length elements and write the image of this element in the cocenter as a linear combination of the standard basis, with coefficient the so-called class polynomials. In other words, class polynomials encode the information of the reduction method, which is used both in the study of cocenter of affine Hecke algebras and in the study of the intersection $\breve \CI \dot w \breve \CI \cap [b]$. This is the reason that the algebraic object class polynomial is a powerful tool to study the geometric object $\breve \CI \dot w \breve \CI \cap [b]$. 

Back to representation theory, one difficulty is that there are infinitely many conjugacy classes of affine Weyl groups as well as infinitely many irreducible representations of affine Hecke algebras. Motivated by the basic $\s$-conjugacy classes of $G(\breve F)$, we have a parametrization of conjugacy classes in terms of their Newton points, and there is a way to reduce any conjugacy class to a conjugacy class with a central Newton point. This leads to the definition of rigid cocenter and rigid quotient of the Grothendieck group of the representations of affine Hecke algebras. The main theorem in \cite{CH2} is that the trace map induces a perfect pairing between the rigid cocenter and the rigid quotient for affine Hecke algebras with generic parameters. This result leads to the density Theorem and trace Paley-Wiener theorem on the relation between the cocenter and representations of affine Hecke algebras. 



We expect that the rigid cocenter also plays a key role in the study of modular representations of affine Hecke algebras, and we propose two conjectures: 
\begin{itemize}
\item The naive kernel conjecture, which predicts how the density theorem fails for affine Hecke algebras at roots of unity and in positive characteristic; 

\item The rigid determinant conjecture, which predicts the deformation of the representations with respect to the change of (generic) parameters. 
\end{itemize}

\subsection{}
\addtocontents{toc}{\protect\setcounter{tocdepth}{1}} 
Most of the results in this note are known by now (although scattered around the literature, with some mild assumption). However, we take this opportunity to present a few new materials, as well as remove some unnecessarily assumptions from the old results in the literature, including: 

\begin{itemize}
\item Some results on the minimal length elements and partial order on the partial conjugacy classes for an affine Weyl group action (by conjugation) on a Coxeter group. See Theorem \ref{par-min} and Proposition \ref{par}. 

\item A parametrization of conjugacy classes of extended affine Weyl groups in terms of standard quadruples. See Theorem \ref{1-19}.

\item The use of $\breve \CI \dot w \breve \CI \cap [b]$ as a group-theoretic model to study some stratifications of Shimura varieties. See \S \ref{Shimura}.

\item The non-emptiness pattern and dimension formula for $\breve \CK \e^\mu \breve \CK \cap [b]$ for special maximal parahoric subgroup $\breve \CK$ (in the literature, it is only stated for hyperspecial maximal parahoric subgroups). See Theorem \ref{2-24} and Theorem \ref{kmuk}.

\item The dimension formula for $\breve \CI \dot w \breve \CI \cap [b]$ for basic $[b]$ and $w$ in the Shrunken Weyl chamber (in the literature, it is only stated for residually split groups). See Theorem \ref{bb-dim}.

\item The non-emptiness pattern and dimension formula for $\breve \CI \dot w \breve \CI \cap [\e^\mu]$ for $G$ a residually split group and $w$ in the very Shrunken Weyl chamber. See Theorem \ref{torus2}. 

\item The naive kernel conjecture \ref{naive} and the rigid determinant conjecture \ref{rig-det} of affine Hecke algebras. 

\end{itemize}

I am grateful to G. Lusztig for all he has taught me about the algebraic groups and Weyl groups over the years. Without his big influence this project would never be initiated. In more recent times, I also learned a lot from J. Adams, D. Ciubotaru, U. G\"ortz, T. Haines, J. Michel, S. Nie, X. Zhu, and especially M. Rapoport, who teaches me Shimura varieties. I am happy to express my gratitude to all of them. This survey aticle is written as an extended note of the talk at the Current Developments in Mathematics conferernce 2015 at Harvard. I thank the organizers for the invitation, and for the opportunity to write this survey. I also thank J. Adams, U. G\"ortz and T. Haines for their remarks on a preliminary version of this note. My research is partially supported by NSF grant DMS-1463852.

\section{Part I: Combinatorics of affine Weyl groups}

\subsection{Coxeter groups}
\addtocontents{toc}{\protect\setcounter{tocdepth}{2}}
A {\it Coxeter system} is a pair $(W, \BS)$, where $\BS$ is a finite set and $W$ is the group generated by the elements in $\BS$, subject to certain relations. First, we have $s^2=1$ for all $s \in \BS$. We call the elements in $\BS$ the {\it simple reflections} of $W$. Next, we have relations between two elements $s \neq t$ in $\BS$. Let $m_{s t} \in \{2, 3, \cdots\} \cup \{\infty\}$ be the order of $s t$. Since $s$ and $t$ have order $2$, we have \[s t s \cdots=t s t \cdots,\] where both sides have exactly $m_{st}$ factors. 
If $m_{st}=2$, then $s t=t s$, i.e.,  the two simple reflections $s$ and $t$ commute. If $m_{st}=3$, then $s$ and $t$ satisfy the Artin's braid relation $s t s=t s t$. If $m_{st}=\infty$, then there is no relation between $s$ and $t$. 

In additional to the group structure, the length function $\ell: W \to \BN$ plays a crucial role in this note. Here for any $w \in W$, $\ell(w)$ is the minimal integer $n$ such that $w=s_1 \cdots s_n$, where $s_i \in \BS$.

Weyl groups (both finite and affine) are important examples of Coxeter groups. They play a crucial role in the study of Lie groups and $p$-adic groups. Now let me give two examples. 

\begin{example}\label{sn}
Let $W=S_n$ be the symmetry group of $\{1, 2, \cdots, n\}$. It is the Weyl group of $GL_n(\BC)$. It is a Coxeter group with the generators $s_i=(i, i+1)$ for $1 \le i \le n-1$. We have a presentation of $W$ in terms of generators and relations: $$W=\<s_1, \cdots, s_{n-1} \mid s_i^2=1, s_i s_{i+1} s_i=s_{i+1} s_i s_{i+1} , s_i s_j=s_j s_i \text{ if } |i-j| \ge 2\>.$$ The length function is given by $$\ell(\s)=\sharp\{(i, j); 1 \le i<j \le n; \s(i)>\s(j)\}.$$
\end{example}

\begin{example}
Let $\tW=\BZ^n \rtimes S_n$, where $S_n$ acts on $\BZ^n$ by permutation. This is the Iwahori-Weyl group of $GL_n(\BQ_p)$. Following \cite[1.12]{Lu-Hecke}, we may realize $W$ as the group of periodic permutations on $\BZ$, i.e., the permutations $f: \BZ \to \BZ$ such that $f(z+n)=f(z)+n$ for all $z \in \BZ$. Define $\chi: W \to \BZ$ by $\chi(f)=\sum_{z=1}^n (f(z)-z)$ and $W_a=\ker(\chi)$. Then $W_a$ is a Coxeter group with the generators $s_i$ for $i=0, 1, \cdots n-1$, where $s_i: \BZ \to \BZ$ is defined by $$s_i(z)=\begin{cases} z+1, & \text{ if } z \equiv i \mod n; \\ z-1, & \text{ if } z \equiv i+1 \mod n; \\ z, & \text{ otherwise}. \end{cases}$$ The order $m_{i j}$ of $s_i s_j$ is given as follows: 

If $n=2$, then $m_{0 1}=\infty$. 

If $n \ge 3$, then $m_{i j}=\begin{cases} 3, & \text{ if } i-j \equiv \pm 1 \mod n; \\ 2, & \text{ otherwise}. \end{cases}$. 

The group $\tW$ is not a Coxeter group. But it is almost as good as a Coxeter group. We may realize $\tW$ as $$\tW=W_a \rtimes \Omega.$$ Here $\Omega \cong \BZ$ is the group generated by $\t \in \tW$, where $\t(m)=m+1$ for all $m \in \BZ$. 
There is a length function $\ell$ on $\tW$, defined by $$\ell(\s)=\sharp(Y_\s/\t_n),$$ where $Y_\s=\{(i, j) \in \BZ \times \BZ; i<j, \s(i)>\s(j)\}$ and $Y_\s/\t_n$ is the (finite) set of $\t_n$-orbits on $Y_\s$ for the $\t_n$ action given by $(i, j) \mapsto (i+n, j+n)$. 

The restriction of this length function to $W_a$ is the length function of the Coxeter group $W_a$. The subgroup $\Omega$ of $\tW$ is the subgroup of length-zero elements.
\end{example}

\subsection{Weakly length-decreasing sequences}\label{decreasing}

Let $(W, \BS)$ be a Coxeter system. For any $J \subset \BS$, let $W_J$ be the subgroup of $W$ generated by $J$. Then $(W_J, J)$ is again a Coxeter system. The two extreme cases are $W$ (for $J=\BS$) and $\{1\}$ (for $J=\emptyset$). 

Let $\CO$ be a $W_J$-orbit on $W$ for a given $W_J$-action. We are interested in the relations between the elements, together with their length, in this orbit. More specifically, we would like to have the following property: 

\smallskip 

\textbf{Red-Min}: For any $w \in W$, there exists a sequence of simple reflections $s_1, \cdots, s_k \in J$ such that
\begin{itemize}
\item For any $1 \le i \le k$, $\ell(s_i s_{i-1} \cdots s_1 \cdot w) \le \ell(s_{i-1} \cdots s_1 \cdot w)$; 

\item The element $s_k s_{k-1} \cdots s_1 \cdot w$ is of minimal length among all the elements in $\CO$. 
\end{itemize}

If this property holds, then one may reduce the questions on any element $w$ in $\CO$ to the questions on the minimal length elements in $\CO$ via the inductive step: $$w_1 \rightsquigarrow s \cdot w_1 \text{ for some } s \in \BS \text{ with } \ell(s \cdot w_1) \le \ell(w_1).$$ This reduction method plays an important role in the study of many problems related to Coxeter groups. 

\begin{example}
We consider the action of $W_J$ on $W$ by left multiplication. Let ${}^J W$ be the set of elements $w$ in $W$ of minimal length in the cosets $W_J w$. Then each $W_J$-orbit contains a unique element $x$ in ${}^J W$. Moreover, $\ell(w x)=\ell(w)+\ell(x)$ for any $w \in W_J$. In particular, $x$ is the unique minimal length element in the $W_J$-orbit of $x$. Using the structure of the Coxeter group $W_J$, it is easy to see that the \textbf{Red-Min} property holds for this orbit. 
\end{example}

\subsection{Conjugation by simple reflections}

We are especially interested in the conjugation action of $W$ and the partial conjugation action (of $W_J$) on $W$. 

Conjugacy classes of $W$ play an important role in the study of representations of $W$ and representations of the corresponding Hecke algebra, which is a deformations of the group algebra of $W$. They also play an important role in the study of conjugacy classes of algebraic groups and $p$-adic groups. 

The $W_J$-conjugacy classes in $W$, for finite and affine Weyl groups $W$ and their proper subgroups $W_J$, are closely related to the conjugation action of parabolic subgroups on algebraic groups and to the conjugacy action of parahoric subgroups on $p$-adic groups. They also have important application to arithmetic geometry and to representation theory. 

We are going to discuss the applications in more details in Part II and in Part III.

\smallskip

We first investigate a few examples of the elements in a fixed conjugacy class of $W$ and discuss some nice properties. 

\begin{example}
Let $W=S_5$ and $C$ be the conjugacy class of $5$-cycles. Then $\sharp C=4!=24$ and the lengths of elements in $C$ are $4, 6$ or $8$. 

Let $w=(1 2 3 4 5)$ and $w'=(1 2 4 5 3)$ be elements in $C$ of minimal length. They are related via $$w=(1 2 3 4 5) \overset{s_1}\approx (2 1 3 4 5) \overset{s_2}\approx (3 1 2 4 5)=w'.$$ 

Here $\approx$ means conjugation by simple reflections which preserve the length.

Suppose we start with an element of $C$ not of minimal length, say $w''=(1 4 2 3 5)$. Then $w''$ is related to a minimal length element in $C$ via $$w''=(1 4 2 3 5) \xrightarrow{s_3} (1 3 2 4 5) \xrightarrow{s_2} (1 2 3 4 5)=w.$$

Here $\to$ means conjugation by simple reflections which weakly decrease the length. 
\end{example}

\begin{example}
Let $W=S_5$ and $C$ be the conjugacy class of transpositions. The minimal length elements in $C$ are the simple reflections. They are not $\approx$-equivalent to each other. However, they are still $\sim$-equivalent. Here $\sim$-equivalent means that the corresponding elements in the associated Braid group are conjugate. 
\end{example}

\smallskip

Now let us investigate two examples of partial-conjugacy classes in $W$. Although left multiplication by $W_J$ and conjugate by $W_J$ are very different, the elements in ${}^J W$ still play a crucial role in the study of the conjugation action of $W_J$ on $W$. 

The following example is a $W_J$-conjugacy class that contains an element in ${}^J W$. 

\begin{example}\label{1.6} Let $W=S_4$ and $J=\{2, 3\}$. 

Let $w=s_1 s_2 s_3 \in {}^J W$ and $w'=s_2 s_3 s_2 s_1 s_2$ be an element in the $W_J$-orbit of $w$. The element $w'$ has larger length than $w$ and $w'$ is related to $w$ via 
$$w'={s_2 s_3 s_2} {s_1 s_2} \xrightarrow{s_2} {s_3 s_2}  {s_1} \xrightarrow{s_3} {s_2 s_3} {s_1} \xrightarrow{s_2} {s_3} {s_1 s_2} \xrightarrow{s_3}{s_1 s_2 s_3}=w.$$



\end{example}

\smallskip 

The next example is a $W_J$-conjugacy class that does not contain an element in ${}^J W$. 

\begin{example}\label{1.7} Let $W=S_5$ and $J=\{1, 3, 4\}$.

Let $w=s_4 s_3 s_2 s_1 s_3 s_2s_4$. It is not of minimal length in its $W_J$-orbit. We have 
$$w={s_4 s_3} {s_2 s_1 s_3 s_2 s_4} \xrightarrow{s_4} {s_3} {s_2 s_1 s_3 s_2} \xrightarrow{s_3} {s_1} {s_2 s_1 s_3 s_2}  \xrightarrow{s_1} {s_3} {s_2 s_1 s_3 s_2} .$$


The elements ${s_1} {s_2 s_1 s_3 s_2}$ and ${s_3} {s_2 s_1 s_3 s_2}$ are of minimal length in the $W_J$-orbit of $w$. Notice that these two elements have the same ${}^J W$-part $x=s_2 s_1 s_3 s_2 \in {}^J W$ and prefix $s_1$ and $s_3$ are elements in $W_J$ conjugate to each other by the element $x$. 
\end{example}

\subsection{$\sim$-equivalence and $\approx$-equivalence}\label{sim-approx}

Conjugacy classes of $W$ play an important role in the study of split groups. Twisted conjugacy classes, on the other hand, play a similar role in the study of non-split groups. In the sequel, we will not only consider the conjugation action, but will consider the twisted conjugation action as well. 

Now let us introduce some notation. 

Let $(W, \BS)$ be a Coxeter system and $J, J' \subset \BS$. Let $\d: W_J \to W_{J'}$ be a group automorphism such that $\d(J)=J'$. We consider the $\d$-twisted partial conjugation of $W_J$ on $W$ defined by $$w \cdot_\d w'=w w' \d(w) \i.$$

We denote by $(W_J)_\d=\{(w, \d(w)); w \in W_J\} \subset W \times W$ the $\d$-graph of $W_J$. We denote by $W/(W_J)_\d$ the set of $\d$-twisted $W_J$-conjugacy classes of $W$. If $\d$ is the identity map, we may write $(W_J)_\D$ for $(W_J)_\d$. 

Let $w, w' \in W$. We write $w \to_{J, \d} w'$ if there exists a sequence of simple reflections $s_1, \cdots, s_k \in J$ such that 
\begin{itemize}
\item $w'=s_k s_{k-1} \cdots s_1 \cdot_\d w$; 

\item for any $1 \le i \le k$, $\ell(s_i s_{i-1} \cdots s_1 \cdot_\d w) \le \ell(s_{i-1} \cdots s_1 \cdot_\d w)$. 
\end{itemize}
We write $w \approx_{J, \d} w'$ if $w \to_{J, \d} w'$ and $w' \to_{J, \d} w$. 

\smallskip

As we have seen in the above examples, a $\d$-twisted $W_J$-conjugacy class $\CO$ of $W$ may contain more than one minimal length elements. We denote by $\CO_{\min}$ the set of minimal length elements in $\CO$ and by $\ell(\CO)$ the length of $w$ for any $w \in \CO_{\min}$. Now we introduce some equivalence relations between the elements in $\CO_{\min}$. 

We write $w \sim_{J, \d} w'$ if there exists a sequence of elements $x_1, \cdots, x_k \in W_J$ such that $w'=x_k x_{k-1} \cdots x_1 \cdot_\d w$ and for any $1 \le i \le k$, 
\begin{itemize}

\item $\ell(x_i x_{i-1} \cdots x_1 \cdot_\d w) =\ell(w)$;

\item $\ell(x_i x_{i-1} \cdots x_1 w \d(x_{i-1} x_{i-2} \cdots x_1) \i)$ or $\ell(x_{i-1} x_{i-2} \cdots x_1 w \d(x_i x_{i-1} \cdots x_1) \i)$ equals $\ell(w)+\ell(x_i)$. 
\end{itemize}

Note that both $\sim_{J, \d}$ and $\approx_{J, \d}$ are equivalence relations and for $w, w' \in W$, 
$$w \approx_{J, \d} w' \Longrightarrow w \sim_{J, \d} w' \Longrightarrow w' \in W_J \cdot_\d w \text{ and } \ell(w)=\ell(w').$$
But the converse, in general, does not hold. 

If $\d$ is the identity map, then we omit $\d$ in the under script. If $J=\BS$, then we omit $J$ in the under script.

\subsection{The discovery of Geck and Pfeiffer}

\subsubsection{Minimal length elements in a finite Coxeter group}
In the study of characters of finite Hecke algebras, Geck and Pfeiffer \cite{GP93} discovers that the minimal length elements in a given conjugacy class of a finite Weyl group has remarkable properties. 

\begin{theorem}\label{f-min}
Let $W$ be a finite Coxeter group and $\d: W \to W$ is a group automorphism with $\d(\BS)=\BS$. Let $\CO \in W/W_\d$. Then 

\begin{enumerate}
\item \textbf{Red-Min} property holds for $\CO$. 

\item $\CO_{\min}$ is a single $\sim_\d$-equivalence class. 

\item If moreover, $\CO$ is elliptic, i.e. $\CO \cap W_J=\emptyset$ for any $J=\d(J) \subsetneqq \BS$, then $\CO_{\min}$ is a single $\approx_\d$-equivalence class. 
\end{enumerate}
\end{theorem}

Part (1) and (2) for the ordinary conjugacy classes of a finite Weyl group are first proved by Geck and Pfeiffer in \cite[Theorem 1.1]{GP93}. The ordinary conjugacy classes for any finite Coxeter groups (including the non-crystallographic types) are studied in \cite{GP00}, where Part (3) is also proved. Part (1) and (2) for the twisted conjugacy classes of a finite Coxeter group are obtained by Geck, Kim and Pfeiffer in \cite[Theorem 2.6]{GKP}. Part (3) for the twisted conjugacy classes is established in \cite[Theorem 7.5]{HeMin}. The proofs in these paper involve a case-by-case analysis. A conceptual proof is found recently in \cite{HN12}. 

\subsubsection{Parametrization of the conjugacy classes of a finite Coxeter group}
Let $\CO$ be a $\d$-twisted conjugacy class of $W$. Let $J$ be a minimal $\d$-stable subset of $\BS$ with $\CO \cap W_J \neq \emptyset$. Then by definition, $\CO \cap W_J$ is a union of $\d$-twisted elliptic conjugacy classes of $W_J$. A nice property is that the elliptic classes never fuse. 

\begin{theorem}\label{neverfuse}
Let $\CO$ be a $\d$-twisted conjugacy class of $W$. Let $J$ be a minimal $\d$-stable subset of $\BS$ with $\CO \cap W_J \emptyset$. Then $\CO \cap W_J$ is a single $\d$-twisted conjugacy class of $W_J$. 
\end{theorem}

For ordinary conjugacy classes, it is due to Geck and Pfeiffer in \cite[Theorem 3.2.11]{GP00}. The general case is in \cite[Theorem 2.3.4]{CH}. 

Now we give a parametrization of $W/W_\d$ in terms of parabolic subgroups $W_J$ and their elliptic conjugacy classes. Set $$\G_\d=\{(J, C); J \subset \BS \text{ with } \d(J)=J, C \text{ is  an elliptic $\d$-twisted conjugacy class of } W_J\}.$$
We say that the pairs $(J, C)$ and $(J', C')$ in $\G_\d$ are equivalent  if there exits $x \in W^\d$ such that $J'=x J x \i$ and $C'=x C x \i$. The following result is obtained in \cite[Proposition 2.4.1]{CH}. 

\begin{proposition}\label{1-10}
Let $W$ be a finite Coxeter group. The map $$f: \G_\d \to W/W_\d, \qquad (J, C) \mapsto W \cdot_\d C$$ induces a bijection between the equivalence classes of $\G_\d$ and the set of $\d$-twisted conjugacy classes of $W$. 
\end{proposition}

\subsection{Partial conjugation action}\label{par-i}

Let $(W, \BS)$ be a Coxeter system and $J, J' \subset \BS$. Let $\d: W_J \to W_{J'}$ such that $\d(J)=J'$. We call a $\d$-twisted $W_J$-conjugacy class {\it distinguished} if it contains an element in ${}^J W$. As we have seen from Example \ref{1.6} and example \ref{1.7}, one may relate any $\d$-twisted $W_J$-conjugacy class to a distinguished one. To make it precise, let me first introduce some notation. 

For $w \in {}^J W$, we write $\text{Ad}(w) \d(J)=J$ if for any simple reflection $s \in J$, there exists a simple reflection $s' \in J$ such that $w \d(s) w^{-1}=s'$. In this case, $w \in {}^J W^{\d(J)}$. It is easy to see that for any $J_1, J_2 \subset \BS$, and $w \in {}^{J_1 \cup J_2} \tilde W$, if $\text{Ad}(w) \d(J_i)=J_i$ for $i=1, 2$, then $\text{Ad}(w) \d(J_1 \cup J_2)=J_1 \cup J_2$.  Thus for any $J \subset \BS$ and $w \in {}^J W$, the set $\{J' \subset J\mid  \text{Ad}(w) \d(J')=J'\}$ contains a unique maximal element. We set $$I(J, w, \d)=\max\{J' \subset J\mid  \text{Ad}(w) \d(J')=J'\}.$$ 

\begin{example}
Let $W=S_4$, $J=\{2, 3\}$ and $w=s_1 s_2 s_3$. Then $I(J, w, id)=\emptyset$. 

Let $W=S_5$, $J=\{1, 3, 4\}$ and $w=s_2 s_1 s_3 s_2$. Then $I(J, w, id)=\{1, 3\}$. 
\end{example}

\smallskip

The classification of $\d$-twisted $W_J$-conjugacy classes of $W$ is obtained in \cite[Section 2]{HeMin}. 

\begin{theorem}\label{w-dec}
We have

(1) $W=\sqcup_{w \in {}^J W} W_J \cdot_\d (W_{I(J, w, \d)} w)$.

(2)  For any $w \in {}^J W$, the inclusion $W_{I(J, w, \d)} \to W_J \cdot_\d (W_{I(J, w, \d)} w)$ induces a natural bijection between the $\Ad(w) \circ \d$-twisted conjugacy classes $C$ on $W_{I(J, w, \d)}$ and the $\d$-twisted $W_J$-conjugacy classes $\CO$ on $W_J \cdot_\d (W_{I(J, w, \d)} w)$. 
\end{theorem}

We also have the following relation between the minimal length elements of $C$ and of $\CO$. This is proved in \cite[\S 3]{HeMin}.  

\begin{theorem}\label{par-main}
Let $w \in {}^J W$ and $C$ be a $\Ad(w) \circ \d$-twisted conjugacy class in $W_{I(J, w, \d)}$. Let $\CO=W_J \cdot_\d (C w)$. Then 

(1) For any $x \in \CO$, there exists $u \in C$ such that $x \to_{J, \d} u w$. 

(2) If $C$ satisfies the \textbf{Red-Min} property, then so is $\CO$. 

(3) If $C_{\min}$ is a single $\sim_{\Ad(w) \circ \d}$-equivalence class, then $\CO_{\min}$ is a single $\sim_\d$-equivalence class. 

(4) If $C_{\min}$ is a single $\approx_{\Ad(w) \circ \d}$-equivalence class, then $\CO_{\min}$ is a single $\approx_\d$-equivalence class. 

\end{theorem}

In particular, we have

\begin{theorem}\label{par-min}
Suppose that $W_J$ is a finite Coxeter group or an affine Weyl group. Let $\CO$ be a $\d$-twisted $W_J$-conjugacy class of $W$. Then

\begin{enumerate}
\item \textbf{Red-Min} property holds for $\CO$. 

\item $\CO_{\min}$ is a single $\sim_\d$-equivalence class. 

\item If moreover, $\CO$ is distinguished, then $\CO_{\min}$ is a single $\approx_\d$-equivalence class. 
\end{enumerate}
\end{theorem}

The case where $W_J$ is a finite Coxeter group is obtained in \cite[\S 3]{HeMin} by combining Theorem \ref{par-main} with Theorem \ref{f-min}. A different proof is obtained by Nie in \cite{Ni13}. The case where $W_J$ is an affine Weyl group can be obtained in the same way by using Theorem \ref{a-min} we are going to discuss instead of Theorem \ref{f-min}. 

\subsection{Affine Weyl groups}

\subsubsection{Definition}
We first recall the definition of affine Weyl groups. 

Let $\sR = (X^*,X_*, R, R^\vee, \Pi)$ be a based reduced root datum, i.e., $X^*$ and $X_*$ are free abelian groups of finite rank together with a perfect pairing $\< ~ , ~ \> : X^* \times X_* \rightarrow \BZ$, $R \subset X^*$ is the set of roots, $R^\vee \subset X_*$ is the set of coroots and $\Pi \subset R$ is the set of simple roots. Let $V=X_* \otimes \BR$. For any $\alpha \in R$, we have a reflection $s_\alpha$ on $V$ defined by $s_{\alpha}(x) = x - \langle \alpha , x \rangle \alpha^\vee$. For any $\a \in R$ and $k \in \BZ$, we have an affine root $\tilde \a=\a+k$ and an affine reflection $s_{\tilde \a}$ on $V$ defined by $s_{\tilde \a}(x)=x-(\<\a, x\>+k) \a^\vee$. 

The finite Weyl group $W_0$ is the subgroup of $GL(V)$ generated by $s_\a$ for $\a \in R$. The affine Weyl group and the extended affine Weyl group are defined by $$W_a=\BZ R^\vee \rtimes W_0, \qquad \tW=X_* \rtimes W_0.$$ Let $\text{Aff}(V)$ be the group of affine transformations on $V$. We may identify $\tW$ with the subgroup of $\text{Aff}(V)$ generated by $t^\l$ for $\l \in X_*$ and $W_0$, where $t^\l$ is the translation $v \mapsto v-\l$ on $V$. We may also identify $W_a$ with the subgroup of $\text{Aff}(V)$ generated by $s_{\tilde \a}$, where $\tilde \a$ runs over all the affine roots. 

The set $\Pi$ of simple roots determines the set $R_+$ of positive roots, the dominant chamber $$\CC=\{x \in V; \<\a, x\> >0 \text{ for all } \a \in \Pi\},$$ the anti-dominant chamber $$\CC^-=\{x \in V; \<\a, x\> <0 \text{ for all } \a \in \Pi\}$$ and the base alcove $$\fka=\{x \in V; -1<\<\a, x\><0 \text{ for all } x \in R_+\} \subset \CC^-.$$

Let $\BS_0=\{s_\a; \a \in \Pi\}$. Then $(W_0, \BS_0)$ is a Coxeter system. Let $\tilde \BS$ be the set of $s_{\tilde \a}$, where the hyperplane $H_{\tilde \a}=V^{\tilde \a}$ is a wall of $\fka$. Then $(W_a, \tilde \BS)$ is again a Coxeter system. 

The length of an element in $\tW$ is the number of walls between the base alcove $\fka$ and the image of $\fka$ under the action of that element. The explicit formula for length function is given by Iwahori and Matsumoto in \cite{IM}: for $\l \in X_*$ and $w \in W_0$. 
$$\ell(t^\l w)=\sum_{\a \in R_+, w \i(\a) \in R_+} |\<\a, \l\>|+\sum_{\a \in R_+, w \i(\a) \in R_-} |\<\a, \l\>-1|.$$ 

This length function extends the length function of the Coxeter group $W_a$ and we have a semi-direct product $$\tW=W_a \rtimes \Omega,$$ where $\Omega=\{w \in \tW; \ell(w)=0\}$ is the stabilizer of the base alcove $\fka$. 

The definitions in \S \ref{sim-approx} still make sense for $\tW$. The only modification that we need is the definition of $\approx$-equivalence relation on $\tW$. We write $w \approx_\d w'$ if there exists $\t \in \Omega$ such that $w \to_\d \t w' \d(\t) \i$ and $\t w' \d(\t) \i \to_\d w$. 

For any $J \subset \Pi$, set $\tW_J=X_* \rtimes W_J$. We call $\tW_J$ a {\it parabolic subgroup} of $\tW$. This is the extended affine Weyl group associated to the root datum $\sR_J=(X^*, X_*, R_J, R^\vee_J, J)$, where $R_J \subset R$ is the sub root system spanned by the roots in $J$ and $R^\vee_J \subset R^\vee$ is the set of corresponding coroots. 

For any $K \subsetneqq \tilde \BS$ with $W_K$ finite, we call $W_K$ a {\it parahoric subgroup} of $\tW$. 

\subsubsection{Arithmetic invariants}\label{I-ari} Let $\d: \tW \to \tW$ be a group automorphism with $\d(\tilde \BS)=\tilde \BS$. We first discuss the classification of $\d$-twisted conjugacy classes of $\tW$. Motivated by Kottwitz's work \cite{Ko85} and \cite{Ko97}, we introduce two arithmetic invariants. 

Since $\d$ preserve the length function on $\tW$, we have $\d(\Omega)=\Omega$. Let $\Omega_\d$ be the set of $\d$-coinvariants on $\Omega$. The Kottwitz map $\k=\k_{\tW, \d}: \tW \to \Omega_\d$ is obtained by composing the natural projection map $\tW \to \tW/W_a \cong \Omega$ with the projection map $\Omega \to \Omega_\d$. It is constant on each $\d$-twisted conjugacy class of $\tW$. This gives one invariant.

Another invariant is given by the Newton map. For any $w \in \tW$, we regard the element $w \d$ as an element in $\text{Aff}(V)$. There exists $n \in \BN$ such that $(w \d)^n=t^\xi$ for some $\xi \in X_*$. Let $\nu_{w, \d}=\xi/n$ and $\bar \nu_{w, \d} \in \CC$ be the unique dominant element in the $W_0$-orbit of $\nu_{w, \d}$. We call $\bar \nu_{w, \d}$ the {\it Newton point} of $w$ (with respect to the $\d$-conjugation action). It is known that $\nu_{w, \d}$ is independent of the choice of $n$ and $\bar \nu_{w, \d}=\bar \nu_{w', \d}$ if $w$ and $w'$ are $\d$-conjugate. Let $\CO$ be a $\d$-twisted conjugacy class. We write $\nu_\CO=\bar \nu_{w, \d}$ for any $w \in \CO$ and we call $\nu_\CO$ the Newton point of $\CO$. 

Define a map $$f: \tW/(\tW)_\d \to \Omega_\d \times \CC, \quad w \mapsto (\k(\CO), \nu_\CO).$$

This map is not surjective. The Newton point of an element $w \in \tW$ is a rational dominant coweight satisfying an integrality conditions. The map is not injective, either. Any $\d$-twisted conjugacy class that intersects a parahoric subgroup is mapped to $(1, 0)$. However, there is a nice lifting of $\Im(f)$ in $\tW/(\tW)_\d$. This is given in terms of straight $\d$-twisted conjugacy classes. We will discuss it in the next subsection.  

\subsection{Straight and non-straight conjugacy classes}\label{str}

\subsubsection{Definition}
Let $w \in \tW$. We say that $w$ is {\it $\d$-straight} if $\ell(w \d(w) \cdots \d^{n-1}(w))=n \ell(w)$ for all $n \in \BN$. We may formulate it in a slightly different way. Regard $\d$ as an element in the group $\tW \rtimes \<\d\>$ and extend the length function from $\tW$ to $\tW\rtimes \<\d\>$ by requiring $\ell(\d)=0$. Then $(w \d)^n=w \d(w) \cdots \d^{n-1}(w) \d^n \in \tW \rtimes \<\d\>$ and $w$ is $\d$-straight if and only if $\ell((w\d)^n)=n \ell(w)$ for all $n \in \BN$. By \cite[\S 2.4]{He99}, $w$ is $\d$-straight if and only if $\ell(w)=\<2 \rho, \bar \nu_{w, \d}\>$, where $\rho$ is the half sum of positive roots in $R$. We say that a $\d$-twisted conjugacy $\CO$ of $\tW$ is {\it straight} if it contains a $\d$-straight element. The following result is obtained in \cite[Theorem 3.3]{HN14}.

\begin{theorem}\label{str-bw} 
The map $f: \tW/\tW_\d \to \Omega_\d \times \CC$ induces a bijection between the set of straight $\d$-twisted conjugacy classes and $Im(f)$. 
\end{theorem}

Note that in a straight $\d$-twisted conjugacy class $\CO$, the set of minimal length elements are exactly the set of $\d$-straight element in it. Moreover, if $\CO'$ is another $\d$-twisted conjugacy classes with $f(\CO')=f(\CO)$, then $\ell(w')> \<2 \rho, \nu_\CO\>$ for any $w \in \CO'$. 

\subsubsection{Minimal length elements} We have seen in \S \ref{par-i} that the map from non-distinguished $\d$-twisted $W_J$-conjugacy classes to distinguished $\d$-twisted $W_J$-conjugacy classes is ``compatible'' with the length function. There is a similar relation between non-straight conjugacy classes and straight conjugacy classes of $\tW$. It is proved in \cite[Proposition 2.7]{HN14}. 

\begin{theorem}\label{a-red}
Let $\CO \in \tW/(\tW)_\d$ and $\CO'$ be a straight $\d$-twisted conjugacy class with $f(\CO)=f(\CO')$. Then for any $w \in \CO$, there exists a triple $(x, K, u)$ with $w \to_\d u x$, where $x$ is a $\d$-straight element in $\CO'$, $K$ is a subset of $\tilde \BS$ such that $W_K$ finite, $x \in {}^K \tW$ and $\Ad(x) \d(K)=K$, and $u$ is an element in $W_K$. 
\end{theorem}

The following result is obtained in \cite[Theorem 2.9]{HN14} by combining Theorem \ref{a-red} with Theorem \ref{f-min}.

\begin{theorem}\label{a-min}
Let $W$ be an affine group and $\d: W \to W$ is a group automorphism with $\d(\BS)=\BS$. Let $\CO \in W/W_\d$. Then 

\begin{enumerate}
\item \textbf{Red-Min} property holds for $\CO$. 

\item $\CO_{\min}$ is a single $\sim_\d$-equivalence class. 
\end{enumerate}
\end{theorem}

In each distinguished $\d$-twisted $W_J$-conjugacy classes, we have a canonical minimal length element, i.e., the element in that conjugacy class that is also of minimal length with respect to the left multiplication of $W_J$. The straight $\d$-twisted conjugacy classes, in general, does not have such a canonical choice of minimal length element. However, all the minimal length elements in a given straight $\d$-twisted conjugacy class still form a single $\approx$-equivalence class. This is proved in \cite[Theorem 3.8]{HN14}. 

\begin{theorem}
Let $\CO$ be a straight $\d$-twisted conjugacy class of $\tW$. Then $\CO_{\min}$ is a single $\approx_\d$-equivalence class. 
\end{theorem}

\subsubsection{Standard quadruples}\label{st-qua}
In this subsection, we give standard representative for the straight $\d$-twisted conjugacy classes and give a parametrization of the $\d$-twisted conjugacy classes. 

For $J \subset \Pi$, set $$\fka_J=\{x \in V; -1< \<\a, x\><0 \text{ for all } \a \in R_{J, +}\}.$$ Let $\tilde J$ be the set of simple reflections of the $\tW_J$, in other words, $\tilde J$ consists of the reflections $s_\a$, where the hyperplanes $H_\a$ is a wall of $\fka_J$. We denote by $\ell_J$ the length function on the quasi-Coxeter group $\tW_J$ and by $\Omega_J \subset \tW_J$ the stabilizer of $\fka_J$. Note that $\ell_J$ is different from the restriction to $\tW_J$ of the length function $\ell$ on $\tW$. 

We say that $(J, x, K, C)$ is a {\it standard quadruple} if 
\begin{enumerate}

\item $J \subset \Pi$ with $\d_0(J)=J$, where $\d_0 \in GL(V)$ is the linear part of $\d \in \text{Aff}(V)$;

\item $\Ad(x) \circ \d$ normalizes $\tW_J$ and induces a permutation on the sef of affine simple reflections of $\tW_J$ and $\<\a, \nu_x\>>0$ for all $\a \in \Pi-J$;

\item $K \subset \tilde J$ with $W_K$ finite and $\Ad(x) \d(K)=K$.

\item $C$ is an elliptic $\Ad(x) \circ \d$-twisted conjugacy class of $W_K$. 
\end{enumerate}

Note that the condition (2) implies that $J$ is the subset of $\Pi$ consisting of simple roots $\a$ with $\<\a, \nu_{x, \d}\>=0$. In other words, $J$ is determined by the element $x$. Similarly, $C$ is contained in the affine Weyl group $W_{J, a}$ of the root datum $\sR_J$ and $K$ is determined by $C$ by taking the support of any element in $C$. It is also worth mentioning that $\nu_{x, \d}$ is dominant and $\nu_{u x, \d}=\nu_{x, \d}$ for any $u \in C$. 

Let $\tilde \G_\d$ be the set of standard quadruples. We say that the standard quadruples $(J, x, K, C)$ and $(J', x', K', C')$ are {\it equivalent} in $\tW$ if $J=J'$, there exists $\t \in \Omega_J$ with $x'=\t x \d(\t) \i$ and there exists $w \in \tW_J$ with $x' \d(w) (x') \i=w$ and $C'=w \t C (w \t) \i$. 

\begin{theorem}\label{1-19}
Let $\tW$ be an extended affine Weyl group. The map $$f: \tilde \G_\d \to \tW/\tW_\d, \qquad (J, x, K, C) \mapsto \tW \cdot_\d C x$$ 
induces a bijection between the equivalence classes of standard quadruples and the set of $\d$-twisted conjugacy classes of $\tW$. 

Moreover, a $\d$-twisted conjugacy class of $\tW$ is straight if and only if the standard quadruples corresponding to it are of the form $(J, x, \emptyset, \{1\})$. 
\end{theorem}

\begin{remark} 
(1) Let $(J, x, K, C)$ be a standard quadruples. Note that for $w, w' \in C$, $w x$ and $w' x$ are $\d$-conjugate by an element in $W_K \subset \tW$. Thus $\tW \cdot_\d C x$ is a single $\d$-twisted conjugacy class of $\tW$. 

(2) A different parametrization is given in \cite[Proposition 5.8]{HN15}. It is more convenient to use this parameterization in the study of cocenter and representations of affine Hecke algebras. 

(3) For any standard quadruple $(J, x, K, C)$, $x$ is contained in a straight $\d$-twisted conjugacy class of $\tW$, but $x$ may not be a $\d$-straight element. 
\end{remark}

\begin{proof}
We first show that every $\d$-twisted conjugacy class $\CO$ of $\tW$ comes from a standard quadruple. By Theorem \ref{a-red}, there exists a triple $(x, K, u)$ with $u x \in \CO_{\min}$, where $x$ is a $\d$-straight element, $K$ is a subset of $\tilde \BS$ such that $W_K$ finite, $x \in {}^K \tW$ and $\Ad(x) \d(K)=K$, and $u$ is an element in $W_K$. We may assume furthermore that $K$ is the smallest subset of $\tilde \BS$ that is $\Ad(x) \circ \d$-stable and contains $\supp(u)$. 

Let $J=\{\a \in \Pi; \<\a, \bar \nu_{x, \d}\>=0\}$ and $z \in {}^J W_0$ with $z(\nu_{x, \d})=\bar \nu_{x, \d}$. Let $x'=z x \d(z) \i$, $K'=z(K)$ and $C$ be the $\Ad(x) \circ \d$-twisted conjugacy class of $W_{K'}$ that contains $z u z \i$. Since $x$ is a $\d$-straight element, $x'$ is a length-zero element in $\tW_J$. As explained in \cite[\S 3.3]{HNx}, $z(K) \subset \tilde J$. By the smallest assumption on $K$, $C$ is an elliptic $\Ad(x) \circ \d$-twisted conjugacy class of $W_K$. Therefore $(J, x', K', C)$ is a standard quadruple associated to $\CO$. 

Now let $(J_1, x_1, K_1, C_1), (J_2, x_2, K_2, C_2)$ be standard quadruples associated to the same $\d$-twisted conjugacy class $\CO$ of $\tW$. Then $\nu_{x_1, \d}=\nu_{x_2, \d}=\nu_\CO$ and $k(x_1)=k(x_1)=k(\CO)$. Thus $J_1=J_2$ and by Theorem \ref{str-bw}, $x_1$ and $x_2$ are in the same straight $\d$-twisted conjugacy class. 

Set $J=J_1$. Let $u_i \in C_i$. Let $w \in \tW$ such that $u_2 x_2=w u_1 x_1 \d(w) \i$. Since $w(\nu_{x_1, \d})=\nu_{x_2, \d}=\nu_{x_1, \d}$, we have $w \in \tW_J$. We write $w$ as $w=w_1 \t$, where $w_1 \in W_{J, a}$ and $\t$ is a length-zero element of $\tW_J$. Then we have $x_2=\t x_1 \d(\t) \i$ as an element in $\Omega_J$. Set $K'_2=\Ad(\t)(K_1)$ and $C'_2=\Ad(\t)(C_1)$. Then $(J, x_2, K'_2, C'_2)$ is a standard quadruple and $(J, x_2, K'_2, C'_2)$ is equivalent to $(J, x_1, K_1, C_1)$. 

We have that $C_2 x_2$ and $C'_2 x_2$ are in the same $\d$-conjugacy class of $\tW_J$. In other words, $C_2$ and $C'_2$ are in the same $\d'$-twisted conjugacy class of $\tW_J$, where $\d'=\Ad(x_2) \circ \d$. Set $V_J=R_J^\vee \otimes_\BZ \BR$. For any $u_2 \in C_2$,  $V_J^{u_2 \d'}$ is spanned by the coweights $\o^\vee_j$ for $j \notin K_2$. For any $u'_2 \in C'_2$,  $V_J^{u'_2 \d'}$ is spanned by the coweights $\o^\vee_j$ for $j \notin K'_2$. Since $C_2$ and $C'_2$ are in the same $\d'$-conjugacy class, by \cite[Proposition 2.4.1]{CH} there exists $w_2 \in \tW_J$ with $\d'(w_2)=w_2$ and $\Ad(w_2)(K_2)=K'_2$. Then $C'_2$ and $w_2 C_2 w_2 \i$ are both elliptic $\d'$-twisted conjugacy classes of $W_{K'_2}$ with the same characteristic polynomials and the same minimal length. By \cite[Theorem 7.5]{HeMin}, $C'_2=w_2 C_2 w_2 \i$. 

The ``moreover'' part is proved in \cite[Proposition 3.2]{HN14}. 
\end{proof}

\subsubsection{Examples} 

In this subsection, we give two examples, the extended affine Weyl groups associated to $GL_3$ and $SL_3$, and their conjugacy classes. 

\begin{example}\label{gl3}
Let us consider the based root datum of $GL_3$. Then $X^*=\BZ e_1 \oplus \BZ e_2 \oplus \BZ e_3 \cong \BZ^3$ and $X^*=\BZ e^\vee_1 \oplus \BZ e^\vee_2 \oplus\BZ e^\vee_3 \cong \BZ^3$, where $\<e_i, e_j^\vee\>=\d_{i j}$. The simple roots are $\a_1=e_1-e_2$ and $\a_2=e_2-e_3$. 

Let $\tW$ be the extended affine Weyl group associated to this root datum and $\d$ be the identity map on $\tW$. Let $\CO$ be a conjugacy class of $\tW$. The Newton point of $\CO$ is an element $(a_1, a_2, a_3) \in \BQ^3$ with $a_1 \ge a_2 \ge a_3$ that satisfies the integrality condition:
$$\text{For any }c \in \BQ, c \, \sharp \{i; a_i=c\} \in \BZ.$$
The image of $\CO$ under the Kottwitz map is determined by the Newton point $$\k(\CO)=a_1+a_2+a_3 \in \BZ \cong \Omega.$$ 

In Table \ref{table1}, we list the number of conjugacy classes with given Newton point $(a_1, a_2, a_3)$, and the associated standard quadruples. 
\begin{table}[H]
\caption{Conjugacy classes of extended affine Weyl group of $GL_3$}
\label{table1}
\begin{tabular}{|c|c|c|c|}
\hline
\multicolumn{2}{|c|}{Conditions on $(a_1, a_2, a_3)$} & Number & Standard quadruples
\\ \hline
\multicolumn{2}{|c|}{$a_1>a_2>a_3$} & $1$ & $(\emptyset, t^{(a_1, a_2, a_3)}, \emptyset, \{1\})$ \\ \hline
\multirow{3}{*}{$a_1=a_2>a_3$} & \multirow{2}{*}{$a_1 \in \BZ$} & \multirow{2}{*}{2} & $(\{1\}, t^{(a_1, a_2, a_3)}, \emptyset, \{1\})$ \\ \cline{4-4} & & & $(\{1\}, t^{(a_1, a_2, a_3)}, \{1\}, \{s_1\})$ \\ \cline{2-4} & $a_1=n+\frac{1}{2}$ & $1$ & $(\{1\}, t^{(n+1, n, a_3)} s_1, \emptyset, \{1\})$  \\ \hline
\multirow{3}{*}{$a_1>a_2=a_3$} & \multirow{2}{*}{$a_2 \in \BZ$} & \multirow{2}{*}{2} & $(\{2\}, t^{(a_1, a_2, a_3)}, \emptyset, \{1\})$ \\ \cline{4-4} & & & $(\{2\}, t^{(a_1, a_2, a_3)}, \{2\}, \{s_2\})$  \\ \cline{2-4} & $a_2=n+\frac{1}{2}$ & $1$ & $(\{2\}, t^{(a_1, n+1, n)} s_2, \emptyset, \{1\})$  \\ \hline
\multirow{5}{*}{$a_1=a_2=a_3$} & \multirow{3}{*}{$a_1 \in \BZ$} & \multirow{3}{*}{3} & $(\{1, 2\}, t^{(a_1, a_1, a_1)}, \emptyset, \{1\})$  \\ \cline{4-4} & & & $(\{1, 2\}, t^{(a_1, a_1, a_1)}, \{i\}, \{s_i\})$ for $0 \le i \le 2$ \\ \cline{4-4} & & & $(\{1, 2\}, t^{(a_1, a_1, a_1)}, \{i, j\}, \{s_i s_j, s_j s_i\})$ for $0\le i<j \le 2$  \\ \cline{2-4} & $a_1=n+\frac{1}{3}$ & $1$ &  $(\{1, 2\}, t^{(n+1, n, n)} s_1 s_2, \emptyset, \{1\})$   \\ \cline{2-4} & $a_1=n+\frac{2}{3}$ & $1$ & $(\{1, 2\}, t^{(n+1, n+1, n)} s_2 s_1, \emptyset, \{1\})$  \\ \hline
\end{tabular}
\end{table}
\end{example}

\begin{example}\label{sl3}
Let us consider the based root datum of $SL_3$. Here $$V=\{a e_1^\vee+b e_2^\vee+c e_3^\vee; a+b+c=0\}.$$ Let $W_a$ be the associated affine Weyl group and $\d$ be the identity map. Let $\CO$ be a conjugacy class of $W_a$. The Newton point of $\CO$ is an element $(a_1, a_2, a_3) \in \BQ^3$ with $a_1 \ge a_2 \ge a_3$ and $a_1+a_2+a_3=0$. 

In Table \ref{table2}, we list some information on the conjugacy classes of $\tW$. Here the difference from Table \ref{table1} comes from the fact that the element $s_0 s_1, s_0 s_2$ and $s_1 s_2$ are conjugate in $\tW$, but not in $W_a$. 
\begin{table}[H]
\caption{Conjugacy classes of the affine Weyl group of $SL_3$}
\label{table2}
\begin{tabular}{|c|c|c|c|}
\hline
\multicolumn{2}{|c|}{Conditions on $(a_1, a_2, a_3)$} & Number & Standard quadruples
\\ \hline
\multicolumn{2}{|c|}{$a_1>a_2>a_3$} & $1$ & $(\emptyset, t^{(a_1, a_2, a_3)}, \emptyset, \{1\})$ \\ \hline
\multirow{3}{*}{$a_1=a_2>a_3$} & \multirow{2}{*}{$a_1 \in \BZ$} & \multirow{2}{*}{2} & $(\{1\}, t^{(a_1, a_2, a_3)}, \emptyset, \{1\})$ \\ \cline{4-4} & & & $(\{1\}, t^{(a_1, a_2, a_3)}, \{1\}, \{s_1\})$ \\ \cline{2-4} & $a_1=n+\frac{1}{2}$ & $1$ & $(\{1\}, t^{(n+1, n, a_3)} s_1, \emptyset, \{1\})$  \\ \hline
\multirow{3}{*}{$a_1>a_2=a_3$} & \multirow{2}{*}{$a_2 \in \BZ$} & \multirow{2}{*}{2} & $(\{2\}, t^{(a_1, a_2, a_3)}, \emptyset, \{1\})$ \\ \cline{4-4} & & & $(\{2\}, t^{(a_1, a_2, a_3)}, \{2\}, \{s_2\})$  \\ \cline{2-4} & $a_2=n+\frac{1}{2}$ & $1$ & $(\{2\}, t^{(a_1, n+1, n)} s_2, \emptyset, \{1\})$  \\ \hline
\multicolumn{2}{|c|}{\multirow{5}{*}{{$a_1=a_2=a_3=0$}}} & \multirow{5}{*}{5} & $(\{1, 2\}, 1, \emptyset, \{1\})$ \\ \cline{4-4} \multicolumn{2}{|l|}{}   & & $(\{1, 2\}, 1, \{i\}, \{s_i\})$ for $0 \le i \le 2$  \\ \cline{4-4} \multicolumn{2}{|l|}{}  & & $(\{1, 2\}, 1, \{1, 2\}, \{s_1 s_2, s_2 s_1\})$  \\ \cline{4-4} \multicolumn{2}{|l|}{}  & & $(\{1, 2\}, 1, \{0, 1\}, \{s_1 s_0, s_0 s_1\})$  \\ \cline{4-4} \multicolumn{2}{|l|}{}  & & $(\{1, 2\}, 1, \{0, 2\}, \{s_0 s_2, s_2 s_0\})$  \\\hline
\end{tabular}
\end{table}

\end{example}

\subsection{Weakly length-increasing sequences}\label{increasing}

Instead of considering the weakly length-decreasing sequences as in \S \ref{decreasing}, one may consider the weakly length-increasing sequences in a given twisted conjugacy class. 

Let $W$ be a finite Coxeter group and $w_0$ be the longest element of $W$. Then the map $$W \to W, \qquad w \mapsto w w_0$$ is an order-reversing map and it sends $\d$-twisted conjugacy classes to $\Ad(w_0) \circ \d$-twisted conjugacy classes. Thus as a variation of Theorem \ref{f-min} (1), we have 

\begin{theorem}
Let $W$ be a finite Coxeter group and $\CO$ be a $\d$-twisted conjugacy class of $W$. Let $\CO_{\max}$ be the set of maximal length elements in $\CO$. Then for any $w \in \CO$, there exists $w' \in \CO_{\max}$ with $w' \to_\d w$. 
\end{theorem}

A similar statement holds for affine Weyl groups. 

\begin{theorem}
Let $\tW$ be an extended affine group and $\CO$ be a $\d$-twisted conjugacy class of $\tW$. Let $w \in \CO$. If $w$ is not of maximal length in $\CO$, then there exists $w' \in \CO$ with $\ell(w')>\ell(w)$ and $w' \to_\d w$. 
\end{theorem}

For ordinary conjugacy classes, this is proved by Rostami in \cite[Main Theorem 1.1]{Ro}. The general case can be proved using the method in \cite{HN14}.

The maximal length elements and the weakly length-increasing sequences play a role in the study of the center of Hecke algebras. 

\subsection{The quotient space $W \sslash (W_J)_\d$}\label{preceq}

\subsubsection{The quotient $W \sslash (W_J)_\d$}
As we have seen in \S\ref{par-i} and \S\ref{str}, some $\d$-twisted $W_J$-conjugacy classes of $W$ are special and any non-special conjugacy classes can be ``reduced'' to a special conjugacy class. 

\begin{itemize}
\item For partial conjugation action, the special classes are the distinguished ones. 

\item For (non-twisted and twisted) conjugation action of an extended affine Weyl group, the special classes are the straight conjugacy classes. 
\end{itemize}

In either case, we denote by $W \sslash (W_J)_\d$ the set of special $\d$-twisted conjugacy classes in $W/(W_J)_\d$. We then have a natural map $$\pi: W/(W_J)_\d \to W \sslash (W_J)_\d.$$ Here the notation $\sslash$ is borrowed from the geometric invariant theory. 

\smallskip

In the rest of this subsection, we will introduce a partial order on $W\sslash (W_J)_\d$ so that the quotient space becomes a topological space. We will discuss in Part II some applications of these topological spaces. 

\subsubsection{The Bruhat order} We first recall the definition of Bruhat order. Let $(W, \BS)$ be a Coxeter system. The Bruhat order on $W$ is defined as follows. 

Let $w \in W$ and $w=s_1 \cdots s_n$ with $n=\ell(w)$ and $s_i \in \BS$ for all $i$. Such expression is called a {\it reduced expression} of $w$. Let $w'$ be another element in $W$. We say that $w' \le w$ if there exists $1 \le i_1<i_2<\cdots<i_k \le n$ such that $w'=s_{i_1} s_{i_2} \cdots s_{i_k}$. The definition of Bruhat order is in fact independent of the choice of reduced expressions of $w$. In the case where $W$ is a finite (resp. an affine) Weyl group, the Bruhat order describes the closures relations between the Schubert cells in the finite (resp. affine) flag varieties. 

Now we discuss some general facts about the Bruhat order. 

\begin{lemma}\label{par-1}
Let $x, w, w'\in W$ with $x \le w$ and $w' \to_{J, \d} w$. Then there exists $x' \in W$ such that $x \to_{J, \d} x'$ and $x' \le w'$. 
\end{lemma}

\begin{proof}
It suffices to prove the case where $w'=s w \s(s)$ for some $s \in J$. 

If $w'>w$, then we may take $x'=x$. Now we assume that $\ell(w)=\ell(w')$. Without loss of generalization, we may assume furthermore that $s w<w$ and $w \d(s)>w$. 

If $s x>x$, then by \cite[Corollary 2.5]{Lu-Hecke}, $x \le s w$ and hence $x<w'$. 

If $s x<x$, then $\ell(s x \d(s)) \le \ell(x)$ and $x \to_{J, \d} s x \d(s)$. By \cite[Corollary 2.5]{Lu-Hecke}, $s x \le s w$ and $s x \d(s) \le s w \d(s)$. 
\end{proof}

\begin{proposition}
Let $\CO \in W/(W_J)_\d$. If \textbf{Red-Min} property holds for $\CO$, then $\CO_{\min}$ is the set of minimal element in $\CO$ with respect to the Bruhat order. 
\end{proposition}

\begin{proof}
Let $w, w' \in \CO$. If $w' \le w$, then $\ell(w') \le \ell(w)$. Thus if $w \in \CO_{\min}$, then $w$ is a minimal element in $\CO$ with respect to the Bruhat order. 

On the other hand, let $w \in \CO$. Suppose that $w \notin \CO_{\min}$, then by \textbf{Red-Min} property, there exists $w' \in \CO_{\min}$ with $w \to_{J, \d} w'$. By Lemma \ref{par-1}, there exists $w_1 \approx_{J, \d} w'$ with $w_1<w$. 
\end{proof}

\subsubsection{The partial order $\preceq$}

\begin{proposition}\label{par}
Suppose that $W_J$ is a finite Coxeter group or is an affine Weyl group. Let $\CO, \CO' \in W \sslash (W_J)_\d$. Then the following conditions are equivalent:

(1) For some $w' \in \CO'_{\min}$, there exists $w \in \CO_{\min}$ with $w \le w'$. 

(2) For any $w' \in \CO'_{\min}$, there exists $w \in \CO_{\min}$ with $w \le w'$. 
\end{proposition}

The case where $W_J$ is a finite Coxeter group is obtained in \cite[Corollary 4.6]{HeMin}, by combining Lemma \ref{par-1} with Theorem \ref{f-min}. The case where $W_J$ is an affine Weyl group is obtained in the same way by using Theorem \ref{a-min} instead of Theorem \ref{f-min}. 

\smallskip

Now we introduce a partial order on $W \sslash (W_J)_\d$. Let $\CO, \CO' \in W \sslash (W_J)_\d$. We write $\CO \preceq \CO'$ if the conditions in Proposition \ref{par} are satisfied. By the equivalence of the two conditions in Proposition \ref{par}, $$\CO \preceq \CO' \text{ and } \CO' \preceq \CO'' \Longrightarrow \CO \preceq \CO''.$$
Thus $\preceq$ gives a partial order on $W \sslash (W_J)_\d$. 

It is worth mentioning that although we state the results for twisted partial conjugacy classes for a Coxeter group, the definition of $W\sslash(W_J)_\d$ and the partial order still works if we replace $W$ by a quasi-Coxeter group. In particular, the definition and results here still holds for extended affine Weyl groups as well. 

The partial order $\preceq$, for finite Weyl group $W$ and $J \subsetneqq \BS$, is first introduced in \cite{He-G} to describe the closure relations between Lusztig's $G$-stable pieces \cite{Lu-par1} and \cite{Lu-par2}. The partial order $\preceq$, for affine Weyl group $W$, is used to describe the closure relations of several interesting stratifications of loop groups, as we will discuss in Part II, \S \ref{closure-ii}.

\subsubsection{The partial order on ${}^J W$}
If $J \subsetneqq \BS$, then there is a natural bijection between $W\sslash (W_J)_\d$ and ${}^J W$. This partial order on ${}^J W$ can be described in an explicit way as follows:

Let $w, w' \in {}^J W$, then $W_J \cdot_\d w \preceq W_J \cdot_\d w'$ if and only if there exists $u \in W_J$ such that $u w \d(u) \i \le w'$.

In this case, we also write $w \preceq_{J, \d} w'$.

Note that the partial order $\preceq_{J, \d}$ on ${}^J W$ is different from the restriction to ${}^J W$ of the Bruhat order on $W$. The latter one implies the first one, but not vice verse in general. The difference can be seen in the following example.

\begin{example}
Let $W=S_4$ and $J=\{3\}$. The simple reflections of $W$ are $s_1, s_2, s_3$. We simply write $s_{a b c \cdots}$ instead of $s_a s_b s_c \cdots$. Let $\d$ be the nontrivial diagram automorphism, i.e., $\d(s_1)=s_3$, $\d(s_2)=s_2$ and $\d(s_3)=s_1$. In Figure \ref{fig1.1}, we draw the Hasse diagram of ${}^J W$, with respect to the usual Bruhat order, the partial order $\preceq_{J, id}$ (the extra relation is in dotted line) and the partial order $\preceq_{J, \d}$ (the extra relation is in double dotted line). 
\begin{figure}
  \caption{Hasse diagram for the partial orders on ${}^J W$}
\[
\begin{xy}
0;<0.7cm,0pt>:
(0, 8)*{s_{12132}}="12132";
(-1,6)*{s_{2132}}="2132";
(1,6)*{s_{1213}}="1213";
(-2,4)*{s_{123}}="123";
(0,4)*{s_{121}}="121";
(2,4)*{s_{213}}="213";
(-2,2)*{s_{12}}="12";
(0,2)*{s_{21}}="21";
(2,2)*{s_{23}}="23";
(-1,0)*{s_1}="1";
(1,0)*{s_2}="2";
(0,-2)*{1}="0";
"1" **@{-};
"12" **@{-};
"123" **@{-};
"2132" **@{.};
"12132" **@{-};
"1213" **@{-};
"121" **@{-};
"12" **@{-};
"2" **@{-};
"0" **@{-};
"1";
"21" **@{-};
"121" **@{-};
"2132" **@{-};
"213" **@{-};
"1213" **@{-};
"121" **@{-};
"2";
"21" **@{-};
"213" **@{-};
"23" **@{-};
"2" **@{-};
"123";
"1213" **@{-};
"1";
"23" **@{:};
\end{xy}
\]
\label{fig1.1}
\end{figure}
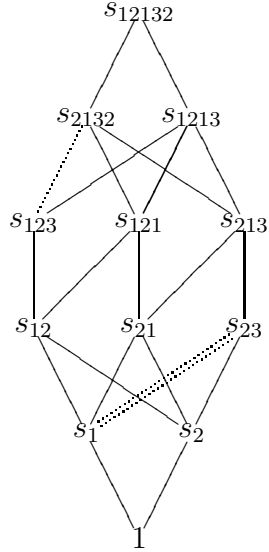
\end{example}

\subsubsection{The partial order on the set of straight conjugacy classes}
Let $\tW$ be an extended affine Weyl group. Proposition \ref{par} gives a partial order $\preceq$ on the set of straight $\d$-twisted conjugacy classes of $\tW$. On the other hand, we have an injective map $$f: \tW/(\tW)_\d \to \Omega_\d \times \CC, \quad w \mapsto (\k(\CO), \nu_\CO).$$ There is a natural partial order on $\Omega_\d \times \CC$ given as follows. 

For $(v, p), (v', p') \in \Omega_\d \times \CC$, we wrie $(v, p) \le (v', p')$ if $v=v'$ and $p$ is less than or equal to $p'$ in the dominance order, i.e., $p'-p$ is a nonnegative linear combination of positive coroots. It is proved in \cite[Theorem 3.1]{He00} that these two partial orders coincide. 

\begin{theorem}
Let $\CO, \CO' \in \tW/(\tW)_\d$. Then $\CO \preceq \CO'$ if and only if $f(\CO) \le f(\CO')$. 
\end{theorem}


\section{Part II: Geometry of $G(\breve F)$}

\subsection{Bruhat decomposition on $G(\kk)$}

As a warm-up, we first recall the Bruhat decomposition of a complex reductive group. 

Let $G$ be a connected reductive group over an algebraically closed field $\kk$. Let $B$ be a Borel subgroup and $T \subset B$ be a maximal torus. Let $N$ be the normalizer of $T$. For any element $w$ in the Weyl group $W_0=N(T)/T$, we choose a representative $\dot w$ in $N(T)$. 

The group $G(\kk)$ is a union of Bruhat cells: $$G(\kk)=\sqcup_{w \in W_0} B \dot w B.$$ 

Moreover, the closure of a Bruhat cell $B \dot w B$ is a union of Bruhat cells and the closure relation is given by the Bruhat order of the Weyl group: $$\overline{B \dot w B}=\sqcup_{w' \le w} B \dot w' B.$$

The dimension of a Bruhat cell $B \dot w B$ is given in terms of length function of the Weyl group: $$\dim (B \dot w B)-\dim B=\ell(w).$$

A variation is the decomposition of the flag variety $G/B$ into a union of Schubert cells $$G/B=\sqcup_{w \in W_0} B \dot w B/B.$$

In fact, the subgroups $B$ and $N$ of $G(\kk)$ form a $(B, N)$-pair, i.e., 
\begin{itemize}

\item $G$ is generated by the subgroups $B$ and $N$.

\item $B \cap N$ is a normal subgroup of $N$.

\item The group $W=N/B \cap N$ is generated by a set of elements $s$ of order $2$, for $s$ in some nonempty set $\BS_0$.

\item For any $s \in \BS_0$ and $w \in W$, $\dot s B w \subset B \dot s \dot w B \cup B \dot w B$. 

\item $\dot s \notin N_G(B)$ for all $s \in \BS_0$. 
\end{itemize}

An important consequence of the axioms of $(B, N)$-pairs is the following multiplication formula of Bruhat cells \[B \dot s B \dot w B=\begin{cases} B \dot s \dot w B, & \text{ if } s w>w; \\ B \dot s \dot w B \sqcup B \dot w B, & \text{ if } s w<w.\end{cases}\] \[B \dot w B \dot s B=\begin{cases} B \dot w \dot s B, & \text{ if } w s>w; \\ B \dot w \dot s B \sqcup B \dot w B, & \text{ if } w s<w.\end{cases}\]

\subsection{$G(\breve F)$ and its Iwahori-Weyl groups}

Let $\BF_q$ be a finite field with $q$ elements. Let $F$ be a non-archimedean local field with valuation ring $\CO_F$ and residue field $\BF_q$. We denote by $\e$ its uniformizer. Let $\breve F$ be the completion of the maximal unramified extension of $F$ with valuation ring $\CO_{\breve F}$ and residue field $\kk=\bar \BF_q$. 

Let $G$ be a connected reductive group over $F$.  Let $\s$ be the Frobenius automorphism of $\breve F/F$. We also denote the induced automorphism on $G(\breve F)$ by $\s$.

Let $S \subset G$ be a maximal $\breve F$-split torus defined over $F$ and let $T$ be its centralizer. Since $G$ is quasi-split over $\breve F$, $T$ is a maximal torus of $G$. Let $\breve \CI$ be the Iwahori subgroup fixing an alcove $\fka$ in the apartment $V$ attached to $S$. The {\it Iwahori-Weyl group} associated to $S$ is 
\[
\tW = N(\breve F)/T(\breve F) \cap \breve \CI,
\]
where $N$ denotes the normalizer of $S$ in $G$. The automorphism $\s$ on $G(\breve F)$ induces an automorphism on the Iwahori-Weyl group $\tW$, which we still denote by $\s$. For $w \in \tW$, we choose a representative $\dot w$ in $N(\breve F)$. 

The relative Weyl group is $W_0=N(\breve F)/T(\breve F)$. The Iwahori-Weyl group $\tW$ is a split extension of $W_0$ by the central subgroup $X_*(T)_\G$, where $\G=Gal(\bar F/\breve F)$. The splitting depends on the choice of a special vertex of $\fka$. 

In the split case, i.e., $S$ is a maximal torus of $G$, the Iwahori-Weyl group $\tW$ is the extended affine Weyl group associated to the root datum of $G$ and $\s: \tW \to \tW$ is the identity map. In the non-split case, the situation is more complicated.

\begin{example}
Let $G$ be the special orthogonal group of a quadratic form in $2n+1$ variables. The local Dynkin diagram is of type $B_n$ if $G$ is split, and is of type ${}^2 B_n$ if $G$ is non-split. The Iwahori-Weyl group of $G(\breve F)$ for split and non-split quadratic forms are the same. However, the group automorphism on the Iwahori-Weyl group induced from the Frobenius automorphism $\s$ are different: it is the identity map in the split case and the unique nontrivial diagram automorphism in the non-split case.

\begin{table}[!htb]
\caption{Local Dynkin diagrams for $SO(2n+1)$}
\label{table3}
\begin{tabular}{|c|c|}
\hline
Name & Local Dynkin diagram \\ \hline
$B_n$ & 
\begin{tikzpicture}
\begin{scope}[start chain]
\dnode{1}
\dnode{2}
\dnode{2}
\dydots
\dnode{2}
\dnodenj{2}
\end{scope}
\begin{scope}[start chain=br going above]
\chainin(chain-2);
\dnodebr{1}
\end{scope}
\path (chain-5) -- node{\(\Rightarrow\)} (chain-6);
\end{tikzpicture}
\\ \hline
\multirow{2}{*}{${}^2 B_n$} & 
\begin{tikzpicture}
\begin{scope}[start chain]
\dnode{1}
\dnode{2}
\dnode{2}
\dydots
\dnode{2}
\dnodenj{2}
\end{scope}
\begin{scope}[start chain=br going above]
\chainin(chain-2);
\dnodebr{1}
\end{scope}
\path (chain-5) -- node{\(\Rightarrow\)} (chain-6);
\end{tikzpicture}
\\ 
& \begin{tikzpicture}[start chain]
\dnodenj{1}
\dnodenj{2}
\dydots
\dnode{2}
\dnodenj{1}
\path (chain-1) -- node{\(\Leftarrow\)} (chain-2);
\path (chain-4) -- node{\(\Rightarrow\)} (chain-5);
\end{tikzpicture} \\ \hline
\end{tabular}
\end{table}
\end{example}

\begin{example}
Let $G$ be the special unitary group of a hermitian form in $2n$ variables over a ramified quadratic extension of $F$. If the form $h$ is split, then the local Dynkin diagram is of type $B$-$C_n$. If the form $h$ is non-split, then the local Dynkin diagram is of type ${}^2 B$-$C_n$. The Iwahori-Weyl group of $G(\breve F)$ for split and non-split Hermitian forms are the same. However, the group automorphism on the Iwahori-Weyl group induced from the Frobenius automorphism $\s$ are different: it is the identity map if $h$ is split and the unique nontrivial diagram automorphism if $h$ is not split. 

\begin{table}[!htb]
\caption{Local Dynkin diagrams for ramified $SU(2n)$}
\label{table4}
\begin{tabular}{|c|c|}
\hline
Name & Local Dynkin diagram \\ \hline
$B$-$C_n$ & 
\begin{tikzpicture}
\begin{scope}[start chain]
\dnode{1}
\dnode{2}
\dnode{2}
\dydots
\dnode{2}
\dnodenj{2}
\end{scope}
\begin{scope}[start chain=br going above]
\chainin(chain-2);
\dnodebr{1}
\end{scope}
\path (chain-5) -- node{{\(\Leftarrow\)}} (chain-6);
\end{tikzpicture}
\\ \hline
\multirow{2}{*}{${}^2 B$-$C_n$} & 
\begin{tikzpicture}
\begin{scope}[start chain]
\dnode{1}
\dnode{2}
\dnode{2}
\dydots
\dnode{2}
\dnodenj{2}
\end{scope}
\begin{scope}[start chain=br going above]
\chainin(chain-2);
\dnodebr{1}
\end{scope}
\path (chain-5) -- node{\color{red}{\(\Leftarrow\)}} (chain-6);
\end{tikzpicture}
\\ 
& \begin{tikzpicture}[start chain]
\dnodenj{1}
\dnodenj{2}
\dydots
\dnode{2}
\dnodenj{1}
\path (chain-1) -- node{\(\Leftarrow\)} (chain-2);
\path (chain-4) -- node{{\(\Leftarrow\)}} (chain-5);
\end{tikzpicture} \\ \hline
\end{tabular}
\end{table}


It is worth mentioning that although the local Dynkin diagram in Table \ref{table4} different from Table \ref{table3}. The only difference is the orientation. As the change of orientation does not change the associated affine Weyl groups, the Iwahori-Weyl group $\tW$ together with the automorphism $\s: \tW \to \tW$ of ramified special unitary group $SU(2n)$ of a split (resp. non-split) hermitian form are the same as of split (resp. non-split) special orthogonal group $SO(2n+1)$. 
\end{example}

\smallskip

In general, we may associate to $G \supset T$ a reduced root datum $\sR$ such that the Iwahori-Weyl group of $G_{sc}$ can be identified with the affine Weyl group $W_a$ associated to $\sR$. The set $\tilde \BS$ of simple reflections of $W_a$ consists of the reflections along the walls of the alcove $\fka$ we fixed in the beginning. The group $\tW$ acts on the set of alcoves in the apartment attached to $S$, where $W_a$ acts transitively on the set of alcoves. We also have $$\tW=W_a \rtimes \Omega,$$ where $\Omega$ is the isotropy subgroup of $\fka$. 

The group $X_*(T)_{\G}$ may have a nontrivial torsion part. However, as explained in \cite[\S 8.1]{HH}, the torsion part $X_*(T)_{\G, tor}$ lies in the center of $\tW$ and thus may be ignored when considering the (non-twisted and twisted) conjugation action. In particular, the results we discussed in Part I on the minimal length elements, straight conjugacy classes, etc, apply to the Iwahori-Weyl groups as well. 

The pair $(\tW/X_*(T)_{\G, tor}, \s)$ for any group $G$ comes from the pair $(\tW', \s')$ for a suitable unramified group. This simple observation allows us to reduce the study of some questions on $G$ to the questions on an unramified group. This observation will be used in \S \ref{nonempty} and \S \ref{dim}. 

\subsection{Iwahori-Bruhat decomposition on $G(\breve F)$ and some variations}\label{bruhat}

\subsubsection{Double cosets of the Iwahori subgroup} After the Bruhat decomposition was found in the 1950's, a rather unexpected extension was found by Iwahori and Matsumoto \cite{IM}, for a split $p$-adic groups. Here in the decomposition, one use the (compact) Iwahori subgroup instead of the (non-compact) Borel subgroup and the (infinite) Iwahori-Weyl group instead of the (finite) Weyl group. This decomposition is then extended by Bruhat and Tits \cite{BT72} to arbitrary $p$-adic groups. 

\begin{theorem}
The pair $(\breve \CI, N)$ is a $(B, N)$-pair. In particular,
$$G(\breve F)=\sqcup_{w \in \tW} \breve \CI \dot w \breve \CI.$$ 
\end{theorem}

We also have an explicit formula on the multiplication of Bruhat cells \[\breve \CI \dot s \breve \CI \dot w \breve \CI=\begin{cases} \breve \CI \dot s \dot w \breve \CI, & \text{ if } s w>w; \\ \breve \CI \dot s \dot w \breve \CI \sqcup \breve \CI \dot w \breve \CI, & \text{ if } s w<w.\end{cases}\] \[\breve \CI \dot w \breve \CI \dot s \breve \CI=\begin{cases} \breve \CI \dot w \dot s \breve \CI, & \text{ if } w s>w; \\ \breve \CI \dot w \dot s \breve \CI \sqcup \breve \CI \dot w \breve \CI, & \text{ if } w s<w.\end{cases}\]

\subsubsection{Double cosets of a parahoric subgroup} We have a similar decomposition for parahoric subgroups as well. Since all the Iwahori subgroups are conjugate, we will only consider the parabolic subgroups $\breve \CK$ that contains $\breve \CI$. For such $\breve \CK$, we denote by $K \subset \tilde \BS$ the corresponding set of simple reflections. The group $W_K$, generated by $K$, is the Weyl group of $\breve \CK$. We denote by ${}^K \tW^K \subset \tW$ the set of minimal length representatives in their $W_K$-double cosets. 

\begin{theorem}
Let $\breve \CK$ be a parahoric subgroup of $G(\breve F)$. Then $$G(\breve F)=\sqcup_{w \in {}^K \tW^K} \breve \CK \dot w \breve \CK.$$
\end{theorem}

Besides the Iwahori subgroup case, another case we are especially interested in is the special maximal parahoric subgroup case. Here we have a split extension of $\tW=X_*(T)_\G \rtimes W_0$ with respect to the vertex corresponding to $\breve \CK$ and $W_K \backslash \tW/W_K$ is in a natural bijection with the set of $W_0$-orbits of $X_*(T)_\G$. We have the Cartan decomposition $$G(\breve F)=\sqcup_{\{\mu\} \in X_*(T)_\G/W_0} \breve \CK \e^\mu \breve \CK.$$

Note that however, the multiplication formula for $\breve \CK$-double cosets, in general, is more complicated than the multiplication formula for $\breve \CI$-double cosets. 

\smallskip

\subsubsection{Lusztig's finer decomposition}
Let $\breve \CK$ be a parahoric subgroup of $G(\breve F)$ and $\breve \CI \subsetneqq \breve \CK$. For $w \in {}^K \tW^K$, we decompose $\breve \CK \dot w \breve \CK$ further into finitely many subsets stable under the action of $\breve K_\sigma$, analogous to the $G$-stable piece decomposition (for reductive groups $G$ over algebraically closed fields) introduced by Lusztig in \cite{Lu-par1}. 

We have the following simple but very useful properties on the conjugation actions on the Bruhat cells: 

\begin{itemize}

\item $\breve \CK \cdot_\s \breve \CI \dot w \breve \CI=\breve \CK \cdot_\s \breve \CI \dot w' \breve \CI$ if $w \approx_{K, \s} w'$; 

\item $\breve \CK \cdot_\s \breve \CI \dot w \breve \CI=\breve \CK \cdot_\s \breve \CI \dot s \dot w \breve \CI \cup \breve \CK \cdot_\s \breve \CI \dot s \dot w \s(\dot s) \breve \CI$ for $s \in K$ with $s w \s(s)<w$. 

\item If $w \in {}^K \tW$ and $u \in W_{I(K, w, \s)}$, then $\breve \CK \cdot_\s \breve \CI \dot x \dot w \breve \CI=\breve \CK \cdot_\s \breve \CI \dot w \breve \CI$. 
\end{itemize}

The following decomposition is essentially contained in \cite[1.4]{Lu-par3} and \cite[Proposition 2.5 \& 2.6]{He11}. The general case is in \cite[Theorem 3.2.1]{GH}. 

\begin{theorem}\label{kwk}
For any $w \in {}^K \tW^K$, $$\breve \CK \dot w \breve \CK=\sqcup_{x \in W_K w W_K \cap {}^K \tW} \breve \CK \cdot_\s \breve \CI \dot x \breve \CI.$$
\end{theorem}

Due to its importance, we give a sketch of the proof that 
$$\breve \CK \dot w \breve \CK \subset \cup_{x \in W_K w W_K \cap {}^K \tW} \breve \CK \cdot_\s \breve \CI \dot x \breve \CI.$$ Note that $\breve \CK \dot w \breve \CK=\sqcup_{w' \in W_K w W_K} \breve \CI \dot w' \breve \CI$. We argue by induction on $w'$ that \begin{equation}\label{iwi} \breve \CI \dot w' \breve \CI \subset \cup_{x \in W_K w W_K \cap {}^K \tW} \breve \CK \cdot_\s \breve \CI \dot x \breve \CI.\end{equation}

If $w'$ is a minimal length element in its $\s$-twisted $W_J$-conjugacy class, then by Theorem \ref{a-red}, $w' \approx_{K, \s} u x$ for some $x \in {}^K \tW$ and $u \in W_{I(K, x, \s)}$. In this case, $x \in W_J w' W_{\s(J)}=W_J w W_{\s(J)}$. The inclusion (\ref{iwi}) holds for $w'$. 

If $w'$ is not a minimal length element in its $\s$-twisted $W_J$-conjugacy class, then by Theorem \ref{a-min}, there exists $w'' \approx_{K, \s} w'$ and $s \in K$ with $s w'' \s(s)<w''$. Then $$\breve \CK \cdot_\s \breve \CI \dot w' \breve \CI=\breve \CK \cdot_\s \breve \CI \dot w'' \breve \CI=\breve \CK \cdot_\s \breve \CI \dot s \dot w'' \breve \CI \cup \breve \CK \cdot_\s \breve \CI \dot s \dot w'' \s(\dot s) \breve \CI.$$ The inclusion (\ref{iwi}) for $w'$ follows from inductive hypothesis on $s w''$ and on $s w'' \s(s)$. 

\subsection{The set $B(G)$ of $\s$-conjugacy classes on $G(\breve F)$}\label{BG}

For any $b \in G(\breve F)$, we denote by $[b]=\{g \i b \s(g); g \in G(\breve F)\}$ its $\s$-conjugacy class. Let $B(G)$ be the set of $\s$-conjugacy classes of $G(\breve F)$.\footnote{To distinguish the conjugacy classes in $G(\breve F)$ in $\tW$, we use \textbf{$\s$-conjugacy classes} for $G(\breve F)$ and \textbf{$\s$-twisted conjugacy classes } for $\tW$.} Kottwitz \cite{Ko85} and \cite{Ko97} has given a classification of the set $B(G)$. This classification generalizes the Dieudonn\'e-Manin classification of isocrystals by their Newton polygons. For the purpose of this paper, it is convenient to view $B(G)$ as $\tW \sslash \tW_\s$. Recall that $\Omega_\s$ is the set of $\s$-coinvariants on $\Omega$ and $\CC$ is the dominant Weyl chamber of the reduced root datum $\sR$ associated to $G$. 

\begin{theorem}
The inclusion $N(\breve F) \to G(\breve F)$ induces a map $\Psi: \tW/\tW_\s \to B(G)$. Moreover, 

(1) the map $\Psi$ is surjective. 

(2) We have the following commutative diagram \[\xymatrix{\tW/\tW_\s \ar[rr]^{\Psi} \ar[dr]_f & & B(G) \ar[ld]^-f \\ & \Omega_\s \times \CC &}.\]

(3) The map $f: B(G) \to \Omega_\s \times \CC, b \mapsto (\k(b), \nu_b)$ is injective. 
\end{theorem}

Here part (3) is Kottwitz's classification of $B(G)$. The two arithmetic invariants on the set $\tW/\tW_\s$ of $\s$-twisted conjugacy classes of $\tW$ we introduced in \S \ref{I-ari} are the restriction of the invariants of $\s$-conjugacy classes of $G(\breve F)$ to the $\s$-twisted conjugacy classes of $\tW$. The surjectivity of $\Psi$ is first proved in \cite{GHKR10} for split groups. The general case is obtained in \cite[\S 3]{He99}. Another proof is given in \cite{GHN}. 

\

Note that we may identify $B(G)$ with $\Im(f)$. On the other hand, there is a natural bijection between $\Im(f)$ and $\tW \sslash \tW_\s$. Therefore, we have the following commutative diagram

\[\xymatrix{\tW \sslash \tW_\s \ar@{<->}[rr]^-{1-1} \ar@{^{(}->}[dr]_f & & B(G) \ar@{^{(}->}[ld]^-f \\ & X_*(T)_\BQ^+ \times \pi_1(G)_{\G_F} &}.\]

For any $\s$-conjugacy class $[b]$ of $G(\breve F)$, we denote by $\CO_{[b]}$ (or just $\CO_b$) the straight $\s$-twisted conjugacy class of $\tW$ that corresponds to $[b]$. 

\subsection{Admissible subsets of $G(\breve F)$}

We have discussed several interesting subsets of $G(\breve F)$: $\breve \CK \dot w \breve \CK$ for $w \in {}^K \tW^K$, $\breve \CK \cdot_\s \breve \CI \dot x \breve \CI$ for $w \in {}^K \tW$ and $[b] \in B(G)$. In the next few subsections, we will investigate some geometric properties (dimension, irreducible components, closure relations, etc.) of these subsets. In order to do so, let us assume furthermore that $F=\BF((\e))$. In this case, the affine flag variety and its deeper level generalization have a natural scheme structure. It is more difficult to do it over $p$-adic fields. We refer to the work of Zhu \cite{Zhx} and the work of Bhatt and Scholze \cite{BS} in this direction.

For any $n \in \BN$, let $\breve \CI_n$ be the $n$-th congruence subgroup of $\breve \CI$, i..e, the $n$-th Moy-Prasad subgroup associated to the barycenter of the base alcove. As in \cite[\S 2.10]{GHKR06}, A subset $V$ of $G(\breve F)$ is called {\it admissible} if for any $w \in \tW$, there exists $n \in \BN$ such that $V \cap \breve \CI \dot w \breve \CI$ is stable under the right action of $\breve \CI_n$. This is equivalent to say that for any $w \in \tW$, there exists $n' \in \BN$ such that $V \cap \overline{\breve \CI \dot w \breve \CI}$ is stable under the right action of $\breve \CI_{n'}$. 

We say that an admissible subset $V$ of $G(\breve F)$ is {\it locally closed} if for any $w \in \tW$, the underline reduced sheme of $(V \cap \overline{\breve \CI \dot w \breve \CI})/\breve \CI_n$ is locally closed in $\overline{\breve \CI \dot w \breve \CI}/\breve \CI_n$ for $n \gg 0$. In this case, we define the closure of $V$ to be $$\overline{V}=\varinjlim_w V_w,$$ where $V_w$ be the inverse image under the projection $G(\breve F) \to G(\breve F)/\breve \CI_n$ of the closure of $(V \cap \overline{\breve \CI \dot w \breve \CI})/\breve \CI_n$ in $G(\breve F)/\breve \CI_n$. 

We say that $V$ is bounded if $V \cap \breve \CI \dot w \breve \CI=\emptyset$ for all but finitely many $w \in W$. 

Let $V$ be a bounded, locally closed admissible subset of $G(\breve F)$. By definition, there exists $n \in \BN$ such that $V$ is stable under the right action of $\breve \CI_n$. We say that $V_1$ is an irreducible component of $V$ if $V_1/\breve \CI_n$ is an irreducible component of $V/\breve \CI_n$ and we denote by $Irr(V)$ the set of irreducible components of $V$. We define $$\dim_{\breve \CI} V_1=\dim (V_1/\breve \CI_n)-\dim (\breve \CI/\breve \CI_n), \qquad \dim_{\breve \CI} V=\max_{V_1 \in Irr(V)} \dim_{\breve \CI} V_1.$$ We also denote by $Irr^{\max}(V)$ the set of irreducible components of $V$ of maximal possible dimension, i.e., of dimension equals $\dim_{\breve \CI}(V)$. 

By convenient, we set $\dim_{\breve \CI}(\emptyset)=-\infty$ and $Irr^{\max}(\emptyset)=\emptyset$. 

For any parahoric subgroup $\breve \CK$, we set $\dim_{\breve \CK}(V)=\dim_{\breve \CI}(V)-\dim_{\breve \CI}(\breve \CK)$. 

In the rest of this subsection, we give some examples of admissible subsets. We start with a trivial example. 

\begin{example}
Let $w \in \tW$. Then $\breve \CI \dot w \breve \CI$ is a bounded, locally closed admissible subset of $G(\breve F)$. The closure of $\breve \CI w \breve \CI$ is $\sqcup_{w' \le w} \breve \CI \dot w' \breve \CI$ and $\dim_{\breve \CI} \breve \CI \dot w \breve \CI=\ell(w)$. 
\end{example}

\smallskip

The following result is easy to prove and we skip the details. 

\begin{lemma}\label{w-n}
Let $w \in \tW$ and $n \ge \ell(w)$. Then for any $g \in \breve \CI w \breve \CI$, $g \breve \CI_n g \i \subset \breve \CI_{n-\ell(w)}$. 
\end{lemma}


\begin{example}
Let $\breve \CK$ be a parahoric subgroup of $G(\breve F)$. By Lemma \ref{w-n}, there exists $n \in \BN$ such that for any $g \in \breve \CK$, $g \breve \CI_n g \i \subset \breve \CI$. Hence for any $w \in \tW$, $\breve \CK \cdot_\s \breve \CI \dot w \breve \CI$ is stable under the right action of $\breve \CI_n$ and thus is an admissible subset of $G(\breve F)$.
\end{example}

\begin{example}
It is also known that each $\s$-conjugacy class of $G(\breve F)$ is admissible. See \cite[Theorem A.1]{He00}.
\end{example}

\subsection{Closure relations}\label{closure-ii}

Let $\breve \CK$ be a parahoric subgroup of $G(\breve F)$. We have the decomposition $$G(\breve F)=\sqcup_{w \in {}^K \tW} \breve \CK \cdot_\s \breve \CI \dot w \breve \CI.$$ Recall that the set ${}^K \tW$ is in natural bijection with $\tW\sslash (W_K)_\s$ and there is a well-defined partial order on $\tW\sslash (W_K)_\s \cong {}^K \tW$ (see \S \ref{preceq}). We show that this partial order describes the closure relations of the admissible subsets $\breve \CK \cdot_\s \breve \CI \dot w \breve \CI$. 

\begin{theorem}\label{1-clos}
Let $\breve \CK$ be a parahoric subgroup and $x \in {}^K \tW$. Then $$\overline{\breve \CK \cdot_\s \breve \CI \dot x \breve \CI}=\sqcup_{x' \in {}^K \tW, x' \le_{K, \s} x} \breve \CK \cdot_\s \breve \CI \dot x' \breve \CI.$$
\end{theorem}

It is proved in \cite[Proposition 2.5]{He11} in the case where $G$ is a split group and $\breve \CK$ is the hyperspecial maximal subgroup. The general case is proved in the same way. We sketch the proof for completeness. 

\begin{proof}
The base idea is to prove by induction. However, in order to do so, we do not only need the elements in ${}^K \tW$, but have to to use all the elements in $\tW$ as well. The more general statement we would like to prove is the following:

For any $w \in \tW$, $\overline{\breve \CK \cdot_\s \breve \CI \dot w \breve \CI}=\sqcup_{x' \in {}^K \tW, x' \le_{K, \s} w} \breve \CK \cdot_\s \breve \CI \dot x' \breve \CI.$ Here we say $x' \le_{K, \s} w$ if there exists $x_1 \approx_{K, \s} x'$ with $x_1 \le w$. 

Since $\breve \CK/\breve \CI$ is proper, $$\overline{\breve \CK \cdot_\s \breve \CI \dot w\breve \CI}=\breve \CK \cdot_\s \overline{\breve \CI \dot w \breve \CI}=\cup_{w' \in \tW, w'\le w} \breve \CK \cdot_\s \breve \CI \dot w' \breve \CI.$$

We argue by induction on $w'$ that $\breve \CI \dot w' \breve \CI \subset \sqcup_{x' \in {}^K \tW, x' \le_{K, \s} w} \breve \CK \cdot_\s \breve \CI \dot x' \breve \CI$ for any $w' \le w$. 

If $w'$ is a minimal length element in its $\s$-twisted $W_J$-conjugacy class, then by Theorem \ref{a-red}, $w' \approx_{K, \s} u x'$ for some $x' \in {}^K \tW$ and $u \in W_{I(K, x', \s)}$. So $\breve \CK \cdot_\s \breve \CI \dot w' \breve \CI=\breve \CK \cdot_\s \breve \CI \dot u \dot x' \breve \CI=\breve \CK \cdot_\s \breve \CI \dot x' \breve \CI$. Since $x' \le u x'$ and $w' \approx_{K, \s} u x'$, by Lemma \ref{par-1} there exists $x'' \approx_{K, \s} x'$ with $x'' \le w' \le w$. By definition, $x' \preceq_{K, \s} w$. 

If $w'$ is not a minimal length element in its $\s$-twisted $W_J$-conjugacy class, then there exists $w'' \approx_{K, \s} w'$ and $s \in K$ with $s w'' \s(s)<w''$. Then $$\breve \CK \cdot_\s \breve \CI \dot w' \breve \CI=\breve \CK \cdot_\s \breve \CI \dot w'' \breve \CI=\breve \CK \cdot_\s \breve \CI \dot s \dot w'' \breve \CI \cup \breve \CK \cdot_\s \breve \CI \dot s \dot w'' \s(\dot s) \breve \CI.$$ Note that $s w'' \s(s)<s w''<w''$. Thus for any $x' \in {}^K \tW$, if $x' \preceq_{K, \s} s w'' \s(s)$ or $x' \preceq_{K, \s} s w''$, then $x' \preceq_{K, \s} w''$. By inductive hypothesis on $s w''$ and on $s w'' \s(s)$, we have $$\breve \CK \cdot_\s \breve \CI \dot w' \breve \CI=\breve \CK \cdot_\s \breve \CI \dot s \dot w'' \breve \CI \cup \breve \CK \cdot_\s \breve \CI \dot s \dot w'' \s(\dot s) \breve \CI \subset \sqcup_{x' \in {}^K \tW, x' \preceq_{K, \s} w''} \breve \CK \cdot_\s \breve \CI \dot x' \breve \CI.$$

For any $x' \in {}^K \tW$ with $x' \preceq_{K, \s} w''$, by Lemma \ref{par-1}, there exists $x'' \approx_{K, \s} x'$ with $x'' \le w' \le w$. By definition, $x' \preceq_{K, \s} w$. 
\end{proof} 

Recall that there is a natural bijection between $\tW \sslash \tW_\s$ and $B(G)$. On $\tW\sslash \tW_\s$, there is a natural partial order $\preceq$ (see \S \ref{preceq}). It is proved in \cite[Theorem B]{He00} that this partial order describes the closure relations between the $\s$-conjugacy classes of $G(\breve F)$.  The idea of the proof is similar to the proof of Theorem \ref{1-clos}, but more technical. 

\begin{theorem}\label{tri-par}
The natural bijection $$\tW\sslash \tW_\s \to B(G), \qquad \CO_{[b]} \to [b]$$ is a bijection of partial order sets. Here the partial order on $\tW\sslash \tW_\s$ is the combinatorial order $\preceq$ and the partial order on $B(G)$ is given by the closure relation. 
\end{theorem}

\subsection{Mazur's inequality}\label{Mazur}

In the next few subsections, we study the intersection of the subsets of $G(\breve F)$ in \S \ref{bruhat} and \S \ref{BG}. 

The intersection of $\breve \CK \e^\mu \breve \CK \cap [b]$, for $G=GL_n$ and $\breve \CK$ the maximal hyperspecial subgroup, first appeared in Mazur's work \cite{Ma73}, in the study of the relation between the Hodge slope of a crystal and the Newton slope of the associated isocrystal. 

By definition, an isocrystal is a pair $(N, g)$, where $N$ is a finite dimensional vector space over $\breve F$ and $g: N \to N$ is a semi-linear bijection. A crystal $M$ of an isocrystal $(N, g)$ is a $g$-stable $\CO_{\breve F}$-lattice of $N$. 

By Dieudonn\'e-Manin theory, there is a natural bijection between the isomorphisms classes of isocrystals and the associated Newton slopes. On the other hand, for each crystal, one may associate a Hodge slope. 

\begin{example}
Let $N$ be a $3$-dimensional vector space with standard basis $e_1, e_2, e_3$ and $g: N \to N$ is defined by $g(e_1)=e_2, g(e_2)=e_3, g(e_3)=\e e_1$. Then all the eigenvalues of $g$ are the cubic roots of $\e$ and Newton slope $\nu(N, g)=(\frac{1}{3}, \frac{1}{3}, \frac{1}{3})$. 

Let $M=\oplus_{i=1}^3 \CO_{\breve F} e_i$ be a crystal of $(N, g)$. Then $$\text{coker} g \mid_M \cong \CO_{\breve F} e_1/\e \CO_{\breve F} e_1=\kk^1 e_1 \oplus \kk^0 e_2 \oplus \kk^0 e_3.$$ The Hodge slope $\mu(M)=(1, 0, 0)$.  

The Hodge polygon and the Newton polygon have the same end point and the Hodge polygon lies above the Newton polygon. In other words, $\mu(M) \ge \nu(N, g)$ with respect to the dominance order. 
\end{example}

\smallskip

In general, we have the following result. 

\begin{theorem}
(1) Let $(N, g)$ be an isocrystal and $M$ be a crystal of $(N, g)$. Then $$\mu(M) \ge \nu(N, g).$$ 

(2) Let $(N, g)$ be an isocrystal of dimension $n$. Let $\mu=(a_1, \cdots, a_n) \in \BZ^n$ with $a_1 \ge a_2 \ge \cdots \ge a_n$ and $\mu \ge \nu(N, g)$. Then there exists a crystal $M$ of $(N, g)$ with $\mu(M)=\mu$. 
\end{theorem}

Here part (1) is obtained by Mazur in \cite{Ma73} and part (2) is obtained by Kottwitz and Rapoport in \cite{KR}.

Now we reformulate the Mazur's inequality in a group-theoretic way. 

An isocrystal of dimension $n$ corresponds to a $\s$-conjugacy class of $GL_n(\breve F)$ and a crystal corresponds to a $\breve \CK$-double coset in $GL_n(\breve F)$, where $\breve \CK=GL_n(\CO_{\breve F})$. 
The above Theorem may be reformulated as follows:

Let $[b]$ be a $\s$-conjugacy class of $GL_n(\breve F)$ and $\mu$ be a dominant coweight of $GL_n(\breve F)$. Then $[b] \cap \breve \CK \e^\mu \breve \CK \neq \emptyset$ if and only if $\mu \ge \nu_b$ and $\k([b])=\k(\e^\mu)$. 

In the sequel, we will study the intersection $\breve \CK \dot w \breve \CK \cap [b]$ for any parahoric subgroup $\breve \CK$ in a reductive group $G(\breve F)$. Such intersection plays an important role in the study of reduction of Shimura varieties and we will discuss some applications in \S \ref{Shimura}. 

\subsection{The intersection of a Bruhat cell with a $\s$-conjugacy class} 

\subsubsection{Deligne-Lusztig reduction method} Now we discuss a reduction method to study the intersection $\breve \CI \dot w \breve \CI \cap [b]$ after Deligne and Lusztig \cite{DL}. 

We first prove the following results. 

\begin{proposition}
Let $X$ be an admissible subset of $G(\breve F)$. Then the inverse image of $X$ under the multiplication map $m: G(\breve F) \times G(\breve F) \to G(\breve F)$ is an admissible subset of $G(\breve F) \times G(\breve F)$.
\end{proposition}

\begin{proof}
Let $w, w' \in \tW$. Then $\breve \CI \dot w \breve \CI \dot w' \breve \CI$ is a bounded subset of $G(\breve F)$. There exists $n \in \BN$ such that $X \cap \breve \CI \dot w \breve \CI \dot w' \breve \CI$ is stable under the right action of $\breve \CI_n$. Suppose that $(g, g') \in \breve \CI \dot w \breve \CI \times \breve \CI \dot w' \breve \CI$ with $g g' \in X$. Then by Lemma \ref{w-n}, $$g \breve \CI_{n+\ell(w')} g' \breve \CI_{n+\ell(w')} \subset g g' \breve \CI_n \subset X \cap \breve \CI \dot w \breve \CI \dot w' \breve \CI.$$ Hence $m \i(X) \cap \breve \CI \dot w \breve \CI \times \breve \CI \dot w' \breve \CI$ is stable under the right action of $\breve \CI_{n+\ell(w')} \times \breve \CI_{n+\ell(w')}$ and $m \i(X)$ is admissible. 
\end{proof}

\smallskip 

Let $w \in \tW$ and $s \in \tilde \BS$ with $s w<w$. Then $\breve \CI \dot s \breve \CI \times \breve \CI \dot s \dot w \breve \CI$ is an admissible subset of $G(\breve F) \times G(\breve F)$ and the multiplication map $G(\breve F) \times G(\breve F) \to G(\breve F)$ induces $$m: \breve \CI \dot s \breve \CI \times \breve \CI \dot s \dot w \breve \CI \to \breve \CI \dot w \breve \CI.$$

Notice that we do not have the notion of isomorphisms between the admissible subsets. However, we may still discuss the dimensions, irreducible components, connected components, etc. And for any admissible subset $X$ of $\breve \CI w \breve \CI$, there is a natural bijection between the irreducible components of $X$ and of $m \i (X)$ and $\dim_{\breve \CI}$ of an irreducible component of $X$ equals $\dim_{\breve \CI \times \breve \CI}$ of its inverse image in $\breve \CI \dot s \breve \CI \times \breve \CI \dot s \dot w \breve \CI$. 

Let $p_{1 2}: G(\breve F) \times G(\breve F) \to G(\breve F)$ be the map defined by $(g_1, g_2) \mapsto (g_2, \s(g_1))$. We consider the following diagram

\begin{equation}\label{g-g} \xymatrix{\breve \CI \dot w \breve \CI \cap [b] \ar[d] & m \i(\breve \CI \dot w \breve \CI \cap [b]) \ar[l]_-m  \ar[r]^-{p_{12}} \ar[d] & Y \ar[r]^-m \ar[d]  & Z \ar[d] \\ 
\breve \CI \dot w \breve \CI & \breve \CI \dot s \breve \CI \times \breve \CI \dot s \dot w \breve \CI \ar[l]_-m \ar[r]^-{p_{12}} & \breve \CI \dot s \dot w \breve \CI \times \breve \CI \s(\dot s) \breve \CI \ar[r]^-m & \breve \CI \dot s \dot w \breve \CI \s(\dot s) \breve \CI
.}\end{equation}

Here $Y=p_{12} m \i (\breve \CI \dot w \breve \CI \cap [b])$ and $Z=m(Y)$. 

Let $(g_1, g_2) \in m \i (\breve \CI \dot w \breve \CI \cap [b])$. Then $g_1 g_2 \in [b]$. Hence $g_2 \s(g_1)=m \circ p_{1 2} (g_1, g_2) \in [b]$. Thus $Z \subset \breve \CI \dot s \dot w \breve \CI \s(\dot s) \breve \CI \cap [b]$. On the other hand, for any $g \in \breve \CI \dot s \dot w \breve \CI \s(\dot s) \breve \CI \cap [b]$, there exists $g_2 \in \breve \CI \dot s \dot w \breve \CI$ and $g_1 \in \breve \CI \dot s \breve \CI$ such that $g_2 \s(g_1)=g \in [b]$. So $g_1 g_2 \in [b]$ and $(g_1, g_2) \in m \i(\breve \CI \dot w \breve \CI \cap [b])$. Therefore $Z=\breve \CI \dot s \dot w \breve \CI \s(\dot s) \breve \CI \cap [b]$. 

\begin{theorem}\label{DL-red}
Let $w \in \tW$ and $s \in \tilde \BS$ with $s w<w$. The diagram \ref{g-g} induces a natural bijection $$Irr^{\max}(\breve \CI \dot w \breve \CI \cap [b]) \leftrightarrow Irr^{\max}(\breve \CI \dot s \dot w \breve \CI \s(\dot s) \breve \CI \cap [b]).$$ We also have that $$\dim (\breve \CI \dot w \breve \CI \cap [b])=\begin{cases} \dim (\breve \CI \dot s \dot w \breve \CI \s(\dot s) \breve \CI \cap [b]), & \text{ if } s w \s(s)>s w; \\ \dim (\breve \CI \dot s \dot w \breve \CI \s(\dot s) \breve \CI \cap [b])+1, & \text{ if } s w \s(s)<s w. \end{cases}$$
\end{theorem}

\begin{proof}
If $s w \s(s)>s w$, then $\breve \CI \dot s \dot w \breve \CI \s(\dot s) \breve \CI=\breve \CI \dot s \dot w \s(\dot s) \breve \CI$. The above diagram gives a natural bijection $C \leftrightarrow C'$ between the irreducible components of $\breve \CI \dot w \breve \CI \cap [b]$ and of $ \breve \CI \dot s \dot w \breve \CI \s(\dot s) \breve \CI \cap [b]$. Moreover, we have $\dim_{\breve \CI} C=\dim_{\breve \CI} C'$. 

If $s w \s(s)<s w$, then $\breve \CI \dot s \dot w \breve \CI \s(\dot s) \breve \CI=\breve \CI \dot s \dot w \breve \CI \sqcup \breve \CI \dot s \dot w \s(\dot s) \breve \CI$ and $Z=Z_1 \sqcup Z_2$, where $Z_1=\breve \CI \dot s \dot w \breve \CI \cap [b]$ and $Z_2=\breve \CI \dot s \dot w  \s(\dot s) \breve \CI \cap [b].$ Here $Z_1$ is open in $Z$ and $Z_2$ is its closed complement. It is worth mentioning that in general $Z_1$ may not be dense in $Z$ (see Example \ref{gh-example}). Also we do not have a nice correspondence between the irreducible components of $\breve \CI \dot w \breve \CI \cap [b]$ and of $Z$. 

However, we still have a natural bijection $C \leftrightarrow D$ between the irreducible components of $\breve \CI \dot w \breve \CI \cap [b]$ and of $Y$. Under this bijection $\dim_{\breve \CI} C=\dim_{\breve \CI} D$. 

On the other hand, $Y=Y_1 \sqcup Y_2$, where $Y_i=m_Y \i(Z_i)$. For $i=1, 2$, there is a natural bijection $D_i \leftrightarrow C'_i$ between the irreducible components of $Y_i$ and $Z_i$. Moreover, we have $\dim_{\breve \CI} D_i=\dim_{\breve \CI} C'_i+1$. Note that if $C'$ is an irreducible component of $Z_2$, then $C'$ may not be an irreducible component of $Z$. However, if $\dim C'=\dim Z$, then it is an irreducible component of $Z$. 

In particular, the above diagram gives a natural bijection $C \leftrightarrow C'$ between the irreducible components of maximal dimension of $\breve \CI \dot w \breve \CI \cap [b]$ and of $\breve \CI \dot s \dot w \breve \CI \s(\dot s) \breve \CI \cap [b]$. Moreover, we have $\dim_{\breve \CI} C=\dim_{\breve \CI} C'+1$. 
\end{proof}

Using Theorem \ref{DL-red}, one may reduce the study of $\dim_{\breve \CI}$ and $Irr^{\max}$ of $\breve \CI \dot w \breve \CI \cap [b]$ to the case where $w$ is of minimal length in its $\s$-twisted conjugacy class. The latter case is studied in \cite[Theorem 3.7]{He99}.

\begin{theorem}\label{min-dl}
Let $w \in \tW$ be an element of minimal length in its $\s$-twisted conjugacy class. Then $$\breve \CI \dot w \breve \CI \subset [\dot w].$$
\end{theorem}

In other words, if $w \in \tW$ be an element of minimal length in its $\s$-twisted conjugacy class, then $$\dim(\breve \CI \dot w \breve \CI \cap [b])=\begin{cases} \ell(w), & \text{ if } \dot w \in [b] \\ -\infty, & \text{ if } \dot w \notin [b] \end{cases} \text{ and } |Irr^{\max}(\breve \CI \dot w \breve \CI \cap [b])|=\begin{cases} 1, & \text{ if } \dot w \in [b] \\ 0, & \text{ if } \dot w \notin [b] \end{cases}.$$ Theorem \ref{DL-red}, together with Theorem \ref{min-dl}, provide a practical way to compute $\dim_{\breve \CI}$ and $Irr^{\max}$ of $\breve \CI \dot w \breve \CI \cap [b]$ for arbitrary $w \in \tW$ and $[b] \in B(G)$. Let us look at an example. 

\begin{example}\label{gh-example}
This example comes from \cite[\S 5]{GH10}. Let $G=SL_4$. The simple reflections in $\tW$ are $s_0, s_1, s_2, s_3$. We simply write $s_{a b c \cdots}$ instead of $s_a s_b s_c \cdots$. We take $w=s_{1201232101}$ and $b=1$. In Figure \ref{gh-example1}, we illustrate how the reduction method is used to study $Y:=\breve \CI \dot w \breve \CI \cap [b]$. Here we skip the step $w_1 \approx w_2$ and only record the step $w_1 \to s w_1 s$ with $s \in \BS_0$ such that $s w_1 s<w_1$. The resulting figure is a binary tree. Here the left branch is the element $s w_1$ (which corresponds to an open part of the intersection) and the right branch is the element $s w_1 s$ (which corresponds to the closed complement of the intersection). Each box contains an element $w_1 \in \tW$ and a number, which equals to $\dim_{\breve \CI} (\breve \CI \dot w_1 \breve \CI \cap [b])$. 

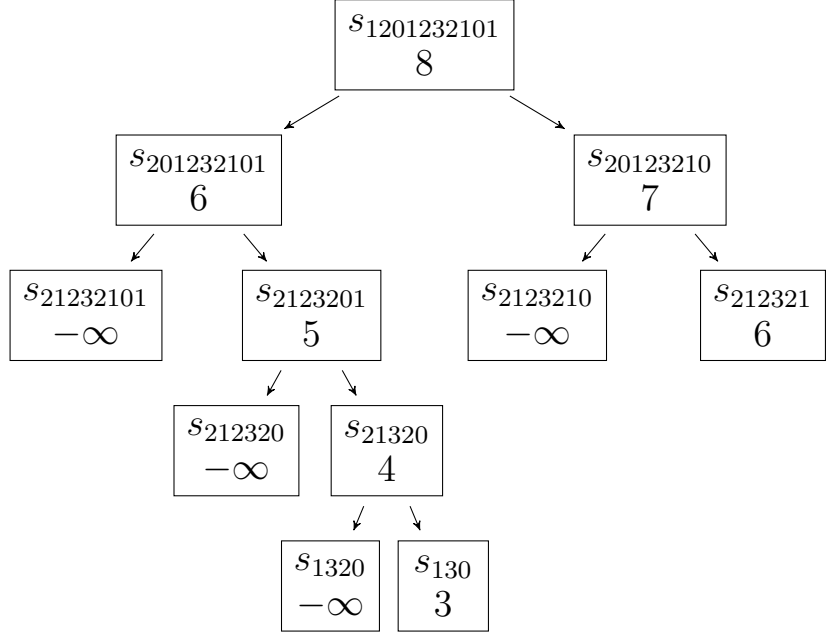
\begin{figure}
\centering

\caption{An example of Deligne-Lusztig reduction.}
{
    \begin{tikzpicture}[->,>=stealth', level/.style={sibling distance = 5cm/#1, level distance = 1.5cm}, scale=1.2,transform shape]
    \node [treenode] {$s_{1201232101}$ \\ 8}
    child
    {
        node [treenode] {$s_{201232101}$ \\ 6} 
        child
        {
            node [treenode] {$s_{21232101}$  \\ $-\infty$} 
        }
        child
       {
          node[treenode]{$s_{2123201}$ \\ 5}
          child
            {
                node [treenode] {$s_{212320}$ \\ $-\infty$ } 
            }
            child
            {
                node [treenode] {$s_{21320}$ \\ 4} 
                child
            {
                node [treenode] {$s_{1320}$ \\ $-\infty$} 
            }
            child
            {
                node [treenode] {$s_{130}$ \\ $3$ } 
            }
            }
       }
      }  
     child
     {
         node[treenode]{$s_{20123210}$  \\ 7}
         child
            {
                node [treenode] {$s_{2123210}$ \\ $-\infty$} 
            }
          child
            {
                node [treenode] {$s_{212321}$ \\ $6$} 
            }
     }
;
\end{tikzpicture}
}
\label{gh-example1}
\end{figure}

From the figure, we see that $Y=Y_1 \sqcup Y_2$, where $Y_1$ is open in $Y$ and $Y_2$ is its closed complement. Moreover, $\dim Y_1=7$ and $\dim Y_2=8$ and they are both irreducible. Thus both $Y_2$ and $\overline{Y_1}$ are the irreducible components of $Y$. In other words, $Y$ has two irreducible components, one is of dimension $8$ and the other one is of dimension $7$. This is an example where the intersection is not equidimensional. 
\end{example}

\subsubsection{``Dimension=degree'' theorem}\label{dim-deg}
In general, such computation is quite difficult to run if the rank of the group is not very small. A better way to do this is to use the class polynomials of affine Hecke algebras, which encode the information of the reduction step.  

We will give a systematic discussion of class polynomials in \S \ref{3.6.3}. In this section, we just give an algorithmic definition of the class polynomials. 

To any $w \in \tW$ and a $\s$-twisted conjugacy class $\CO$ of $\tW$, we associate the class polynomial $F_{w, \CO} \in \BZ[q]$. It is defined inductively on $w$.\footnote{The definition here differs from the definition of class polynomials $f_{w, \CO}$ defined in \cite[\S 2]{He99} as the normalization of Hecke algebras we use here is different from \cite{He99}. The relation is $F_{w, \CO}=v^{\ell(w)-\ell(\CO)} f_{w, \CO} \mid_{q=v^2}$.}

If $w$ is of minimal length in its $\s$-twisted conjugacy class, then $$F_{w, \CO}=\begin{cases} 1, & \text{ if } w \in \CO; \\ 0, & \text{ otherwise}. \end{cases}$$

If $w$ is not of minimal length in its $\s$-twisted conjugacy class, then by Theorem \ref{a-red}, there exists $w' \in \tW$ and $s \in \tilde \BS$ such that $w \approx_\s w'$ and $s w' \s(s)<w'$. We set $$F_{w, \CO}=(q-1) F_{s w', \CO}+q F_{s w' \s(s), \CO}.$$

In particular, if we regard $F_{w, \CO}$ as a polynomial of $q-1$, then all the coefficients are nonnegative integers.

By convention, we define the degree of the polynomial $0$ to be $-\infty$ and the leading coefficient of the polynomial $0$ to be $1$. 

\begin{theorem}\label{deg1}
Let $w \in \tW$ and $[b] \in B(G)$. Let $F_{w, [b]}=\sum_{\CO \in \tW/\tW_\s, f(\CO)=f([b])} q^{\ell(\CO)} F_{w, \CO}$. Then 

(1) $\dim_{\breve \CI} (\breve \CI \dot w \breve \CI \cap [b])=\deg F_{w, [b]}$. 

(2) The cardinality of $Irr^{\max}(\breve \CI \dot w \breve \CI \cap [b])$ equals the leading coefficient of $F_{w, [b]}$. 
\end{theorem}

The proof is similar to the proof of \cite[Theorem 6.1]{He99} and is proved by induction on $w$. 

As a consequence, we have

\begin{corollary}
Let $w \in \tW$ and $[b] \in B(G)$. If $\ell(w)< \<2 \rho, \nu_b\>$, then $\breve \CI \dot w \breve \CI \cap [b]=\emptyset$. 
\end{corollary}

The reason is that if the reduction method always decreases the length of the elements involved. In particular, if $\ell(w)< \<2 \rho, \nu_b\>$, then after apply the reduction, one can only get $\s$-twisted conjugacy classes of $\tW$ different from $\CO_{[b]}$. 

\smallskip

We have similar ``dimension=degree'' theorem for parahoric subgroups. 

\begin{theorem}\label{deg2}
Let $\breve \CK$ be a $\s$-stable parahoric subgroup of $G(\breve F)$ and $w \in {}^K \tW^K$. Set $F_{w, K, [b]}=\sum_{w' \in W_K w W_K} F_{x, [b]}$. Then 

(1) $\dim_{\breve \CK} (\breve \CK \dot w \breve \CK \cap [b])=\deg F_{w, K, [b]}-\dim_{\breve \CI} \breve \CK$. 

(2) The cardinality of $Irr^{\max}(\breve \CK \dot w \breve \CK \cap [b])$ equals the leading coefficient of $F_{w, K, [b]}$.  
\end{theorem}

\begin{theorem}\label{2.22}
Let $\breve \CK$ be a $\s$-stable parahoric subgroup of $G(\breve F)$ and $x \in {}^K \tW$. Then 

(1) $\dim_{\breve \CK} (\breve \CK \cdot_\s \breve \CI \dot x \breve \CI \cap [b])=\deg F_{x, [b]}$. 

(2) The cardinality of $Irr^{\max}(\breve \CK \cdot_\s \breve \CI \dot x \breve \CI \cap [b])$ equals the leading coefficient of $F_{x, [b]}$.  
\end{theorem}

\smallskip

Note that $F_{w, K, [b]}$ only depends on the Iwahori-Weyl group $\tW$ together with $\s: \tW \to \tW$, the element $w \in \tW$, the subset $K$ of $\tilde \BS$ and straight $\s$-twisted conjugacy classes $\CO_{[b]}$ of $\tW$. Thus one may relate the questions on the emptiness/non-emptiness pattern, dimension, the number of irreducible components of maximal dimension, etc., of $\breve \CI \dot w \breve \CI \cap [b]$ in a ramified group to the questions in an unramified group. This allows us to translate some results in the unramified groups (proved by other methods) to results in the ramified group (where the methods in the unramified groups might not apply). 

\subsubsection{Relation with affine Deligne-Lusztig varieties} Now we discuss some relations between the intersection $\breve \CK \dot w \breve  \CK \cap [b]$ and the affine Deligne-Lusztig varieties. 

Let $\breve \CK$ be a $\s$-stable parahoric subgroup of $G(\breve F)$ and $w \in {}^K \tW^K$. Let $b \in G(\breve F)$. The affine Deligne-Lusztig variety is introduced by Rapoport in \cite{R:guide} $$X_{K, w}(b)=\{g \breve \CK \in G(\breve F)/\breve \CK; g \i b \s(g) \in \breve \CK \dot w \breve \CK\}.$$

If $\breve \CK$ is the Iwahori subgroup $\breve \CI$, we simply write $X_w(b)$ instead of $X_{\breve \CI, w}(b)$. If $\breve \CK$ is a special maximal parahoric subgroup, then we may write $$X_{K, \mu}(b)=\{g \breve \CK \in G(\breve F)/\breve \CK; g \i b \s(g) \in \breve \CK \e^\mu \breve \CK\}$$ for the corresponding affine Deligne-Lusztig variety. We have 

\begin{theorem}\label{ii-adlv}
Let $\breve \CK$ be a $\s$-stable parahoric subgroup of $G(\breve F)$ and $w \in {}^K \tW^K$. Let $b \in G(\breve F)$. Then $$\dim_{\breve \CK}(\breve \CK \dot w \breve \CK \cap [b])-\dim X_{\breve \CK, w}(b)=\<2 \rho, \nu_b\>.$$
\end{theorem}

\begin{proof}
We first consider the case where $\breve \CK=\breve \CI$. In this case, the Deligne-Lusztig reduction method still works for $\dim X_w(b)$ (see \cite[Proof of Theorem 6.1]{He99}). The reduction steps are the same. The only difference lies in the basic step. When $w$ is a minimal length element in its $\s$-twisted conjugacy class, then $$\dim X_w(b)=\begin{cases} \ell(w)-\ell(\CO_b), & \text{ if } \dot w \in [b]; \\-\infty, & \text{ if } \dot w \notin [b].\end{cases}$$ Thus $\dim X_w(b)$ equals the degree of $\sum_{\CO' \in \tW/\tW_\s, f(\CO')=f(\CO_b)} q^{\ell(\CO')-\ell(\CO_b)} F_{w, \CO'}$ and $$\dim_{\breve \CI}(\breve \CI \dot w \breve \CI \cap [b])-\dim X_w(b)=\ell(\CO_b)=\<2 \rho, \nu_b\>.$$

In the general case, $X_{\breve \CK, w}(b)$ equals the image of $\sqcup_{w' \in W_K w W_K}X_{w'}(b)$ under the natural projection map $\pi: G(\breve F)/\breve I \to G(\breve F)/\breve \CK$. Note that each fiber of $$\sqcup_{w' \in W_K w W_K}X_{w'}(b) \to X_{\breve \CK, w}(b)$$ is isomorphic to $\breve \CK/\breve \CI$. Thus 
\begin{align*}
\dim X_{\breve \CK, w}(b) &=\max_{w' \in W_K w W_K} \dim X_{w'}(b)-\dim \breve \CK/\breve \CI \\ &=\max_{w' \in W_K w W_K}  \dim_{\breve \CI} (\breve \CI \dot w'\breve \CI \cap [b])-\dim \breve \CK/\breve \CI-\<2 \rho, \nu_b\> \\ &=\deg F_{K, w, [b]}-\dim \breve \CK/\breve \CI-\<2 \rho, \nu_b\> \\ &=\dim_{\breve \CK}(\breve \CK \dot w \breve \CK \cap [b])-\<2 \rho, \nu_b\>.\qedhere
\end{align*}
\end{proof}

\subsection{Non-emptiness pattern}\label{nonempty} 

In this subsection, we discuss several cases that explicit non-emptiness patterns are known. 

\subsubsection{Special maximal parahoric} We first consider the intersection $\breve \CK \e^\mu \breve \CK \cap [b]$, where $\breve \CK$ is a special maximal parahoric subgroup. If $G=GL_n$, then the intersection is nonempty if and only if the Mazur's inequality is satisfied (see \S \ref{Mazur}). Kottwitz and Rapoport \cite{KR} conjecture that it is still the case for any reductive group $G$ (and special maximal parahoric subgroups $\breve \CK$). 

\begin{theorem}\label{2-24}
Let $\breve \CK$ be a special maximal parahoric subgroup of $G(\breve F)$. Let $\mu$ be a dominant coweight and $[b] \in B(G)$. Then $[b] \cap \breve \CK \e^\mu \breve \CK \neq \emptyset$ if and only if $\mu \ge \nu_b$ and $\k([b])=\k(\e^\mu)$. 
\end{theorem}

The ``only if'' side is proved by Rapoport and Richartz in \cite{RR}. The ``if'' side is proved for type A and C in \cite{KR}. It is then proved by Lucarelli \cite{Luc} for classical split groups and then by Gashi \cite{Ga} for unramified cases. The general case is obtained in \cite[Theorem 7.1]{He99}, combining the relation between the class polynomials and the non-emptiness pattern of such intersection, the relation between Iwahori-Weyl groups for ramified and unramified groups, and Gashi's result in the unramified case. 

\subsubsection{Reduction to quasi-split, adjoint groups}\label{red-qs}
Now we consider the intersection $\breve \CI \dot w \breve \CI \cap [b]$. We first discuss the reduction from arbitrary group to a quasi-split, adjoint group. 

As explained in \cite[\S 2.2]{GHN}, it suffices to consider the case where $G$ is adjoint. Let $H$ be the quasi-split inner form of $G$. Then $G(\breve F)=H(\breve F)$ and $\breve I$ is the Iwahori subgroup of $H(\breve F)$ as well. We identify the Iwahori-Weyl group of $H$ with $\tW$. The difference is the group automorphisms on $\tW$ induced from the Frobenius morphisms. We have $\s_G=\Ad(\g) \circ \s_H \in \text{Aut}(\tW)$ for some $\g \in \Omega$. The map $g \mapsto g \g$ gives a bijection from the $\s_G$-conjugacy class $[g]_G$ to the $\s_H$-conjugacy class $[g \g]_H$. We have $$\dim_{\breve \CI} (\breve \CI \dot w \breve \CI \cap [g]_G)=\dim_{\breve \CI} (\breve \CI \dot w \dot \g \breve \CI \cap [g \g]_H).$$

In the sequel, we denote by $\varsigma$ the group automorphism $\s_H$ on $\tW$. 

\subsubsection{$P$-alcove elements}
For quasi-split groups, an important notion is the $P$-alcove element, first introduced by G\"ortz, Haines, Kottwitz, and Reuman \cite{GHKR10} for split group and then generalized in \cite{GHN} to quasi-split groups. We recall the definition. 

Assume that $G$ is quasi-split. Let $P=M N$ be a standard $\s$-stable parabolic subgroup of $G$ and $x \in W_0$. We say that $w \in \tW$ is a $(P, x)$-alcove element if 
\begin{itemize}
\item $x \i w \s(x) \in \tW_M$;

\item $\dot x N(\breve F) \dot x \i \cap \dot w \breve \CI \dot w \i\subset \breve \CI$. 
\end{itemize}

The reason to put the above two requirement on $w$ is that we would like to reduce the study of $\breve \CI \dot w \breve \CI$ to the study of a Iwahori double coset in a Levi subgroup of $G$. In particular, we need that $w$ is $\s$-conjugate to an element in the Iwahori-Weyl group of a standard Levi subgroup. This is where the first condition comes from. The second condition is more involved. Roughly speaking, we have the decomposition $$\breve \CI=(\breve \CI \cap N(\breve F)) (\breve \CI \cap M(\breve F)) (\breve \CI \cap N^-(\breve F)),$$ where $N^-$ is the unipotent radical of the opposite parabolic subgroup. And we need this decomposition to be ``compatible'' with $\breve \CI \dot w \breve \CI$ so that we are able to eventually get rid of $\breve \CI \cap N(\breve F)$ after $\s$-conjugation by a suitable element. 

The upshot is the following result, first obtained for split groups by G\"ortz, Haines, Kottwitz, and Reuman \cite{GHKR10} and then generalized to quasi-split groups in \cite{GHN}. 

\begin{theorem}\label{P-al}
Let $G$ be a quasi-split group. Let $w \in \tW$ be a $(P, x)$-alcove element. Then any element in $\breve \CI \dot w \breve \CI$ is $\s$-conjugate by $\breve \CI$ to an element in $$(\breve \CI \cap \dot x M(\breve F) \dot x \i) w \s (\breve \CI \cap \dot x M(\breve F) \dot x \i) \subset \dot x M(\breve F) \s(\dot x \i).$$ 
\end{theorem}

\subsubsection{Basic $\s$-conjugacy classes}
We say a $\s$-conjugacy class $[b]$ of $G(\breve F)$ is {\it basic} if the associated Newton point is central, i.e., $\<2 \rho, \nu_b\>=0$. Use the notion of $P$-alcove elements, we have a complete description of the non-emptiness pattern of $\breve \CI \dot w \breve \CI \cap [b]$ for basic $\s$-conjugacy class $[b]$. 

In \cite{GHN}, we introduced the ``no Levi obstruction'' (NLO) condition. We say that there is {\it no Levi obstruction} if for any pair $(P, x)$ such that $w$ is a $(P, x)$-alcove element, there exists an element $b_J \in [b] \cap M(\breve F)$ such that $\k_M (x \i w \s(x))=\k_M(b_J)$. By Theorem \ref{P-al}, ``no Levi obstruction'' is a necessary condition for $\breve \CI \dot w \breve \CI \cap [b] \neq \emptyset$. In \cite{GHN}, it is proved that the notion of $P$-alcove elements is compatible with the Deligne-Lusztig reduction, and as a consequence, ``no Levi obstruction'' is also a sufficient condition. The following result is \cite[Theorem A]{GHN}. 

\begin{theorem}
Let $G$ be a quasi-split group. Let $w \in \tW$ and $[b] \in B(G)$ be a basic $\s$-conjugacy class. Then $\breve \CI \dot w \breve \CI \cap [b] \neq \emptyset$ if and only if there is no Levi obstruction. 
\end{theorem}

Let us make the ``no Levi obstruction'' condition more explicit.  

A {\it critical strip} of the apartment $V$ by definition is the subset $\{v; -1<\<\a, v\><0\}$ for a given positive finite root in the reduced root datum $\sR$. We remove all the critical strips from $V$ and call each connected component of the remaining subset of $V$ {\it a Shrunken Weyl chamber}. 

Let $w \in \tW$. Let $x$ by the unique element in $W_0$ such that $w (\fka) \subset x \CC^-$. Let $\eta_\s(w)$ be the unique element in $W_0$ such that $x \i w \s(x) \s(p) \in \eta_\s(w) \CC$ for any sufficiently regular element $p \in \CC^-$. In the quasi-split case, $\s$ preserves the antidominant chamber $\CC^-$. The element $w$ can be written as $x t^\l y$ with $x \in W_0$ and $t^\l y(\fka) \subset \CC^-$, then $x \i w \s(x) \s(\CC^-)=t^\l y \s(x) \CC^-$. Thus $\eta_\s(w)=y \s(x)$. In general, we have $\s=\Ad(\g) \circ \varsigma$ for some $\g \in \Omega$ and $x \i w \s(x) \s(p)=x \i w \g \varsigma(x) \g \i \g(\varsigma(p))=x \i w \g \varsigma(x) \varsigma(p).$ Therefore \begin{equation}\label{eta-eta} \eta_\s(w)=\eta_{\varsigma}(w \g).\end{equation} We have the following explicit description of non-emptiness pattern, first conjectured by G\"ortz, Haines, Kottwitz and Reuman in \cite[Conjecture 1.1.1]{GHKR10} for split groups. The quasi-split case is proved in \cite[Theorem B]{GHN} and the general case follows from \S\ref{red-qs}, the equality (\ref{eta-eta}) and the result for the quasi-split case. 

\begin{theorem}\label{basic-emp}
Assume that $G$ is simple. Let $w \in \tW$such that $w \s(\fka)$ is contained in a Shrunken Weyl chamber. Then for any basic $\s$-conjugacy class $[b]$, $\breve \CI \dot w \breve \CI \cap [b] \neq \emptyset$ if and only if $\k(w)=\k(b)$ and that $\eta_\s(w)$ is not contained in the relative Weyl group $W_M$ for any proper standard $\varsigma$-stable Levi subgroup $M$ of quasi-split inner form of $G$.  
\end{theorem}

\subsubsection{Translation elements}
We have similar results for some nonbasic $\s$-conjugacy classes as well. 

Let $G$ be a residually split group and $\mu$ be a dominant coweight. For any $x \in W_0$, we denote by $\CC^\sharp_x$ the Shrunken Weyl chamber inside $x(\CC^-)$. We set $$\CC^{\doublesharp}_{\mu, x}=t^{x(\mu)} (\CC^\sharp_x)=\CC^\sharp_x-x(\mu) \subset x(\CC^-)$$ and call it a {\it very Shrunken Weyl chamber} with respect to $\mu$. Then 

\begin{theorem}\label{nonbasic-emp}
Assume that $G$ is simple. Let $G$ be a residually split group and $\mu$ be a dominant coweight. Let $w \in \tW$ such that $w(\fka)$ is contained in a very Shrunken Weyl chamber with respect to $\mu$. If $\breve \CI \dot w \breve \CI \cap [\e^\mu] \neq \emptyset$, then $k(w)=k(\e^\mu)$ and $\eta_\s(w)$ is not contained in the relative Weyl group $W_M$ for any proper standard Levi subgroup $M$ of $G$.  
\end{theorem}

Here the proof is similar to the proof of the ``only if'' part of Theorem \ref{basic-emp}. We skip the details. We will also see in the next subsection that the converse is true and we have an explicit dimension formula. 

\subsection{Dimension formula}\label{dim}

In this subsection, we discuss several cases that an explicit dimension formula is known. 

\subsubsection{Defect} Let us first recall the notion of {\it defect} introduced by Kottwitz in \cite{Ko06}. Let $b \in G(\breve F)$. By definition, the defect $\text{def}(b)$ of $b$ is the difference between the $F$-rank of $G$ and the $F$-rank of $J_b$, where $J_b=\{g \in G(\breve F); g \i b \s(g)=b\}$ is the $\s$-centralizer of $b$. It can be reformulated in another way. 

Recall that $V$ is the apartment associated to the maximal $\breve F$-split torus $S$ defined over $F$. Then $\dim V$ equals the $F$-rank of $G$. Let $w \in \CO_b$ and $V_w=\{v \in V; w \s(v)=v+\nu_w\}$. By \cite[\S 1.9]{Ko06}, $\dim V_w$ equals the $F$-rank of $J_b$. Thus $$\text{def}(b)=\dim V-\dim V_w.$$ Note that the right hand side only depend on the Iwahori-Weyl group $\tW$, the group automorphism $\s: \tW \to \tW$ and a straight $\s$-twisted conjugacy class $\CO_b$. This formula will be used in our reduction from ramified groups to unramified groups. 

\subsubsection{Special maximal parahoric} We start with the special maximal parahoric subgroup case. 

\begin{theorem}\label{kmuk}
Let $\breve \CK$ be a special maximal parahoric subgroup of $G(\breve F)$.  Let $\mu$ be a dominant coweight and $[b] \in B(G)$. If $\nu_b \le \mu$, then \begin{gather*} \dim_{\breve \CK} (\breve \CK \e^\mu \breve \CK \cap [b])=\<\rho, \mu+\nu_b\>-\frac{1}{2}\text{def}_G(b), \\ \dim X_{K, \mu}(b)=\<\rho, \mu-\nu_b\>-\frac{1}{2}\text{def}_G(b). \end{gather*}
\end{theorem}

The dimension formula of $X_{K, \mu}(b)$ is conjectured by Rapoport in \cite[Conjecture 5.10]{R:guide}. For split groups, the conjectural formula is obtained by G\"ortz, Haines, Kottwitz and Reuman \cite{GHKR06} and Viehmann \cite{Vi06}. The conjectural formula for quasi-split unramified groups is obtained independently by Zhu \cite{Zhx} and Hamacher \cite{Ha1}. The general case can be obtained by combining Theorem \ref{deg2} (which reduce the study of dimension formula to the study of degree of class polynomials), the relation between Iwahori-Weyl groups for ramified and unramified groups, and the results of Zhu and Hamacher in the unramified case. 

\subsubsection{Virtual dimension} Now we consider the case where $\breve \CK=\breve \CI$ is the Iwahori subgroup and we give an upper bound of the dimension. 

\begin{theorem}\label{vir}
For any $w \in \tW$ and $[b] \in B(G)$, define the {\it virtual dimension} $$d_w([b])=\frac{1}{2} (\ell(w)+\ell(\eta_\s(w))-\text{def}_G(b))+\<\rho, \nu_b\>.$$ Then $$\dim_{\breve \CI} (\breve \CI \dot w \breve \CI \cap [b]) \le d_w([b]).$$
\end{theorem}

\begin{proof}
As explained in \S \ref{red-qs}, it suffices to consider the case where $G$ is a inner form of a quasi-split, adjoint group $H$. We have $\s_G=\Ad(\g) \circ \s_H \in \text{Aut}(\tW)$ for some $\g \in \Omega$. We have $\dim_{\breve \CI} (\breve \CI \dot w \breve \CI \cap [b]_G)=\dim_{\breve \CI} (\breve \CI \dot w \g \breve \CI \cap [b \g]_H)$. We also have $d_w([b]_G)=d_{w \g}([b \g]_H)$ since $\eta_{\s_G}(w)=\eta_{\s_H}(w \g)$.  

Suppose that $w \g \in W_0 t^\mu W_0$, where $\mu$ is a dominant coweight. By Theorem \ref{ii-adlv}, \cite[Theorem 10.3]{He99}, and Theorem \ref{kmuk}, we have
\begin{align*}
& \dim_{\breve \CI} (\breve \CI \dot w \g \breve \CI \cap [b \g]_H)=\dim X^H_{w \g}(b \g)+\<2 \rho, \nu_{b \g}^H\> \\ & \le \dim X^H_\mu(b \g)+\frac{1}{2}(\ell(w \g)+\ell(\eta_{\s_H}(w)))-\<\rho, \mu\>+\<2\rho, \nu_{b \g}^H\> \\ &=\<\rho, \mu-\nu_{b \g}^H\>-\frac{1}{2}\text{def}_H(b \g)+\frac{1}{2}(\ell(w \g)+\ell(\eta_{\s_H}(w)))-\<\rho, \mu\>+\<2\rho, \nu_{b \g}^H\> \\ &=d_{w \g}([b \g]_H).\qedhere
\end{align*}
\end{proof}

In several cases, we are able to get an equality $\dim X_w(b)=d_w([b])$. 

\subsubsection{Basic $\s$-conjugacy classes} Let $[b]$ be a basic $\s$-conjugacy class. Recall that in Theorem \ref{basic-emp}, we give a criterion on $\breve \CI \dot w \breve \CI \cap [b] \neq \emptyset$ for $w \in \tW$ such that $w \s(\fka)$ is contained in the Shrunken Weyl chamber. Combining Theorem \ref{ii-adlv}, \cite[Theorem 12.1]{He99} and Theorem \ref{vir}, we have

\begin{theorem}\label{bb-dim}
Assume that $G$ is simple. Let $w \in \tW$ such that $w \s(\fka)$ is contained in the Shrunken Weyl chamber and that $(\eta_\s(w))$ is not contained in the relative Weyl group $W_M$ for any proper standard $\varsigma$-stable Levi subgroup $M$ of $G$. Then for any basic $\s$-conjugacy class of $G(\breve F)$ with $\k([b])=\k(w)$, we have $$\dim_{\breve \CI} (\breve \CI \dot w \breve \CI \cap [b])=d_w([b]).$$
\end{theorem}

The dimension formula for the affine Deligne-Lusztig varieties $X_w(b)$ with $w$ and $b$ as above is conjectured by G\"ortz, Haines, Kottwitz, and Reuman in \cite[Conjecture 1.1.3]{GHKR10} and is previously verified in \cite[\S 12]{He99} for residually split group. 

\subsubsection{Translation elements} Now we consider the translation elements. We assume furthermore that $G$ is residually split. We first relate the intersection $\breve \CI \dot w \breve \CI \cap [\e^\mu]$ with $\breve \CI \dot w' \breve \CI \cap [1]$ for some $w'$. 

Let $H$ be a closed subgroup of $G$. A bounded subset of $V$ of $H(\breve F)$ is called admissible if it is stable under the right action of $H(\breve F) \cap \breve \CI_n$ for some $n \in \BN$. In this case, we say that $V$ is locally closed in $H(\breve F)$ if $V/(H(\breve F) \cap \breve \CI_n)$ is locally closed in $H(\breve F)/(H(\breve F) \cap \breve \CI_n)$. We define $$\dim_{H(\breve F) \cap \breve \CI} V=\dim V/(H(\breve F) \cap \breve \CI_n)-\dim (H(\breve F) \cap \breve \CI)/(H(\breve F) \cap \breve \CI_n).$$

We also need the following result of G\"ortz, Haines, Kottwitz and Reuman \cite[Theorem 6.3.1]{GHKR06}.

\begin{theorem}\label{superset}
Let $G$ be a split group. For any $w \in \tW$ and $\mu \in X_*(T)$, there is an equality $$\dim X_w(\e^\mu)=\max_{x \in W_0} \dim_{\dot x U(\breve F) \dot x\i \cap \breve \CI} (\breve \CI \dot w \breve \CI \e^{-x(\mu)} \cap \dot x U(\breve F) \dot x \i)+\<\rho, \mu-\bar \mu\>,$$ where $U$ is the subgroup of $G$ generated by the (finite) positive root subgroups. 
\end{theorem}

\smallskip

Now we are able to compare the dimension of affine Deligne-Lusztig varieties for various $b$ in the torus $S$ for a residually split group. 

\begin{theorem}\label{torus}
Let $G$ be a residually split group. Let $w \in \tW$ and $\mu \in X_*(T)_\G$ such that $\ell(w)=\ell(w t^{-\mu})+\ell(t^{\mu})$. Then $$\dim X_w (t^\mu) \ge \dim X_{w t^{-\mu}}(1).$$ 
\end{theorem}

\begin{proof}
By Theorem \ref{deg1} and the remark after Theorem \ref{2.22}, it suffices to prove the case where $G$ is split. 

By Theorem \ref{superset}, there exists $x \in W_0$ such that 
$$\dim X_{w t^{-\mu}}(1)=\dim_{\dot x U(\breve F) \dot x\i \cap \breve \CI} (\breve \CI \dot w \breve \CI \cap \dot x U(\breve F) \dot x \i).$$

Again by Theorem \ref{superset}, $$\dim X_w(\e^\mu)=\dim X_w(\e^{x \i(\mu)}) \ge \dim_{\dot x U(\breve F) \dot x\i \cap \breve \CI} (\breve \CI \dot w \breve \CI \e^{-\mu} \cap \dot x U(\breve F) \dot x \i)+\<\rho, x \i(\mu)-\bar \mu\>.$$ Here the first equality follows from the fact that $\e^\mu$ and $\e^{x \i(\mu)}$ are in the same $\s$-conjugacy class.\footnote{This is where the assumption that $G$ is residually split is used. At present, I do not know how to modify the argument so that it works for quasi-split groups as well.}

Since $\ell(w)=\ell(w t^{-\mu})+\ell(t^{\mu})$, we have the following commutative diagram
\[\xymatrix{
\breve \CI \dot w \e^{-\mu} \breve \CI \times_{\breve \CI} \breve \CI \e^\mu \breve \CI \e^{-\mu} \ar[r]^-m_-{\cong} \ar[d]^-{\cong} & \breve \CI \dot w \breve \CI \e^{-\mu} \ar@{=}[d] \\
\breve \CI \dot w \e^{-\mu} \breve \CI \times_{\breve \CI \cap \e^{\mu} \breve \CI \e^{-\mu}} \e^\mu \breve \CI \e^{-\mu} \ar[r]^-m_-{\cong} & \breve \CI \dot w \breve \CI \e^{-\mu}.
}
\]
Here $m$ is the multiplication map. 

The restriction to $\dot x U(\breve F) \dot x \i$ gives an injective map $$(\breve \CI \dot w \e^{-\mu} \breve \CI \cap \dot x U(\breve F) \dot x \i) \times_{\breve \CI \cap \e^{\mu} \breve \CI \e^{-\mu} \cap \dot x U(\breve F) \dot x \i} (\e^\mu \breve \CI \e^{-\mu} \cap \dot x U(\breve F) \dot x \i) \to \breve \CI \dot w \breve \CI \e^{-\mu} \cap \dot x U(\breve F) \dot x \i.$$ 
By \cite[Lemma 2.13.1]{GHKR06}, \begin{align*} & \dim_{\dot x U(\breve F) \dot x \i \cap \breve \CI} (\breve \CI \dot w \breve \CI \e^{-\mu} \cap \dot x U(\breve F) \dot x \i) \\ &\ge \dim_{\dot x U(\breve F) \dot x \i \cap \breve \CI}(\breve \CI \dot w \e^{-\mu} \breve \CI \cap \dot x U(\breve F) \dot x \i)+\dim((\e^\mu \breve \CI \e^{-\mu} \cap \dot x U(\breve F) \dot x \i)/(\breve \CI \cap \e^\mu \breve \CI \e^{-\mu} \cap \dot x U(\breve F) \dot x \i)) \\ &=\dim_{\dot x U(\breve F) \dot x \i \cap \breve \CI}(\breve \CI \dot w \e^{-\mu} \breve \CI \cap \dot x U(\breve F) \dot x \i)+\sum_{\a \in R_+} \min\{{\<x \a, \mu\>, 0}\} \\ &=\dim_{\dot x U(\breve F) \dot x \i \cap \breve \CI}(\breve \CI \dot w \e^{-\mu} \breve \CI \cap \dot x U(\breve F) \dot x \i)+\<\rho, x \i(\mu)-\bar \mu\>=\dim X_{w t^{-\mu}}(1).\qedhere\end{align*}
\end{proof}

We have the following result on the non-emptiness pattern and dimension formula for the intersection $\breve \CI \dot w \breve \CI \cap [\e^\mu]$. For split groups, it is conjectured by G\"ortz, Haines, Kottwitz, and Reuman in \cite[Conjecture 9.5.1 (b)]{GHKR10}. 

\begin{theorem}\label{torus2}
Let $G$ be a residually split, simple group. Let $\mu$ be a dominant coweight. Let $w \in \tW$ such that $\k(w)=\k(\e^\mu)$ and that $w (\fka)$ is contained in a very Shrunken Weyl chamber with respect to $\mu$. If $\eta_\s(w)$ is not contained in the relative Weyl group $W_M$ for any proper standard Levi subgroup $M$ of $G$. Then $$\dim_{\breve \CI} (\breve \CI \dot w \breve \CI \cap [\e^\mu])=d_w([\e^\mu]).$$ In particular, $\breve \CI \dot w \breve \CI \cap [\e^\mu] \neq \emptyset$. 
\end{theorem}

\begin{proof}
By Theorem \ref{vir}, $\dim_{\breve \CI} (\breve \CI w \breve \CI \cap [\e^\mu] \le d_w([\e^\mu])$. 

Let $x \in W_0$ such that $w (\fka) \subset x(\CC^-)$. Let $\mu'$ be the coweight with $w\i t^{x(\mu)} w=t^{\mu'}$. By definition, $\ell(w t^{-\mu'})=\ell(t^{-x(\mu)} w)=\ell(w)-\ell(t^\mu)$ and $w t^{-\mu'}(\fka) \subset x(\CC^-)$ is in the Shrunken Weyl chamber. Thus $\eta_\s(w)=\eta_\s(w t^{-w\i x(\mu)})$. Note that $[\e^\mu]=[\e^{\mu'}]$. By Theorem \ref{ii-adlv}, Theorem \ref{torus} and Theorem \ref{bb-dim}, \begin{align*} \dim_{\breve \CI} (\breve \CI \dot w \breve \CI \cap [\e^{\mu'}]) & =\dim X_w (\e^{\mu'})+\<2 \rho, \mu\> \ge \dim X_{w t^{-\mu'}}(1)+\<2 \rho, \mu\> \\ &=d_{w t^{-\mu'}}([1])+\<2 \rho, \mu\>=\frac{1}{2}(\ell(w t^{-\mu'})+\ell(\eta_\s(w t^{-\mu'})))+\<2 \rho, \mu\> \\ &=\frac{1}{2}(\ell(w)+\ell(\eta_\s(w))+\<\rho, \bar \mu\>=d_w([\e^\mu]).\qedhere\end{align*}
\end{proof}

\subsection{The Kottwitz-Rapoport conjecture}

\subsubsection{Admissible set $\Adm(\mu)$}
In fact, we are not only interested in the intersection of a $\s$-conjugacy class with a single $\breve \CI$-double coset, but also interested in the intersection of a $\s$-conjugacy class with a certain union of $\breve \CI$-double cosets. Such union of $\breve \CI$-double cosets arises in the study of Shimura varieties. Let us first look at the following example \cite[\S 4.4]{Hai}. 

\begin{example}
Let $R$ be a $\BZ_p$-algebra. We consider the pairs $(\CF_0, \CF_1)$ of locally free rank $1$ submodules of $R^2$ such that the diagram commutes
\[
\xymatrix{
R \oplus R \ar[r]^{\begin{tiny} \begin{bmatrix} p & 0 \\ 0 & 1 \end{bmatrix} \end{tiny}} & R \oplus R \ar[r]^{\begin{tiny} \begin{bmatrix} 1 & 0 \\ 0 & p \end{bmatrix} \end{tiny}} & R \oplus R \\
\CF_0 \ar[u] \ar[r] & \CF_1 \ar[u] \ar[r] & \CF_0 \ar[u] .
}
\]

This functor represents a closed subscheme of $\BP^1_{\BZ_p} \times \BP^1_{\BZ_p}$, which we denote by $M^{loc}$. This is the ``local model'' of Shimura variety associated to the fake unitary group $U(1, 1)$. 

In the generic fiber, $M^{loc}_{\BQ_p} \cong \BP^1_{\BQ_p}$. 

In the special fiber, $M^{loc}_{\BF_p}$ is the union of two $\BP^1_{\BF_p}$ meeting in a point and we may regard $M^{loc}_{\BF_p}$ as the union of the three Schubert cells in the affine flag variety of $GL_2(\BF_p((\e)))$. These two open Schubert cells correspond to the elements $t^{[1, 0]}$ and $t^{[0, 1]}$ in the Iwahori-Weyl group of $GL_2$ and the closed Schubert cell correspond to the element $t^{[1, 0]} (1 2)$, which is the unique length-zero element in $t^{[1, 0]} W_a=t^{[0, 1]} W_a$. 
\end{example}

In general, for any coweight $\mu$, we define the admissible set $$\Adm(\mu)=\{w \in \tW; w \le t^{x(\mu)} \text{ for some } x\in W_0\}.$$ Similarly, for any parahoric subgroup $K$ of $G$, we set $\Adm(\mu)^K=W_K \Adm(\mu) W_K$ and denote by $\Adm(\mu)_K$ the image of $\Adm(\mu)^K$ in $W_K \backslash \tW/W_K$. The admissible set is an interesting combinatorial object. We refer to \cite{HH} and \cite{Hax} for some nice properties on $\Adm(\mu)$. 

The set $\Adm(\mu)_K$ is expected to parametrize a natural stratification of the special fiber of the local model of a Shimura variety with level $K$ structure. This has been established for Shimura varieties of PEL type at places of tame, parahoric reduction by Pappas and Zhu \cite{PZ}. 

\subsubsection{The union $M_K$ of $\breve \CK$-double cosets}
We consider the following set $$M_K=\sqcup_{w \in \Adm(\mu)_K} \breve \CK \dot w \breve \CK.$$ We discuss the relation of $M_K$ with the admissible subsets of $G(\breve F)$ in \S \ref{bruhat} and \S \ref{BG}. 

As a $\breve \CK \times \breve \CK$-double cosets, we have
\begin{theorem}
Let $\breve \CK$ be a parahoric subgroup of $G(\breve F)$. Then  $$\breve \CK \backslash M_K/\breve \CK=\Adm(\mu)_K.$$
\end{theorem} 
This is just a reformulation of definition. 

Next, as a subset of $G(\breve F)$ stable under the $\s$-twisted conjugation action of $\breve \CK$,

\begin{theorem}
Let $\breve \CK$ be a parahoric subgroup of $G(\breve F)$. Then $$M_K=\sqcup_{w \in \Adm(\mu) \cap {}^K \tW} \breve \CK \cdot_\s \breve \CI \dot w \breve \CI.$$
\end{theorem}

This follows from Theorem \ref{kwk} and the fact that $\Adm(\mu)^K \cap {}^K \tW=\Adm(\mu) \cap {}^K \tW$ proved in \cite[Theorem 6.1]{He00} (see \cite{HH} for another proof). 

As a consequence, we may compare the set $M_K$ for different parahoric subgroups. 

\begin{corollary}
Let $\breve \CK' \subset \breve \CK$ be parahoric subgroups. Then $$M_K=\breve \CK \cdot_\s M_{K'}.$$
\end{corollary}

\subsubsection{$\Adm(\mu)$ and $B(G, \mu)$}
Now we discuss the relation of $M_K$ with the $\s$-conjugacy classes of $G(\breve F)$. 

We denote by $B(G, \mu)$ the subset of $B(G)$ consisting of $\s$-conjugacy classes $[b]$ of $G(\breve F)$ such that $\k(b)=\k(\mu)$ and the Newton point $\nu_b$ is less than or equal to the $\varsigma$-average of $\mu$ in the dominance order. The relation between $M_K$ and $B(G)$ is as follows. 

\begin{theorem}
Let $\breve \CK$ be a parahoric subgroup of $G(\breve F)$. Let $[b] \in B(G)$. Then $[b] \cap M_K \neq \emptyset$ if and only if $[b] \in B(G, \mu)$. 
\end{theorem}

This result is conjectured by Kottwitz and Rapoport in \cite{KR} and \cite{R:guide}. The ``only if'' part is a group-theoretic version of Mazur's inequality. The case where $G$ is an unramified group and $\breve \CK$ is a hyperspecial maximal subgroup, is proved by Rapoport and Richartz in \cite[Theorem 4.2]{RR}. Another proof is given by Kottwitz in \cite{Ko3}. The case where $G$ is an unramified group and $\breve \CK$ is an Iwahori subgroup, is proved in \cite[Notes added June 2003, (7)]{R:guide}. The ``if'' part is the ``converse to Mazur's inequality'' and is proved by Wintenberger in \cite{Wi} in case $G$ is quasi-split. The general case of both directions is proved in \cite[Theorem A]{He00}. 

\subsubsection{The closure relation on $B(G, \mu)$}
In \S \ref{closure-ii}, we have discussed the closure relation between the $\s$-conjugacy classes of $G(\breve F)$. For $[b], [b'] \in B(G, \mu)$, $[b'] \subset \overline{[b]}$ if and only if $\nu_{b'} \le \nu_b$ in the dominance order (here $\k(b')=\k(b)$ is automatically satisfied). 

Now we show that $M_K=\sqcup_{[b] \in B(G, \mu)} M_K \cap [b]$ is a weak stratification. This is the group-theoretic analogy of the Grothendieck conjecture on the Newton strata of Shimura varieties. 

\begin{theorem}
For $[b], [b'] \in B(G, \mu)$, $\overline{M_K \cap [b]} \cap [b'] \neq \emptyset$ if and only if $\nu_{b'} \le \nu_b$. 
\end{theorem}

\begin{remark}
The proof is similar to the proof of \cite[Theorem 5.6]{HR2} on Shimura varieties. The statement in loc.cit relies on several axioms on the Shimura varieties. The statement here, on the other hand, is valid unconditionally.  
\end{remark}

\begin{proof}
If $\overline{M_K \cap [b]} \cap [b'] \neq \emptyset$, then $\overline{[b]} \cap [b'] \neq \emptyset$ and thus $\nu_{b'} \le \nu_b$. 

Now suppose that $\nu_{b'} \le \nu_b$. Since $\k(b')=\k(b)$, by Theorem \ref{tri-par}, $\CO_{[b']} \preceq \CO_{[b]}$. By \cite[Theorem 3.3]{He00}, there exists a $\s$-straight element $w \in \Adm(\mu) \cap \CO_{[b]}$. By Theorem \ref{min-dl}, $\breve \CI \dot w \breve \CI \subset M_K \cap [b]$. Thus $\overline{\breve \CI \dot w \breve \CI} \subset \overline{M_K \cap [b]}$. By the definition of $\preceq$, there exists an element $w' \in \CO_{[b']}$ with $w' \le w$. Thus $\dot w' \in [b']$ and $\dot w' \in \overline{\breve \CI \dot w \breve \CI} \subset \overline{M_K \cap [b]}$. In other words, $\overline{M_K \cap [b]} \cap [b'] \neq \emptyset$.
\end{proof}

\subsection{Application}\label{Shimura}

The study of the admissible subsets of $G(\breve F)$ and their intersections we discussed above has found important applications in the study of Shimura varieties. They serve as the group-theoretic model of the some characteristic subsets (Newton strata, Kottwitz-Rapoport strata, Ekedahl-Oort strata, etc) in the reduction modulo $p$ of a Shimura variety with parahoric level structure. These subsets have been studied intensively in the last two decades. We refer to the survey articles by Rapoport \cite{R:guide}, by Haines \cite{Hai}, and to the new preprint \cite{HR2} on a group-theoretic approach to study these characteristic subsets. 

In Table \ref{table2-2}, we give a very rough comparison between the admissible subsets of $G(\breve F)$ we discussed above and the characteristic subsets of Shimura varieties. 

Let $(\bG, \{h\})$ be a Shimura datum. Let $\bK=K^p K_p$ be an open compact subgroup of $\bG(\BA_f)$, where $K=K_p$ is a parahoric subgroup of $\bG(\bQ_p)$. Let $G=\bG \otimes_{\BQ} \BQ_p$. Let $\mu$ be the dominant coweight corresponding to $\{h\}$. We regard $M_K$ as the group-theoretic model for the special fiber $Sh_K$ of the Shimura variety $Sh_\bK$. 

The stratifications $$M_K=\sqcup_{w \in \Adm(\mu)_K} \breve \CK \dot w \breve \CK, \qquad M_K=\sqcup_{[b] \in B(G, \mu)} (M_K \cap [b])$$ are the group-theoretic analogy of the Kottwitz-Rapoport stratification and the Newton stratification of the corresponding Shimura varieties. The stratification $$M_K=\sqcup_{w \in \Adm(\mu) \cap {}^K \tW} \breve \CK \cdot_\s \breve \CI \dot w \breve \CI$$ is the group-theoretic analogy of the Ekedahl-Kottwitz-Oort-Rapoport stratification, a stratification for Shimura varieties with parahoric level structure which interpolates between the Kottwitz-Rapoport stratification of Shimura varieties with Iwahori level structure and the Ekedahl-Oort stratification \cite{Oo}, \cite{Vi2} of Shimura varieties with hyperspecial level structure. 

\begin{table}
\caption{Group-theoretic analogy of Shimura varieties}
\begin{tabular}{|l | l |}
\hline
Group-theoretic side & Shimura variety side  \\
\hline
$M_K$ & $Sh_K$ \\
\hline
$\breve \CK \dot w \breve \CK$ for $w \in W_K \backslash \tW/W_K$ & Kottwitz-Rapoport stratum $KR_{K, w}$  \\
\hline
$M_K \cap [b]$ for $[b] \in B(G)$ & Newton stratum $\mathit S_{K, [b]}$ \\
\hline
$\breve \CK \cdot_\s \breve \CI \dot w \breve \CI$ for $w \in {}^K \tW$ & EKOR stratum $EKOR_{K, w}$\\
\hline
$M_K=\sqcup_{w \in \Adm(\mu)_K} \breve \CK \dot w \breve \CK$ & Non-emptiness of KR strata (in the natural range) 
\\ \hline
$M_K=\sqcup_{[b] \in B(G, \{\mu\})} (M_K \cap [b])$ & Non-emptiness of Newton strata
\\ \hline
$M_K=\sqcup_{w \in \Adm(\mu) \cap {}^K \tW} \breve \CK \cdot_\s \breve \CI \dot w \breve \CI$ & Non-emptiness of EKOR strata
\\ \hline
$\breve \CK \dot w \breve \CK=\sqcup_{x \in {}^K \tW \cap W_K w W_K} \breve \CK \cdot_\s \breve \CI \dot w \breve \CI$ & Each KR stratum is a union of EKOR strata
\\ \hline
$\overline{\breve \CK \dot w \breve \CK}=\sqcup_{w' \le w} \breve \CK \dot w' \breve \CK$ & Closure relation between KR strata
\\ \hline
$\overline{M_K \cap [b]} \cap [b'] \neq \emptyset \iff \nu_{b'} \le \nu_b$ & Closure relation between Newton strata
\\ \hline
$\overline{\breve \CK \cdot_\s \breve \CI \dot w \breve \CI}=\sqcup_{w' \preceq_{K, \s}} \breve \CK \cdot_\s \breve \CI \dot w' \breve \CI$ & Closure relation between EKOR strata
\\ \hline
$\dim_{\breve \CK} \breve \CK \dot w \breve \CK=\max\{\ell(x); x \in W_K w W_K \cap {}^K \tW\}$ & Conjectural dimension for KR strata
\\ \hline
$\dim_{\breve \CK} \breve \CK \cdot_\s \breve \CI \dot w \breve \CI=\ell(w)$ & Conjectural dimension for EKOR strata
\\ \hline
For $\s$-straight $w$, $\breve \CK \cdot_\s \breve \CI \dot w \breve \CI \subset [\dot w]$ & Special EKOR stratum in a single Newton stratum
\\ \hline
For $\breve \CK' \subset \breve \CK$, $M_K=\breve \CK \cdot_\s M_{K'}$ & Change of parahoric
\\ \hline
\end{tabular}
\label{table2-2}
\end{table}

We do not have a simple dimension formula for Newton strata. However, we have the following diagram
\[\xymatrix{
& \{g \in G(\breve F); g \i b \s(g) \in M_K\} \ar[dl]_-{p_1} \ar[dr]^-{p_2} & \\
X(\mu, b)_K & & M_K \cap [b]
}
\]

Here the map $p_1$ is the projection map and the map $p_2$ is a $J_b$-torsor. Note that $J_b$ is a discrete group. This suggests that there might be a way to ``locally'' identify $M_K \cap [b]$ with $\{g \in G(\breve F); g \i b \s(g) \in M_K\}$. Notice that $$\dim_{\breve \CK} (M_K \cap [b])-\dim X(\mu, b)_K=\<2 \rho, \nu_b\>,$$ where $X(\mu, b)_K=\{g \breve \CK\in G(\breve F)/\breve \CK; g \i b \s(g) \in M_K\}$ is a union of affine Deligne-Lusztig varieties. The number $\<2 \rho, \nu_b\>$ is exactly the conjectural dimension of central leaves.

\section{Part III: Cocenters of affine Hecke algebras}

\subsection{Motivation: group algebras} To explain some basic idea of the cocenter-representation duality, we start with a naive example. 

Let $G$ be a finite group and $V$ be a finite dimensional complex representation of $G$. We define the character of $V$ by $\chi_V(g)=Tr(g, V)$ for $g \in G$. Then $\chi_V$ is a class function, i.e., $\chi_V(g)=\chi_V(g')$ if $g$ and $g'$ are conjugate in $G$. 

Let $V_1, \cdots, V_k$ be the irreducible representations of $G$ (up to isomorphism) and $\CO_1, \cdots, \CO_l$ be the conjugacy classes of $G$. It is well known that $k=l$ and the matrix $(\chi_{V_i}(g_j))_{1 \le i, j \le k}$ is invertible, where $g_j$ is a representative of $\CO_j$. This matrix is called the {\it character table} of $G$. 

We reformulate it in a different way. 

Let $H_1=\BC[G]$ be the group algebra of $G$. Let $[H_1, H_1]$ be the commutator of $H_1$, the subspace of $H_1$ spanned by $[h, h']:=h h'-h' h$ for all $h, h' \in H_1$. We call the quotient space $\bar H_1=H_1/[H_1, H_1]$ the {\it cocenter} of $H_1$.  

It is easy to see that

(1) For any $g, g'$ in a given conjugacy class $\CO$ of $G$, the image of $g$ and $g'$ in $\bar H_1$ are the same. We denote it by $[\CO]$. 

(2) $\{[\CO_1], \cdots, [\CO_k]\}$ is a basis of $\bar H_1$. 

Let $R(H_1)=R(G)$ be the Grothendieck group of finite dimensional complex representations of $H_1$. Then $\{V_1, \cdots, V_k\}$ is a basis of $R(H_1)$. The trace map $$Tr: H_1 \to R(H_1)^*, \qquad g \mapsto (V \mapsto Tr(g, V))$$ factors through $Tr: \bar H_1 \to R(H_1)^*$. Moreover,
\[\text{ $Tr: \bar H_1 \to R(H_1)^*$ is an isomorphism of vector spaces. }\]
We call it the {\it cocenter-representation duality} of the group $G$. 

\subsection{Hecke algebras}\label{2.2} Let $(W, \BS)$ be a Coxeter system. Fix a set of indeterminates $\mathbf c=\{c(s); s\in \BS\}$
such that $c(s)=c(t)$ if $s$ and $t$ are conjugate in $W$, and let
$\Lambda=\BZ[c(s); s\in \BS].$

The generic Hecke algebra
  $\CH=\CH(W,\mathbf c)$ is the $\Lambda$-algebra generated by $\{T_w;
  w\in \tW\}$ subject to the relations:
\begin{enumerate}
\item $T_w\cdot T_{w'}=T_{ww'}$, if $\ell(ww')=\ell(w)+\ell(w')$;
\item $(T_s+1)(T_s-c(s))=0,$ $s\in \BS$.
\end{enumerate}

Note that the definition of Hecke algebra works for extended affine Weyl groups as well. 

Let $\kk$ be a field. If we assign an element $q_s \in \kk$ to $c(s)$ for $s \in \BS$, then we can regard $\kk$ as a $\L$-module and $H_{\fbq}=\CH \otimes_\L \kk$ is the specialization of $\CH$. In particular, if $q_s=1$ for all $s \in \BS$, then we obtain the group algebra $H_1=\kk[W]$. 

We are interested in the cocenter of $\CH$, the representations of $\CH$ and their relation via the trace map. 

One difficulty is that given two elements $w$ and $w'$ in the same conjugacy class of $W$, the image of $T_w$ and $T_{w'}$ in $\bar \CH$ may not be the same. In fact, we should look at the ``strongly conjugate'' elements or the elements in the same $\approx$-equivalence class. 

\begin{proposition}\label{ww'}
Let $w, w' \in W$. 

(1) If $w \approx w'$, then the image of $T_w$ and $T_{w'}$ in $\bar \CH$ are the same. 

(2) If $w \sim w'$ and $q_s \neq 0$ for all $s \in \BS$, then the image of $T_w$ and $T_{w'}$ in $\bar H_{\fbq}$ are the same.
\end{proposition}

\begin{proof}
(1) It suffices to prove the case where $w \xrightarrow{s} w'$ and $\ell(w)=\ell(w')$. Without loss of generality, we may assume furthermore that $s w<w$. Then $T_w=T_s T_{s w}$ and $T_{w'}=T_{s w} T_{s}$. Hence the image of $T_w$ and $T_{w'}$ are the same. 

(2) It suffices to prove the case where $w'=x w x \i$ and $\ell(x w)=\ell(w)+\ell(x)=\ell(w')+\ell(x)$. Then we have $T_x T_w=T_{x w}=T_{w' x}=T_{w'} T_x$. Since $q_s \neq 0$ for all $s \in \BS$, $T_x$ is invertible in $H_{\fbq}$. Thus $T_w=T_x \i T_{w'} T_x \equiv T_{w'} \mod [H_{\fbq}, H_{\fbq}]$. 
\end{proof}

\begin{proposition}\label{3-2}
We denote by $W_{\min}$ the set of elements of $W$ minimal length in their conjugacy class. Suppose that the {\bf Red-Min} Property holds for every conjugacy class of $W$. Then for any $w \in W$, the image of $T_w$ in $\bar \CH$ is a linear combination of the image of $T_x$ for $x \in W_{\min}$.
\end{proposition}

\begin{proof}
We argue by induction on $w$. 

If $w \in W_{\min}$, then the statement automatically holds for $w$. 

If $w \notin W_{\min}$, then by the {\bf Red-Min} property, there exists $w' \in W$ and $s \in \BS$ such that $w' \approx w$ and $s w' s<w'$. By Proposition \ref{ww'} (1), $$T_w \equiv T_{w'} \mod [\CH, \CH].$$ Since $s w' s<w$, \begin{align*} T_{w'} &=T_s T_{s w' s} T_s \equiv T_{w'} T_s^2=(c(s)-1) T_{s w' s} T_s+c(s) T_{s w' s} \\ &=(c(s)-1) T_{s w'}+c(s) T_{s w' s} \mod [\CH, \CH].\end{align*} Now the statement for $w$ follows from inductive hypothesis on $s w'$ and on $s w' s$. 
\end{proof}

\begin{theorem}\label{3-3}
Let $W$ be a finite Coxeter group or an extended affine Weyl group. Suppose that $q_s \neq 0$ for all $s \in \BS$. Then 

(1) For any conjugacy class $\CO$ of $W$, the image of $T_w$ in $\bar H_{\fbq}$ is independent of the choice of $w \in \CO_{\min}$. We denote the image by $T_\CO$. 

(2) The set $\{T_\CO; \CO \in W/W_\D\}$ spans $\bar H_{\fbq}$. 
\end{theorem}

\subsection{Finite Hecke algebras}\label{finite-ct}
In this section, we assume that $W$ is a finite Coxeter group and $H_{\fbq}$ is a finite Hecke algebra with nonzero parameters. 

By Tits' deformation theorem \cite[Theorem 7.4.6 \& 7.4.7]{GP00}, for generic parameter $\mathbf c=\{c(s); s \in \BS\}$, the Hecke algebra $H_{\mathbf c}$ is isomorphic to the group algebra $H_1$, and hence the number of irreducible representations of $H_{\mathbf c}$ equals the number of irreducible representations of $H_1$, and hence equals the number of conjugacy classes of $W$. 

By comparing the number of irreducible representations of $H$ and the number of conjugacy classes of $W$, one deduces that $\{T_\CO\}$ is in fact a basis of $\bar H$. Therefore, 

\begin{theorem}
Let $W$ be a finite Coxeter group and $H_{\mathbf c}$ be the Hecke algebra of $W$ with generic parameter $\mathbf c$. Then $Tr: \bar H_{\mathbf c} \to R(H_{\mathbf c})^*$ is an isomorphism. 
\end{theorem}

This leads to the definition and study of ``character table'' of finite Hecke algebras in \cite{GP93}. 

\begin{example}
Let $W=S_3$ be the Weyl group of type $A_2$. There are three conjugacy classes of $W$: $\{1\}, \{s_1, s_2\}, \{s_1 s_2, s_2 s_1\}$. For generic parameter $q \in \kk^\times$, there are three irreducible representations of $H_{\fbq}$: trivial representation, Steinberg representation and a $2$-dimensional irreducible representation $\pi$. The character table is given as follows: 
\[\begin{tabular}{|c | c | c| c|}
\hline
$A_2$ &  $triv$ & $St$ & $\pi$ \\ \hline
$T_{s_1 s_2}$ &    $q^2$      &  $1$ & $-q$ \\ \hline
$T_{s_1}$       &   $q$       & $-1$ & $q-1$ \\ \hline
$1$       &    $1$      &  $1$ & $2$ \\
\hline
\end{tabular}\]
\end{example}

\begin{example}
Let $W$ be the Weyl group of type $C_2$. There are three conjugacy classes of $W$: $\{1\}, \{s_1, s_2 s_1 s_2\}, \{s_2, s_1 s_2 s_1\}, \{s_1 s_2, s_2 s_1\}, \{s_1 s_2 s_1 s_2\}$. For generic parameter $\fbq=(q_1, q_2) \in \kk^\times \times \kk^\times$, there are five irreducible representations of $H_{\fbq}$. We label these five modules by the bipartitions which parameterized the corresponding representations of the finite Weyl group: $2 \times 0$, $11 \times 0$, $0 \times 2$, $0 \times 11$, and $1 \times 1$. The character table is given as follows: 
\[
\begin{tabular}{|c|c|c|c|c|c|}
\hline
$C_2$&$2\times 0$ &$11\times 0$ &$0\times 2$ &$0\times 11$ &$1\times 1$ \\
\hline
$T_{s_1 s_2} $ &$q_1q_2$ &$-q_2$ &$-q_1$ &$1$ &$0$  \\
\hline
$T_{s_1 s_2 s_1 s_2}$ &$(q_1q_2)^2$ &$q_2^2$ &$q_1^2$ &$1$ &$-2q_1q_2$ \\
\hline
$T_{s_1}$ &$q_1$ &$-1$ &$q_1$ &$-1$ &$q_1-1$ \\
\hline
$T_{s_2}$ &$q_2$ &$q_2$ &$-1$ &$-1$ &$q_2-1$ \\
\hline
$1$ &$1$ &$1$ &$1$ &$1$ &$2$ \\
\hline
\end{tabular}
\]
\end{example}

\subsection{Representations of affine Hecke algebras and of $p$-adic groups}
In this subsection, we recall the relations between the representations of affine Hecke algebras and of $p$-adic groups. The main reference of this subsection is Vign\'eras' talk at ICM2002 \cite{Vig02}. This serves as one of the main motivation to our study of cocenter-representation duality of affine Hecke algebras. 

Let $G$ be a connected reductive group over a non-archimedean local field $F$ with residual field $\BF_q$ and let $\CI$ be an Iwahori subgroup of $G(F)$. Let $\kk$ be an algebraically closed field of characteristic $l \neq p$. Let $H_\kk(G, \CI)=\End_{\kk[G]} \kk[\CI \backslash G(F)]$ be the Iwahori Hecke $\kk$-algebra of $G(F)$. It is isomorphic to the extended affine Hecke algebra associated to $G$ with parameter $q \in \kk$. 

For any $\kk$-representation $V$ of $G(F)$, the space $V^\CI$ of $\CI$-fixed points is a right $H_\kk(G, \CI)$-module. Moreover, 

\begin{theorem}
The map $V \mapsto V^\CI$ gives a bijection between the irreducible $\kk$-representations of $G$ with $V^\CI \neq 0$ and the simple right $H_\kk(G, \CI)$-modules. 
\end{theorem}

For complex representations, the equivalence of category is well-known. For modular representations, it is due to Vign\'eras \cite{Vig98}. 

The case $l=p$ is very different. However, there is still a conjectural relation between the representations of $G$, and of $H_\kk(G, \CI)$ (with zero parameter) and of the pro-$p$-Iwahori Hecke algebra $H_\kk(G, \CI_p)$. For $GL_2(\BQ_p)$, it is proved by Barthel and Livn\'e in \cite{BL1} and \cite{BL2}. 

\subsection{Affine Hecke algebras}

We recall the definition of affine Hecke algebras. Let $\sR=(X^*, X_*, R, R^\vee, \Pi)$ be a based root datum and $\tW=X_* \rtimes W_0$ be the extended affine Weyl group associated to $\sR$ (see \S \ref{1.7}). Let $\mathbf c=\{c(s); s\in \tilde \BS\}$ such that $c(s)=c(t)$ if $s$ and $t$ are conjugate in $\tW$, and let $\Lambda=\BZ[c(s); s\in \tilde \BS].$ We define the extended affine Hecke algebra $\tilde \CH$ and its specialization $\tilde H_{\fbq}$ in a similar way as in \S\ref{2.2}. We fix a set of nonzero parameters $\fbq$ in an algebraically closed field $\kk$ and we may simply write $\tilde H$ instead of $\tilde H_{\fbq}$. The basis $\{T_w\}_{w \in \tW}$ gives the Iwahori-Matsumoto presentation of $\tilde H$. It is related to the quasi-Coxeter structure of $\tW$. 

Another presentation we need for $\tilde H$ is the Bernstein-Lusztig presentation:
$$\tilde H=\oplus_{x \in X_*, w \in W_0} \kk\, \th_x T_w.$$
We refer to \cite{L1} for the definition of $\th_x$. This presentation is related to the semi-product $\tW=X_* \rtimes W_0$. It plays an important role in the study of representations of affine Hecke algebras, especially the (parabolic) induction and restriction functors. 

For any $J \subset \Pi$, let $\tilde H_J$ be the parabolic subalgebra of $\tilde H$ generated by $\{\theta_x T_w; x \in X, w \in W_J\}$. This is the extended affine Hecke algebra for the parabolic subgroup $\tilde W_J=X_* \rtimes W_J$, for some parameters. 

We may then define the (parabolic) induction and restriction functors between the Grothendieck groups of the finite dimension $\kk$-representations of $\tilde H$ and $\tilde H_J$. 
$$i_J: R(\tilde H_J) \to R(\tilde H), \qquad r_J: R(\tilde H) \to R(\tilde H_J).$$

We have the central character $\chi_t$ of $\tilde H_J$, here $t$ runs over a complex torus $T^J$ associated to $J$. If $\s \in R(\tilde H_J)$, then $\s \chi_t \in R(\tilde H_J)$.  Following Bernstein, Deligne and Kazhdan, we say that a form $f \in R(\tilde H)^*$ is {\it good} if for any $J \subset \Pi$ and $\s \in R(\tilde H_J)$, the function $t \mapsto f(i_J(\s \chi_t))$ is a regular function on $T^J$. 

The following result is obtained by Bernstein, Deligne and Kazhdan \cite{BDK} and \cite{Kaz}. 

\begin{theorem}
Let $\BH=\BC^\infty_c(G(F))$ be the (complex) Hecke algebra of a $p$-adic group $G(F)$. Then 

(1) The image of $Tr: \bar \BH \to R(\BH)^*$ is the set of good forms. 

(2) The map $Tr: \bar \BH \to R(\BH)^*$ is injective. 
\end{theorem}

Part (1) is referred to as trace Paley-Wiener Theorem and Part (2) is referred to as the density Theorem. The proofs rely on p-adic analysis. 


Our goal is to understand for which parameters, the trace Paley-Wiener Theorem and the density Theorem holds for affine Hecke algebras and to have an explicit basis of the cocenter. 

We have seen in Theorem \ref{3-3} that the set $\{T_\CO; \CO \in \tW/\tW_\D\}$ forms a basis of $\overline{\tilde H}$. We also have the Bernstein-Lusztig presentation of these elements. It is obtained  in \cite[Theorem B]{HN15}. 

\begin{theorem}
Let $i_J: \tilde H_J \to \tilde H$ be the inclusion and $\bar i_J: \overline{\tilde H_J} \to \overline{\tilde H}$ the induced map. Let $(J, x, K, C) \in \tilde \G$ and $\CO$ be the associated $\d$-conjugacy class of $\tW$, then $$T_\CO=\bar i_J(T^J_{C x}),$$ where $T^J_{C x}$ is the image of $T^J_{u x}$ in $\overline{\tilde H_J}$ for any minimal length element $u \in C$. 
\end{theorem}

Note that the relation between Iwahori-Matsumoto basis and Bernstein-Lusztig basis of $\tilde H$ is complicated. The relation between the minimal length elements in $\CO$ (with respect to the length function of $\tW$) and the minimal length elements in $C$ (with respect to the length function of $\tW_J$) is also complicated. It is amazing that these two complexities cancel each other out. This leads to the above matching between Iwahori-Matsumoto basis and Bernstein-Lusztig basis of the cocenter of $\tilde H$. 

Compared to the finite Hecke algebras, it is more difficult to prove that $\{T_\CO\}$ is linearly independent in the cocenter of extended affine Hecke algebras. The reason is that both $\overline{\tilde H}$ and $R(\tilde H)$ are of infinite rank and thus one can't compare their rank directly to make the conclusion. For equal parameter case, linearly independence is proved in \cite{HN14} using Lusztig's $J$-ring. The general case is proved in \cite{CH2} using the cocenter-representation duality that we are going to discuss.

\subsection{Cocenter-representation duality for affine Hecke algebras}\label{6}

\smallskip

In this section, we fix a nonzero parameter $\fbq$ in $\BC$ and we simply write $\tilde H$ instead of $\tilde H_\fbq$. 

\subsubsection{Elliptic quotient} Elliptic representation theory, introduced by Arthur \cite{A}, studies the Grothendieck group of certain representations of a Lie-theoretic
group modulo those induced from proper parabolic subgroups. The elliptic
theory of representations of semisimple $p$-adic groups and
Iwahori-Hecke algebras is further studied intensively, e.g., Schneider-Stuhler
\cite{SS}, Bezrukavnikov \cite{Bez}, Reeder \cite{Re}, Opdam-Solleveld
\cite{OS}. 

For an affine Hecke algebra $\tilde H$ (of a given nonzero parameters in $\BC$), the {\it elliptic quotient} is defined to be $$R(\tilde H)_{ell}=R(\tilde H)_\BC/\sum_{J \subsetneqq \Pi} i_J(R(\tilde H_J)_\BC).$$

Opdam and Solleveld in \cite{OS} studied the affine Hecke algebras for positive parameters and showed that 

\begin{theorem}
Let $\fbq$ be a positive parameter function on $\tilde H$. Then  

(1) The dimension of $R(\tilde H)_{ell}$ is at most the number of elliptic conjugacy classes of $\tW$.

(2) The inequivalent discrete series forms an orthogonal set of $R(\tilde H)_{ell}$. 
\end{theorem}

In particular, one has an upper bound of the number of irreducible discrete series for affine Hecke algebra of positive parameters. A lower bound can be obtained by counting the central characters of the discrete series. This leads to the classification of irreducible discrete series for affine Hecke algebras of positive parameters in \cite{OS2}. 

The method of Opdam-Solleveld is analytic, passing from $\tilde H$ to its Schwartz algebra, a certain topological completion of $\tilde H$. For Hecke algebras of $p$-adic groups, the elliptic quotient and its relation with the trace map was also studied in \cite{BDK} and analytic methods were used to obtain an upper bound of the dimension of the elliptic quotient. 

\subsubsection{Rigid cocenter and rigid quotient} Here is our motivation to develop a different method to study elliptic representation theory for affine Hecke algebras. 

First, we would like to understand the affine Hecke algebras for arbitrary parameters (not just the positive parameters) and for representations over a field of positive characteristic. It is desirable to have a more algebraic method. 

Second, we would like to put the elliptic quotient in the framework of ``cocenter-representation duality''. Namely, the elliptic quotient $R(\tilde H)_{ell}$ corresponds to a subspace of $\overline{\tilde H}$ via the trace map $Tr: \overline{\tilde H} \times R(\tilde H)_\BC \to \BC$. What is this subspace? Results in \cite{BDK} indicates that this subspace is very complicated and may not have a nice explicit description. 

The idea in \cite{CH2} is to replace the elliptic quotient by the so-called rigid quotient. For simplicity, we only consider the extended affine Hecke algebras associated to semisimple root data. The reductive root data can be reduced to semisimple ones via \cite[\S 8.1]{CH2}.

We define the {\it rigid cocenter} $$\overline{\tilde H}^\rig =\text{span}\{T_\CO; \nu_\CO=0\}$$ and the {\it rigid quotient} $$R(\tilde H)_{\rig}=R(\tilde H)_\BC/\<i_J(\s)-i_J(\s \chi_t); J \subset \Pi, \s \in R(\tilde H_{J}), t \in T^J\>,$$ the quotient of $R(\tilde H)$ by the difference of induced modules. Both the rigid cocenter and the rigid quotient are finite dimensional since the root datum is semisimple. 

It is proved in \cite[Theorem 1.1]{CH2} that

\begin{theorem}\label{rig-pairing}
(1) For generic parameters, the trace map $Tr: \bar H \times R(\tilde H)_\BC \to \BC$ induces a perfect pairing \[Tr: \overline{\tilde H}^\rig  \times R(\tilde H)_{\rig} \to \BC.\]

In particular, $\dim R(\tilde H)_{\rig}=\dim \overline{\tilde H}^\rig$ equals the number of conjugacy classes in $\tW$ whose Newton points equal to zero. 

(2) For arbitrary nonzero parameters, the induced map $Tr: \overline{\tilde H}^\rig \to R(\tilde H)_{\rig}^*$ is surjective. 
\end{theorem}

As some consequences,

\begin{itemize}
\item Trace Paley-Wiener Theorem: For arbitrary nonzero parameters, the image of the map $Tr: \overline{\tilde H} \to R(\tilde H)^*$ is $R(\tilde H)^*_{good}$. 

\item Density Theorem: For generic parameters, the map $Tr: \overline{\tilde H} \to R(\tilde H)^*$ is injective. 

\item Basis Theorem: The set $\{T_\CO; \CO \in \tW/\tW_\D\}$ forms a basis of $\overline{\tilde H}$. 
\end{itemize}




\smallskip

We also have the following deformation theorem. 

\begin{theorem}\label{family}
For generic parameters $\fbq$, there exists a basis $\{V_{\pi, \fbq}\}$ of $R(\tilde H)_\BC$ such that for any $w \in \tW$ and any $\pi$, the action of $T_w$ on $V_{\pi, \fbq}$ depends analytically on $\fbq$. 
\end{theorem}

For affine Hecke algebra with positive parameters, a similar result was obtained by Opdam and Solleveld in \cite{OS2} using Schwartz algebras. The idea of our proof (without the restriction on positive parameters) is to use Lusztig's graded affine Hecke algebras \cite{L1} and the deformation theorem for graded affine Hecke algebras obtained in \cite{CH}. The passage from affine Hecke algebras to graded affine Hecke algebras is analytic. This is the reason that we have the analytic deformation theorem here. However, we expect that there exists a family of representations depending algebraically on $\fbq$. 

\subsubsection{Class polynomials}\label{3.6.3}

Now suppose that all the parameters $q_s$ are equal. We write $q=q_s$ for any $s \in \tilde \BS$. For any $w \in \tW$, there exists $F_{w, \CO} \in \BZ[q]$ for each conjugacy class $\CO$ of $\tW$ such that \begin{equation}\label{class-a} T_w \equiv \sum_{\CO} F_{w, \CO} T_\CO \mod [\tilde H, \tilde H].\end{equation}

The polynomials $F_{w, \CO}$ are called the {\it class polynomials} and can be computed inductively on $w$ as we discussed in \S \ref{dim-deg}. Note that the set $\{T_\CO\}$ is a basis of $\overline{\tilde H}$. Thus the polynomials $F_{w, \CO}$ are uniquely determined by the equality \ref{class-a} and thus is independent of the choice of the reduction procedure in \S \ref{dim-deg}. 

We also have that \begin{equation}\label{class-b} Tr(T_w, V)=\sum_{\CO} F_{w, \CO} Tr(T_\CO, V).\end{equation} for any finite dimensional representation $V$ of $\tilde H$. Therefore, the character value of any element in $\tilde H$ is known if one knows the character value of all the minimal length elments, together with the class polynomials. Moreover, by the density Theorem for generic parameter $q$, the polynomials $F_{w, \CO}$ are also determined by the equality \ref{class-b}. This gives a representation-theoretic definition of the class polynomials. 

As we have seen in \S \ref{dim-deg}, the class polynomials (especially the leading term of the class polynomial), have found important application in the arithmetic geometry. It is desirable to have an explicit formula for the class polynomials. This is a challenging problem. Some computations for type $\tilde A_2$ is done by Yang \cite{Yang}. 

\subsection{The ``Modular case"}

\subsubsection{Naive kernel conjecture}
Now we move to the ``modular case''. Here instead of complex representations, we consider representations over an algebraically closed field $\kk$ of positive characteristic. We also consider the case where the parameter is a root of unity. In this situation, one can't expect to have a perfect pairing as in Theorem \ref{rig-pairing} (1). However, we expect that the cocenter-representation duality fails in a ``controllable'' way and thus provide useful information on the representations of affine Hecke algebras in the ``modular case''. 

We still have a surjective map $Tr: \overline{\tilde H}^\rig \to R(\tilde H)^*_{rig}$, but it fails to be injective. It is interesting to see whether the whole kernel comes from the non-semisimplicity of the parahoric subalgebras. 


Let us make it precise. Let $K \subset \tilde \BS$. If $W_K$ is finite, then we say that the subalgebra $H_K$, generated by $T_w$ for $w \in W_K$, is a {\it parahoric subalgebra} of $\tilde H$. In this case, $H_K$ is a finite Hecke algebra. We set $\Omega(K)=\{\t \in \Omega; \t K \t \i=K\}$ and define the {\it extended parahoric subalgebra} $$H_K^\sharp=H_K \rtimes \kk[\Omega(K)].$$ We have a natural map $\bar H_K^\sharp \to \overline{\tilde H}^\rig$. The composition $\bar H_K^\sharp \to \overline{\tilde H}^\rig \to R(\tilde H)_{\rig}^*$ is not injective in general. It contains $\ker(\bar H_K^\sharp \to R(H_K^\sharp)^*)$ in its kernel. 

Now we introduce the {\it naive kernel}. Let $\ker^{naive}(Tr: \overline{\tilde H}^\rig \to R(\tilde H)_{\rig}^*)$ be the sum of the image in $\overline{\tilde H}^\rig$ of $\ker(\bar H_K^\sharp \to R(H_K^\sharp)^*)$, where $H_K$ runs over all the parahoric subalgebras of $\tilde H$. 

\begin{conjecture}\label{naive}
The naive kernel $\ker^{naive}(Tr: \overline{\tilde H}^\rig \to R(\tilde H)_{\rig}^*)$ equals the kernel of the trace map $Tr: \overline{\tilde H}^\rig \to R(\tilde H)_{\rig}^*$.
\end{conjecture}

If this conjecture holds, then one may reduce the study of parametrization of irreducible modular representations of extended affine Hecke algebras to the study of parametrization of irreducible modular representations of extended finite Hecke algebras. The latter case has been studied extensively by Geck and Jacon \cite{GJ11}. By combining all these together, one obtain the rank of the rigid quotient and the elliptic quotient of $R(\tilde H)$. 

In particular, if all the extended parahoric subalgebras of $\tilde H$ are semisimple, then the parameterization of irreducible modules of $\tilde H$ are independent of the change of parameters $\fbq$. If $q_s=q \in \BC^\times$ for all $s \in \tilde \BS$, then the condition that all the extended parahoric subalgebras are semisimple is equivalent to the condition that $q$ is not a root of the Poincar\'e polynomial of $\tilde H$. The classification of irreducible representations in the case where $q$ is not a root of unity is obtained by Kazhdan and Lusztig \cite{KL} . It is proved later by Xi \cite{Xi} that the Kazhdan-Lusztig classification remains valid if $q$ is not a root of Poincar\'e polynomial. 

Now we provide some examples. 

\begin{example}
Let $G=SL_3$. Then $\Omega=\{1\}$. By Example \ref{sl3}, the rigid cocenter $\overline{\tilde H}^\rig$ has a basis $\{T_1=1, T_{s_1}, T_{s_1 s_2}, T_{s_0 s_1}, T_{s_0 s_2}\}$. There are three maximal (extended) parahoric subalgebras: $H_K$ for $K=\{1, 2\}, \{0, 1\}, \{0, 2\}$. They are finite Hecke algebras associated to the group $S_3$. 

The Poincare polynomial $P_{S_3}(q)=\Phi_3(q) \Phi_2(q)$, where $\Phi_i$ is the $i$-th cyclotomic polynomial. If $q$ is not a root of $P_{S_3}(q)$, then the three parahoric Hecke algebras are semisimple. In particular, the map $Tr: \bar H_K \to R(H_K)^*$ is injective and the naive kernel is trivial. The conjectural rank of $R(\tilde H)_\rig$ equals the dimension of $\bar \CH^\rig$, which is $5$. 

If $\Phi_2(q)=0$, then the kernel of the map $Tr: \bar H_{\{1, 2\}} \to R(H_{\{1, 2\}})^*$ is spanned by $T_{s_1}+1$. Note that the image of $T_{s_1}+1, T_{s_2}+1, T_{s_0}+1$ in $\overline{\tilde H}^\rig$ are the same. Thus the naive kernel is spanned by the image of $T_{s_1}+1$ in $\overline{\tilde H}^\rig$. The conjectural rank of $R(\tilde H)_\rig$ equals the dimension of $\overline{\tilde H}^\rig$ minus $1$, which is $4$. 

If $\Phi_3(q)=0$, then the kernel of the map $Tr: \bar H_{\{1, 2\}} \to R(H_{\{1, 2\}})^*$ is spanned by $(q+1) T_{s_1 s_2}+(q+2) T_{s_1}+1$. The naive kernel is spanned by the image of $(q+1) T_{s_1 s_2}+(q+2) T_{s_1}+1$, $(q+1) T_{s_0 s_2}+(q+2) T_{s_0}+1$, $(q+1) T_{s_0 s_1}+(q+2) T_{s_0}+1$ in $\overline{\tilde H}^\rig$. The conjectural rank of $R(\tilde H)_\rig$ equals the dimension of $\overline{\tilde H}^\rig$ minus $3$, which is $2$. 

Here the characteristic of $\kk$ does not affect the conjectural dimension of $R(\tilde H)_\rig$. The following table lists the conjectural rank of $R(\tilde H)_\rig$ for various choice of $q$ and $char(\kk)$. 

\[\begin{tabular}{|c | c | c|}
\hline
$SL_3$ &  $char(\kk) \neq 3$ & $char(\kk)=3$ \\ \hline
$\Phi_3(q) \Phi_2(q) \neq 0$ &    5       &  5 \\ \hline
$\Phi_2(q)=0$       &    4       & 4 \\ \hline
$\Phi_3(q)=0$       &    2      &  2 \\
\hline
\end{tabular}\]
\end{example}

\begin{example}
Let $G=PGL_3$. Then $\Omega=\mu_3$. By Example \ref{gl3}, the rigid cocenter $\overline{\tilde H}^\rig$ has a basis $\{T_1=1, T_{s_1}, T_{s_1 s_2}, \t, \t^2\}$, where $\t$ is a nontrivial element in $\Omega$. There are four maximal extended parahoric subgroups: $H_K$ for $K=\{1, 2\}, \{0, 1\}, \{0, 2\}$ and $H_{\emptyset}^\sharp$, the group algebra of $\Omega$. The first three are conjugate by $\Omega$ and we only need to consider the contribution of $\ker(\bar H_K^\sharp \to R(H_K^\sharp)^*)$ for $K=\{1, 2\}$ and $K=\emptyset$ to the naive kernel. 

If $\Phi_2(q)=0$ or $\Phi_3(q)=0$, then the kernel of $Tr: \bar H_{\{1, 2\}} \to R(H_{\{1, 2\}})^*$ is $1$-dimensional and this makes an one-dimensional contribution of the naive kernel. 

If $char(\kk)=3$, then the only irreducible representation of $\Omega$ is trivial and the kernel of $Tr: \bar H_{\emptyset}^\sharp \to R(H_{\emptyset}^\sharp)^*$ is $2$-dimensional. This makes a two-dimensional contribution of the naive kernel. 

The following table lists the conjectural rank of $R(\tilde H)_\rig$ for various choice of $q$ and $char(\kk)$. 
\[
\begin{tabular}{|c | c | c|}
\hline
$PGL_3$ &  $char(\kk) \neq 3$ & $char(\kk)=3$ \\ \hline
$\Phi_3(q) \Phi_2(q) \neq 0$ &    5       &  3 \\ \hline
$\Phi_2(q)=0$       &    4       & 2 \\ \hline
$\Phi_3(q)=0$       &     4       & 2 \\
\hline
\end{tabular}
\]

Here the characteristic of $\kk$ does affect the conjectural dimension of $R(\tilde H)_\rig$. It is also interesting to compare this example with the previous example to see how the different isogeny classes affect the conjectural rank of $R(\tilde H)_\rig$. 
\end{example}

\subsubsection{Rigid determinant}

Recall that in \S \ref{finite-ct}, we discussed the character table for finite Hecke algebras. In particular, we may compute the determinant of the character table. It is a polynomial of the generic parameters $\fbq$ and it is determined by the finite Hecke algebra (up to sign depending on the ordering of the irreducible representations). 

We may define the rigid character table for extended affine Hecke algebras (with generic complex parameters) as well. Here we use a family of representations, depending analytically  
on the parameters, whose images in $R(\tilde H)_\rig$ form a basis. The existence of such family is proved in Theorem \ref{family}. We define the rigid determinant as the determinant of the rigid character table. It is again a polynomial of the parameters $\fbq$ and it is determined up to scalar. 

We denote by $\det^\rig_\spadesuit$ the rigid determinant of an extended affine Hecke algebra of type $\spadesuit$ and by $\det_\heartsuit$ for the determinant of the character table of an (extended) finite Hecke algebra of type $\heartsuit$. 

\begin{conjecture}\label{rig-det}
Let $\tilde H$ be an extended affine Hecke algebra with generic parameter $\fbq$. Then the rigid determinant $\det^\rig \in \kk[\fbq]$ is a factor of $\Pi_{K} \det_{H^\sharp_K}$, where $H_K$ runs over all the parahoric subalgebras of $\tilde H$.
\end{conjecture}

\begin{remark}
In fact, in the conjecture we only need to take the product of the maximal extended parahoric subalgebras. And we also expect that there is an explicit formula for the rigid determinant, as an alternating product of the determinants of the character tables for various of (partially) extended parahoric subalgebras, as we will see in the examples below. However, due to its complexity, we do not formulate the (conjectural) explicit formula here. 

Also this conjecture predicts that the pairing between the rigid cocenter and the rigid quotient is a perfect pairing if all the extended parahoric subalgebras are semisimple (i.e., the determinant of the character tables are invertible). This coincides with the prediction from the naive kernel conjecture \ref{naive}. 
\end{remark}

Now we give two examples to support the rigid determinant conjecture. We also provide explicit formulas for the rigid determinants in these examples. 

\begin{example}
Let $G=Sp(4)$ and let $\tilde H$ be the affine Hecke algebra attached to the affine diagram of type $C_2$ with three parameters
\begin{equation*}\label{c2-aff}
\xymatrix{q_0& q_1\ar@{<=}[l]& q_2\ar@{=>}[l]}
\end{equation*}

The rigid cocenter $\overline{\tilde H}^\rig$ has a basis $\{T_1=1, T_{s_0}, T_{s_1}, T_{s_2},  T_{s_0 s_1}, T_{s_0 s_2}, T_{s_1 s_2}, T_{s_0 s_1 s_0 s_1}, T_{s_1 s_2 s_1 s_2}\}$. There are three maximal parahoric subalgebras: $H_K$ for $K=\{1, 2\}, \{0, 1\}, \{0, 2\}$. Here $T_{s_i}$ for $i=0, 1, 2$ each appears in the cocenters of two out of three maximal parahoric subalgebras. We have 
\begin{gather*} \det{}_{C_2(q_0, q_1)}=(1+q_0)^2(1+q_1)^2 (1+q_0 q_1)(q_0+q_1), \\ 
\det{}_{A_1(q_0) \times A_1(q_2)}=(1+q_0)^2 (1+q_2)^2, \\
\det{}_{C_2(q_1, q_2)}=(1+q_1)^2(1+q_2)^2 (1+q_1 q_2)(q_1+q_2), \\
\end{gather*}

The rigid determinant is computed in \cite[Example 7.9]{CH2}. We have $$\det{}^\rig_{\tilde C_2(q_0, q_1, q_2)}=\frac{\det{}_{C_2(q_0, q_1)} \det{}_{A_1(q_0) \times A_1(q_2)} \det{}_{C_2(q_1, q_2)}}{\det{}_{A_1(q_0)} \det{}_{A_1(q_1)} \det{}_{A_1(q_2)}}.$$
\end{example}

\begin{example}
Let $G=PSp(4)$ and let $\tilde H$ be the affine Hecke algebra attached to the affine diagram of type $C_2$ with three parameters
\begin{equation*}
\xymatrix{q_2& q_1\ar@{<=}[l]& q_2\ar@{=>}[l]\ar@{<->}@/_1pc/[ll]}
\end{equation*}

The rigid cocenter $\overline{\tilde H}^\rig$ has a basis $\{T_1=1, T_{s_1}, T_{s_2},  T_{s_0 s_2}, T_{s_1 s_2}, T_{s_1 s_2 s_1 s_2}, \t, T_{s_1} \t, T_{s_0} \t\}$. There are four maximal extended parahoric subalgebras: $H_K^\sharp$ for $K=\{1, 2\}, \{0, 1\}, \{0, 2\}, \{1\}$. The first two are conjugate by $\Omega$. The elements $T_{s_1}, T_{s_2}, \t$ each appears twice in the cocenters of the three isomorphism classes of maximal parahoric subalgebras. We have 
\begin{gather*} \det{}_{C_2(q_1, q_2)}=(1+q_1)^2(1+q_2)^2 (1+q_1 q_2)(q_1+q_2), \\
\det{}_{(A_1(q_2) \times A_1(q_2)) \rtimes \mu_2}=4 (1+q_2)^4, \\
\det{}_{A_1(q_1) \times \mu_2}=4(1+q_1).
\end{gather*}

The rigid determinant is computed in \cite[Example 7.10]{CH2}. We have $$\det{}^\rig_{{\tilde C_2 \rtimes \mu_2}(q_1, q_2)}=\frac{\det{}_{C_2(q_1, q_2)} \det{}_{(A_1(q_2) \times A_1(q_2)) \rtimes \mu_2} \\det{}_{A_1(q_1) \times \mu_2}}{8 \det{}_{A_1(q_1)} \det{}_{A_1(q_2)} \det{}_{\mu_2}}.$$
\end{example}

\subsection{$0$-Hecke algebras}
If all the parameters equal to zero, then we say that the corresponding Hecke algebra a {\it $0$-Hecke algebra}. In this subsection, we discuss the cocenter and representations of affine $0$-Hecke algebras. The purpose is two-folded: 
\begin{itemize}
\item Affine $0$-Hecke algebras serve as models for the pro-$p$ Iwahari-Hecke algebras, which play an important role in the study of mod-$p$ representations of $p$-adic groups. 

\item For affine $0$-Hecke algebras, the cocenter-representation duality fails. The good news is that it fails in a ``controllable'' way. This serves as an evidence of the naive kernel conjecture \ref{naive}. 
\end{itemize}

\subsubsection{The cocenter}
Let $W$ be a finite or (extended) affine Weyl group. As we have seen in \S \ref{finite-ct} and \S \ref{6}, the cocenter of $H_{\fbq}$ (for nonzero parameters $\fbq$) has a standard basis indexed by $W_{\min}/\sim$. For $0$-Hecke algebras, the situation is different. 

By Proposition \ref{3-2} the induction argument still works and the cocenter of $H_0$ is spanned by the image of $T_w$, where $w$ runs over the elements in $W_{\min}$. However, for $w \sim w'$, $T_w$ and $T_{w'}$ may have different image in $\bar H_0$. 

By Proposition \ref{ww'} (1), if $w \approx w'$, then the image of $T_w$ and $T_{w'}$ in $\bar H_0$ are the same. For any $\approx$-equivalence class $\Sigma$ of $W_{\min}$, we denote by $T_{\Sigma}$ the image of $T_w$ in $\bar H_0$ for any $w \in \Sigma$. Then 

\begin{theorem}
The set $\{T_{\Sigma}; \Sigma \in W_{\min}/\approx\}$ forms a $\BZ$-basis of $\bar H_0$. 
\end{theorem}

For finite Hecke algebras, it is proved in \cite[Theorem 5.5]{He-0}, where the basis theorem of $\bar H_{\fbq}$ (for nonzero parameters $\fbq$) is used. A different and simpler approach is found in \cite[\S 2.4]{HNx}, which works for both finite and affine Hecke algebras. 

\subsubsection{Finite $0$-Hecke algebras} In this subsection, we assume that $W$ is a finite Coxeter group. Recall that $$\G=\{(J, C); J \subset \BS, C \text{ is an elliptic conjugacy class of } W_J\}.$$ It is proved in Proposition \ref{1-10} that the set of equivalence classes of $\G$ is in natural bijection with the set of conjugacy classes of $W$. The following result is proved in \cite[Corollary 3.2]{He-0}. 

\begin{proposition}
The map $$\G \to W_{\min}/\approx, \qquad (J, C) \mapsto C_{\min}$$ gives a bijection from $\G$ to the $\approx$-equivalence classes of $W_{\min}$. 
\end{proposition}

The irreducible representations of $H_0$ are easy to construct. 
For any $J \subset \BS$, let $\l_J$ be the one-dimensional representation of $H_0$ defined by 
\[
\l_J(T_s)=\begin{cases} -1, & \text{ if } s \in J; \\ 0, & \text{ if } s \notin J. \end{cases}
\]

By \cite{No}, the set $\{\l_J\}_{J \subset \BS}$ is the set of all the irreducible representations of $H_0$. 

The following result \cite[Proposition 5.6]{He-0} describes how the cocenter-representation duality fails for $H_0$. 

\begin{proposition}
The trace map $\bar H_0 \to R(H_0)$ is surjective and the kernel is spanned by $T_{\Sigma}-T_{\Sigma'}$, where $\Sigma, \Sigma'$ are $\approx$-equivalence classes of $W_{\min}$ of the same support. 
\end{proposition}

Now let us compare the trace function for $H_{\fbq}$ for generic parameters $\fbq$ and $H_0$. Let $\Sigma \in W_{\min}/\approx$ and $(J, C) \in \G$ be the associated pair. Then the trace function of $H_0$ remembers the support $J$ of the element $T_\Sigma$, but ignores the difference between various elliptic conjugacy classes of $W_J$. On the other hand, the trace function of $H_{\fbq}$ remembers the elliptic conjugacy classes, but ignore the difference between the various subsets of $\BS$ that are conjugate to each other. 

\subsubsection{Affine $0$-Hecke algebras}

Recall that $\tilde \G$ is the set of standard quadruples defined in \S\ref{st-qua} for ordinary conjugation action. It is proved in Theorem \ref{1-19} that the set of equivalence classes of $\tilde \G$ is in natural bijection with the set of conjugacy classes of $\tW$. 



For any $\Sigma \in \tW_{\min}/\approx$ corresponding to a standard quadruple $(J, x, K, C)$, we set $J_\Sigma=J$. We set \begin{gather*} \overline{\tilde H_0}^{\rig}=\oplus_{\Sigma \in \tW_{\min}/\approx, J_\Sigma=\Pi} \BZ T_\Sigma, \qquad \overline{\tilde H_0}^{\nrig}=\oplus_{\Sigma \in \tW_{\min}/\approx, J_\Sigma \subsetneqq \Pi} \BZ T_\Sigma \end{gather*}

Let $K \subset \tilde \BS$ with $W_K$ finite. Recall that $\Omega(K)=\{\t \in \Omega; \t(K)=K\}$. Let $\chi \in \Hom_\BZ(\Omega(K), \kk^\times)$. We extend $\chi$ to a $1$-dimensional $H_{a, 0} \rtimes \Omega(K)$-module, where $T^J_s$ acts by $-1$ if $s \in K$ and by $0$ if $s \in \tilde \BS \smallsetminus K$. Then we set $$\pi_{K, \chi}=\tilde \CH_0 \otimes_{H_{a, 0} \rtimes \Omega(K)} \chi.$$ We say that $(K, \chi) \sim (K', \chi')$ if they are conjugate by an element in $\Omega$. It is easy to see that if $(K, \chi) \sim (K', \chi')$, then $\pi_{K, \chi}$ is isomorphic to $\pi_{K', \chi'}$. 

We set $$R(\tilde \CH_0)_\rig=\oplus_{(K, \chi)/\sim; K \subset \tilde \BS \text{ with } W_K \text{ finite }} \BZ \pi_{K, \chi}.$$ The following result is proved in \cite[Proposition 5.1]{HNx}. 

\begin{theorem}
For $M \in R(\tilde \CH_0)$, $Tr(\overline{\tilde H_0}^{\nrig}, M)=0$ if and only if $M \in R(\tilde \CH_0)_\rig$. 
\end{theorem}

In other words, the trace map $Tr: \overline{\tilde H_0} \to R(\tilde \CH_0)$ induces a map $$Tr: \overline{\tilde H_0}^\rig \to R(\tilde \CH_0)_\rig^*.$$ This map is surjective, but not injective. Similar to the finite $0$-Hecke algebra case, the trace map on the rigid cocenter remembers the support $K$ of the element $T_\Sigma$, but ignore the difference between various elliptic conjugacy classes of $W_K$. Thus the above Theorem provides an evidence to the naive kernel conjecture \ref{naive}. 

\smallskip

In \cite{Vig05}, Vign\'eras introduces the Bernstein-Lusztig basis $\{E_w; w \in \tW\}$ of $\tilde \CH_0$. She also gives the definition of the supersingular modules of $\tilde H_0$. By definition, a module $M$ of $\tilde H_0$ is {\it supersingular} if $E_w M=0$ for $w \in \tW$ with $\ell(w) \gg 0$. 

We have the following characterization of supersingular modules. 

\begin{proposition}
Let $M \in R(\tilde \CH_0)$. The following conditions are equivalent:

(1) $M$ is supersingular.

(2) $M \in \oplus_{(K, \chi)/\sim; K \subset \tilde \BS \text{ with } W_K, W_{\tilde \BS-K} \text{ finite}} \BZ \pi_{K, \chi}$.

(3) $Tr(\overline{\tilde H_0}^{nrig}+\iota(\overline{\tilde H_0}^{nrig}), M)=0$, where $\iota$ is the involution on $\tilde \CH_0$ defined in \cite[Corollary 2]{Vig05}. 
\end{proposition}

Here the equivalence between (1) and (2) is first obtained by Vign\'eras in \cite{V14-3}. The criterion (3) for the supersingular modules is given in \cite[Proposition 5.4]{HNx} and a new proof of the equivalence between (1) and (2) is also given in loc.cit.




\end{document}